\documentclass{amsart}
\usepackage{amsmath,amssymb,amsthm,amscd}
\usepackage[a4paper,text={140mm,240mm},centering,headsep=5mm,footskip=10mm]{geometry}
\usepackage[matrix,arrow, curve]{xy}
\numberwithin{equation}{section}
\usepackage{color}
\usepackage{graphicx}
\usepackage[hypertexnames=false,linkcolor=black, colorlinks=true, citecolor=black, filecolor=black,urlcolor=black]{hyperref}

\newtheorem{Thm}{Theorem}[section]
\newtheorem{Cor}[Thm]{Corollary}
\newtheorem{Lem}[Thm]{Lemma}
\newtheorem{Prop}[Thm]{Proposition}

\newtheorem{MainThm}[Thm]{Main Theorem}

\theoremstyle{definition}
\newtheorem{Def}[Thm]{Definition}
\newtheorem{Setup}[Thm]{Setup}
\newtheorem{Rk}[Thm]{Remark}
\newtheorem{Ex}[Thm]{Example}
\newtheorem{Not}[Thm]{Notation}

\newtheorem{Obs}[Thm]{Observation}
\newtheorem{Algorithm}[Thm]{Algorithm}

\numberwithin{equation}{section}  

\newcommand{\tnu}{\widetilde{\nu}}
\newcommand{\fp}{\mathfrak{p}}
\newcommand{\hR}{\widehat{R}}
\newcommand{\hy}{\widehat{y}}
\newcommand{\hf}{\widehat{f}}

\newcommand{\maxIdeal}{M}
\newcommand{\Resfield}{{k}}
\newcommand{\resfield}{{k(x)}}
\newcommand{\resfielt}{{k(x')}}
\newcommand{\RSP}{r.s.p.}
\newcommand{\IN}{\mathbb{N}}	
\newcommand{\N}{\mathbb{N}}	
\newcommand{\F}{\mathbb{F}}
\newcommand{\INinfty}{\mathbb{N}_\infty}	
\newcommand{\IQinfty}{\mathbb{Q}_\infty}	
\newcommand{\INN}{\IN^\IN}	
\newcommand{\A}{\mathbb{A}}
\newcommand{\IQ}{\mathbb{Q}}
\newcommand{\IP}{\mathbb{P}}	
\newcommand{\IE}{\mathbb{E}}	
\newcommand{\IR}{\mathbb{R}}	
\newcommand{\R}{\mathbb{R}}	
\newcommand{\IZ}{\mathbb{Z}}	
\newcommand{\cO}{\mathcal{O}}
\newcommand{\cB}{\mathcal{B}}
\newcommand{\cI}{\mathcal{I}}
\newcommand{\cX}{\mathcal{X}}

\newcommand{\cZ}{\mathcal{Z}}
\newcommand{\vp}{\varphi}

\newcommand{\io}{\iota}			
\newcommand{\ioo}{\iota_0}		
\newcommand{\ioc}{\iota_{hs}}	
\newcommand{\iop}{\iota_{poly}}	
\newcommand{\slope}{s}
\newcommand{\ideal}{I}
\newcommand{\nus}{\nu^\ast} 

\newcommand{\car}{\mathrm{char}}
\newcommand{\ord}{\mathrm{ord}}
\newcommand{\Dir}{\mathrm{Dir}}
\newcommand{\Rid}{\mathrm{Rid}}
\newcommand{\Spec}{\mathrm{Spec}}
\newcommand{\wrt}{with respect to }
\newcommand{\WLOG}{without loss of generality }

\newcommand{\poly}[3]{\Delta(  #1  ; #2 ; #3  ) }
\newcommand{\cpoly}[2]{\Delta(  #1  ; #2  ) }

\newcommand{\rot}[1]{{\textcolor{black}{#1}}}
\newcommand{\blau}[1]{\textcolor{black}{#1}}

\newcommand{\cyan}[1]{\textcolor{black}{#1}}

\begin{document}

\title[\resizebox{5in}{!}{A strictly decreasing invariant for resolution of singularities in dimension two}]
{
	A strictly decreasing invariant for resolution of singularities in dimension two
}

\author{Vincent Cossart}
\address{V. Cossart: Professeur \'em\'erite at Laboratoire de Math\'emathiques LMV UMR 8100, Universit\'e de Versailles, 45 avenue des \'Etats-Unis,78035 VERSAILLES Cedex, France.}
\email{cossart@math.uvsq.fr}

\author{Bernd Schober}
\thanks{The second author is supported by a Research Fellowship of the Deutsche Forschungsgemeinschaft}
\address{B. Schober: 
Institut f\"ur Algebraische Geometrie,
Leibniz Universit\"at Hannover,
Welfengarten 1,
30167 Hannover,
Germany.}
\email{schober@math.uni-hannover.de}



\maketitle

\begin{abstract}
	We construct a local invariant for resolution of singularities of \rot{two}-dimensional excellent \blau{Noetherian} schemes with boundary.
	We prove that the invariant strictly decreases at every step of the algorithm 
	of Cossart, Jannsen and Saito.
\end{abstract}

\tableofcontents      

\section*{Introduction}
\label{Intro}
Jannsen, Saito and the first author \cite{CJS} have proved that every Noetherian excellent scheme of dimension at most two with boundary admits a canonical resolution of singularities via a finite sequence of permissible blow-ups.
\rot{For a precise formulation of the statement and the meaning of canonical, we refer to \cite{CJS} Theorems 0.1, 0.2, 0.3.
	In particular, the algorithm commutes with any Zariski or \'etale localization and also commutes with completion.}
This result applies to the embedded as well as the non-embedded case.
(For the notion of a boundary in the non-embedded case we refer to Remark \ref{Rk:nonembedded} and the preceding paragraph).
In the algorithm of \cite{CJS}, the centers of the successive blow-ups are determined by the geometry of the maximal singular locus plus the history of the preceding resolution process.
\blau{(See Algorithm \ref{Algorithm_CJS} and Remark \ref{Rk:global}).}
This method is based on Hironaka's work for resolution of hypersurface singularities in \cite{HiroBowdoin}.
The work of \cite{CJS} is different from other approaches to resolution of singularities in the sense that 
\rot{the center for blowing up is not given as the maximal locus of a sophisticated invariant.}
Finiteness of the sequence of blow-ups constructed is proved indirectly by deducing a contradiction from the hypothesis that the sequence be infinite.

Earlier results on resolution of singularities in positive characteristic are due to Abhyankar \cite{AbhyDim2}, Lipman \cite{Lipmandim2}, and Hironaka \cite{HiroBowdoin}.
For a comparison of the different approaches, see \cite{CGO}.
In \cite{DaleDim2} Cutkosky gives a simplified version of Abhyankar's proof.
The result due to Lipman is valid for \rot{two}-dimensional excellent schemes, but in contrast to \cite{CJS}, involves not only blowing up regular centers but also normalizations.
It is therefore not clear whether the proof extends to the embedded situation.

More recently, Benito and Villamayor \cite{AngOrlDim2} gave another proof for resolution of singularities in dimension two using the techniques of Rees algebras and generic projections. 
Kawanoue and Matsuki applied and simplified these methods in \cite{IFPdim2} for their idealistic filtration program.
Note that \cite{AngOrlDim2} and \cite{IFPdim2} apply to schemes of finite type over a perfect field.

In characteristic zero, the centers for the resolution are determined by a certain invariant, which strictly decreases after blowing up.
Due to the lack of maximal contact, the proof does not apply in positive characteristic.
Moh \cite{Moh} even shows that a natural candidate for an invariant in positive characteristic may increase after certain blow-ups with regular centers,
although he also proves that there is a bound on the increase.
{In \cite{Moh1}, he defines a so-called adjustor in order to overcome the detected increase in dimension three for special hypersurfaces\rot{. See} also \cite{HerWagn} for a variant in dimension two.}

Already in his thesis \cite{VincentThesis}, the first author introduces a new invariant $ \omega $ for resolution of singularities that is not affected by the increase phenomena.
Moreover, $ \omega $ plays a key role in Piltant and his proof of the existence of a resolution of singularities in dimension three\rot{. See} \cite{CPmixed}, \cite{CPmixed2}. 

For varieties of dimension two over perfect fields, the proof of Kawanoue and Matsuki \cite{IFPdim2} provides a strictly decreasing invariant.
In this article, we construct a strictly decreasing invariant for the resolution of singularities 
\rot{for the general excellent \blau{Noetherian} scheme of dimension two.}

\smallskip

Let $ X $ be a reduced excellent \blau{Noetherian} scheme of dimension at most two embedded in a regular scheme $ Z $.
Let 
$ \cB $ be a boundary on $ Z $ (Definition \ref{Def:boundary}, Remark \ref{Rk:nonembedded}), which is distinguished into old and new components via a so-called history function $ O_X $ (Definition \ref{Def:history_function}).
Based on a precise study of the algorithm of \cite{CJS}, we develop a local invariant $ \io = \io (X, Z, (\cB, O_X), x ) $, for $ x \in X $.
Since our invariant is local, we reduce the non-embedded to the embedded case $ X \subset Z $ by the Cohen structure theorem.
The invariant consists of three parts, $ \io  = ( \, \ioo, \, \ioc, \, \iop \, ) $:
\begin{itemize}
	\item 	$ \ioo \in \INN \times \IN^3 $ is a rough measure for the complexity of the singularity of $ X $ at $ x$,
	%
	%
	\item	$ \ioc \in \INN \times \IN^3 \times \IQinfty^2 $ ($ \IQinfty := \IQ_{\geq 0} \cup \{ \infty \} $) measures the singularity of certain distinguished components of the maximal singular locus of $ X $ at $ x $. 
	\item	$ \iop \in \IQinfty^4 $ comes from Hironaka's characteristic polyhedron and takes care of the singularities which arise during the resolution process.
\end{itemize}

We prove that those entries of $ \iota $ which seem to \rot{depend} on the embedding are actually independent of it (Theorems \ref{Thm:alphabetainvariant_1_deltaonly} and \ref{Thm:alphabetainvariant_part2_nodelta}).
Hence we can omit $ Z $ and write $ \iota = \iota(X, (\mathcal{B}, O_X) , x) $.
For simplicity, we abbreviate  $ \iota(X, \mathcal{B}, x) := \iota(X, (\mathcal{B}, O_X) , x) $ without explicitly mentioning the history function. 

The last six entries in $ \IQ_\infty^2 \times \IQ_\infty^4 = \IQ_\infty^6 $ are in fact contained in  $ \left( \frac{1}{N} \cdot \IN_\infty \right)^6 $ ($ \IN_\infty :=  \IN \cup \{ \infty \}$), 
where $N$ is a natural number depending on the first entries (see Corollary~\ref{Cor:delta_in_1/N_Z^e}, the end of Definition~\ref{Def:ioc}, and inequality \eqref{blabla}): our invariant $ \io ( X,\cB, x)$
cannot strictly drop infinitely many times.

The assumption that $ X $ has dimension at most two is crucial here.
Because of this assumption the ``new" singularities which appear during the blow-up process have a very nice form.
More precisely, the new irreducible components of the maximal singular locus are already regular and intersect the exceptional divisors of the preceding blow-ups (or more generally, the boundary components) transversally.
Therefore $ \iop $ is an effective tool to measure singularities in dimension two.

Our main result is

\begin{MainThm}
\label{MainThm}
Let $ X \subset Z $, $ \cB$ a boundary on $ Z $, and $ \pi_Z : Z' \to Z $ be the blow-up {with center $ D \subset X $} following the algorithm of \cite{CJS}.
Let {$ x \in D $} and $ x' \in \pi_Z^{ -1 } (x ) $.
Then we have
$$
	\io ( X', \cB', x' ) < \io ( X, \cB, x ).
$$
\end{MainThm}

Here we use the lexicographical order on $ \INN \times \IN^3 \times \INN \times \IN^3 \times \IQ_\infty^2 \times \IQ_\infty^4 $, where we equip $ \INN $ with the product order 
(this is a partial order).
The entries of the invariant with values in $ \INN $ are certain Hilbert-Samuel functions.
By the upper semi-continuity property, the maximal Hilbert-Samuel locus 
$ X_{\mathrm{max}}:=\bigcup_{\nu \in \Sigma_X^{\mathrm{max}}} X(\nu) $ 
consists of finitely many mutually disjoint components
$X(\nu)$, for $\nu\in \Sigma_X^{\mathrm{max}}$ (where $ \Sigma_X^{\mathrm{max}} $ denotes the set of maximal values of the Hilbert-Samuel function on $ X $). 
We use Corollaries~{5.18} and {5.19} of \cite{CJS} which state that, to get (canonical, functorial) resolution of singularities, it is enough \rot{to} decrease the Hilbert-Samuel function over every component $ X (\nu) $ of $ X_{\mathrm{max}} $
(``$\nu$-elimination": see \cite{CJS} Definition 5.14).
The latter implies that we can restrict our attention to a single maximal value. 
{By Theorem \ref{MainThm}, the invariant $ \io $ (which is not necessarily upper semi-continuous, Example \ref{Ex:e_not_usc}) measures the improvement for the singularities lying above the center of the blow-up in every step of \cite{CJS}.}
\rot{Using Abhyankar's notion of good and bad points, Theorem \ref{MainThm} leads to a direct proof for the finiteness of the sequence of blow-ups constructed by \cite{CJS} (section 6).}

\smallskip

In the first three sections we recall and study the important notions for \cite{CJS}:
	(1) the log-Hilbert-Samuel function and its behavior under permissible blowing ups, the construction of $ \ioo $;
	(2) the \cite{CJS} algorithm and why dimension two is a particular good situation;
	(3) Hironaka's characteristic polyhedron and its invariants in the two dimensional case.
In \S 4 we construct $ \ioc $  and prove that it strictly improves until there are only new components in the locus of maximal singularity.
\rot{In} \S 5 we discuss in detail how to obtain $ \iop $ and prove its improvement in the remaining \cite{CJS}-algorithm until the log-Hilbert-Samuel functions drops strictly at every point. 
\rot{Finally, we provide a short proof for the termination of the \cite{CJS}-algorithm in \S 6.}

%
%
%
%
%
%
%
%
%
%
%
%
%
%
%
%
%



\section{First measure for the complexity of the singularity}

We first recall some definitions and results of \cite{CJS}.
Using them we deduce the first rough measure $ \ioo \in \INN \times \IN^3 $ for the singularity.

\smallskip

Let $ X $ be an excellent Noetherian scheme.
Since we want to achieve resolution of singularities only by using blow-ups centered in $ \mathrm{Sing}(X) $, we may suppose that $ X $ is connected.


The first entry for $ \ioo $ is the Hilbert-Samuel function $ H_X : X \to \INN $ of $ X $.
In order to be able to compare the values in $ \INN $ we equip it with the product order, i.e., for $ \nu, \mu \in \INN $ we have $ \nu \leq \mu $ if and only if $ \nu ( n ) \leq \mu ( n ) $ for every $ n \in \IN $. 

Before we recall the precise definition of $ H_X $ let us mention the following:
Since $ X $ is excellent it is in particular catenary. 
Thus for any two irreducible closed subschemes $ Y \subset Z $ of $ X $ all maximal chains of irreducible closed subschemes $ Y = Y_0 \subset Y_1 \subset \ldots \subset Y_r = Z $ have the same length $ r $ denoted by $ \mathrm{codim}_Z ( Y ) $.
For three of them, $ Y \subset Z \subset W $, one has $ \mathrm{codim}_W ( Y ) = \mathrm{codim}_W ( Z ) + \mathrm{codim}_Z ( Y ) $.

\begin{Def}[\cite{CJS} Definition 1.28]
\label{Def:HS_function}
Let $ X $ be an excellent Noetherian scheme and fix an integer $ N \in \IN $ with $ N \geq \dim ( X ) $.
\begin{itemize}
	\item[(1)]	Let $ x \in X $ and denote by
				$ ( \cO = \cO_{X,x}, \maxIdeal, \resfield) $ the local ring of $ X $ at $ x $.
				The {\em Hilbert-functions $ H_\cO^{ (t) } $ of $ \cO $}, for $ t \in \IN $, are then defined as the elements of $ \INN $ given by
				$$ 
					H_\cO^{(0)} ( n ) = \dim_\resfield ( \maxIdeal^n / \maxIdeal^{ n + 1 } ) \,,
					\hspace{10pt}
					n \in \IN,
				$$
				and, for $ t \geq 1 $, via the formula
				$$	
					H_\cO^{ (t) } ( n ) = \sum_{i=0}^n H_\cO^{ ( t - 1 ) } (i) \,,
					\hspace{10pt}
					n \in \IN.
				$$
		
	
	\item[(2)]	Define $ \phi_X : X \to \IN $ by 
				$$ 
					\phi_X ( x ) := N  - \psi_X ( x ) ,
				$$ 
				$$				
					\psi_X ( x ) := \min \{  \mathrm{codim}_Z ( x ) \mid Z \in \{\mbox{irreducible components of $ X $ containing $ x $}\} \}
				$$
		
	
	\item[(3)]	The {\em Hilbert-Samuel function $ H_X : X \to \IN^\IN  $ of $ X $} is defined by
				$$
					H_X ( x ) = H_{ \cO_{X,x} }^{( \phi_X (x) ) } \in \INN
				$$ 
		
	
	\item[(4)]	The level sets of the Hilbert-Samuel function for an element $ \nu \in \INN $ are defined as
				$$
					X ( \geq \nu ) := \{ x \in X \mid H_X ( x ) \geq \nu \} 
						\hspace{10pt}
						\mbox{ and }
						\hspace{10pt}
					X (  \nu ) := \{ x \in X \mid H_X ( x ) = \nu \}  .
				$$
				Moreover, we put $ \Sigma_X := \{ H_X ( x ) \mid x \in X \} $, and denote by $ \Sigma^{max}_X $ the set of maximal elements in $ \Sigma_X $.
				Then $ X_{max} = \bigcup_{ \nu \in \Sigma^{max}_X } X ( \nu ) $ is called the {\em Hilbert-Samuel locus of $ X $}.
				We also call it the {\em maximal Hilbert-Samuel locus of $ X $}.
\end{itemize} 
\end{Def}

In fact, in the original definition the scheme $ X $ is only required to be locally Noetherian and catenary. 
{For explanations on the role of the integer $ N$, we refer the reader to \cite{CJS} Remark 1.29(a).}

Note that since we assume $ X $ to be Noetherian $ \Sigma_X $ is {\em finite}, by Theorem~\ref{Thm:specializationHS} below.

\smallskip

The Hilbert-Samuel function of $ X $ measures how far $ X $ is apart from being regular. 
In order to recall this statement we further need to introduce the following functions (\cite{CJS} Definition 1.13):

For $ t \in \IN $, define the functions $ \Phi^{ ( t ) } \in \INN $ via, 
{$ \Phi^{ ( 0 ) } (0) = 1 $ and $ \Phi^{ ( 0 ) } ( n ) = 0 $},
for $ n > 0 $, and similarly to the above, 
$  \Phi^{ (t) } ( n ) := \sum_{i=0}^n  \Phi^{ ( t - 1 ) } (i) $, for $ n \in \IN $ and $ t \geq 1 $.


\begin{Lem}[\cite{CJS} Lemma 1.31, Remark 1.32]
	Let $ X $ and $ N $ be as in Definition \ref{Def:HS_function}.
	For $ x \in X $, we have $ H_X ( x ) \geq \Phi^{ ( N ) } $, and $ H_X ( x ) = \Phi^{ ( N ) } $ if and only if $ x $ is a regular point.
	In particular, $ X $ is regular if and only if $ X ( \nu ) = \emptyset $, for all $ \nu \in \INN $ except for $ \nu = \Phi^{ ( N ) } $.
\end{Lem}

The following crucial result holds.

\begin{Thm}[\cite{CJS} Theorem 1.33]
\label{Thm:specializationHS}
	Let $ X $ be an excellent Noetherian scheme.
	\begin{itemize}
		\item[(1)]
				If $ x \in X $ is a specialization of $ y \in X $, i.e., $ x \in \overline{ \{ y \} } $, then $ H_X ( x ) \geq H_\rot{X} ( y ) $.
			
	
		\item[(2)]
				For any $ y \in X $, there exists a dense open subset $ U \subset  \overline{ \{ y \} } $ such that $ H_X ( x ) = H_X ( y ) $ for every $ x \in U $.
			
	
		\item[(3)]
				The function $ H_X : X \to \INN $ is upper semi-continuous, i.e., for every $ \nu \in \INN $, $ X ( \geq \nu ) $ is closed in $ X $.
	\end{itemize}
\end{Thm}

\begin{Rk}
\label{Rk:LevelSetsProp}
Part (3) of the theorem implies that the level sets $X(\nu)$ are closed in $ X $, for $ \nu \in \Sigma^{\max}_X $.
Further, we obtain $X(\nu) \cap X(\mu) = \emptyset $, for $ \nu, \mu \in \Sigma^{\max}_X $,  $ \nu \neq \mu $, by (1).
\end{Rk}


Let $ x \in X $.
By passing to the completion of $ \cO_{X,x} $ and by the Cohen structure theorem we may assume that locally at $ x $ we are in an embedded situation, i.e., the singularity is given by an ideal $ J \subset R $ in a complete regular local ring $ ( R ,\maxIdeal, \resfield = R / \maxIdeal ) $ (which is necessarily excellent).

The next natural candidate for the study of singularities is the tangent cone $ C_{x}(X) $ of $ X $ at $ x $. 
This is the cone living in the graded ring $ gr_\maxIdeal ( R ) = \bigoplus\limits_{t \geq 0} M^t / M^{t+1} $ and is defined by the ideal 
$$
	In_\maxIdeal ( J ) := \langle \, in_\maxIdeal ( f )  \mid f \in J \, \rangle,
$$
where $ in_\maxIdeal ( f ) = f \mod \maxIdeal^{ \nu + 1 } \in gr_\maxIdeal (R) $ denotes the {\em initial form of $ f $ \wrt $ \maxIdeal $} and
$ \nu := \nu ( f ) := \ord_\maxIdeal ( f ) = \max \{ m \in \IN \mid f \in \maxIdeal^ m \} $, for $ f \neq 0 $, and $ in_\maxIdeal ( 0 ) = 0 $.

Going back to Hironaka we have the notion of the directrix of a cone.

\begin{Obs} \label{Obs:Directrix} 
The directrix is an extremely natural notion: 
one is interested in a minimal set $ ( Y_1,\ldots,Y_r ) $ of variables needed to write generators of $ In_M(J) \subset gr_M(R)$. 
In fact, $ \langle Y_1,\ldots,Y_r \rangle $ is the ideal of the biggest subvector space $ W $ of $ \Spec ( gr_M(R)) \cong \A_{k(x)}^{\mathrm{dim}(R)} $, which leaves the tangent cone $ C_x (X)\subset \A_{k(x)}^{\mathrm{dim}(R)} $ stable under translation, i.e., for which $ C_x (X) + W = C_x (X) $ holds.
This vector space is {\em the directrix $ \Dir_x ( X ) $ of the tangent cone at $ x $}.
Sometimes we also call it the directrix of $ J $ at $ M  $ and write $ \Dir_\maxIdeal ( J ) $. 

There is a more intrinsic definition: 
$C_{x}(X)$ is naturally embedded in the Zariski tangent space $T_x(X) := \Spec (k(x)[\frac{M_x}{M_x^2}]) $, where $M_x$ is the maximal ideal of $\mathcal{ O}_{X,x}$.
We define the directrix as the biggest subvector space of $ T_x ( X) $ leaving $C_{x}(X)$ stable under translation.  
As $T_x(X)$ is a subspace of $ \Spec ( gr_M(R) )$, both definitions coincide.  
See \cite{CJS} Definition 1.26.
\end{Obs}

\begin{Def}
\label{Def:dir}
	Let $ J \subset R $ be the ideal which defines $ X $ locally at $ x $.
	\begin{itemize}
	\item[(1)]
	A system of regular elements $ ( y ) = ( y_1, \ldots, y_r ) $ in $ R $ is said to determine the directrix $ \Dir_x ( X ) $ of $ X $ at $ x $ if the generators of $ In_\maxIdeal ( J ) \subset gr_\maxIdeal ( R ) $ are contained in $ \resfield [ Y ] $, $ Y_j = y_j \mod \maxIdeal^2 $, for $ 1 \leq j \leq r $,
	$$
		( \, In_\maxIdeal ( J ) \cap \resfield [ Y ] \,) \, gr_\maxIdeal ( R ) = In_\maxIdeal ( J ),
	$$
	and, moreover, $ ( y ) $ is required to be minimal with this property, i.e., the number of elements $ r $ has to be minimal.
		
	
	\item[(2)]	
	We denote by $ e_X ( x ) := \dim_\resfield ( \Dir_x ( X ) ) \leq \dim ( X ) $ the dimension of the directrix.
	Since we are in the embedded situation, we have $ e_X ( x ) = n - r $, where $ n $ is the dimension of $ R $.
		
	
	\item[(3)]
	Let $ K/ \resfield $ be a field extension.
	Then $ e_X ( x )_{K} $ denotes the dimension of the directrix associated to the homogeneous ideal $ In_\maxIdeal ( J )_{K} $.
	\end{itemize}
\end{Def}

If $ X \subset Z $ for some regular Noetherian scheme $ Z $, then an alternative invariant to the Hilbert-Samuel-function $ H_X $ is Hironaka's $ \nu^\ast $-invariant.
Locally one can define it as follows:

\begin{Def}
\label{Def:standardbasis}
	Let $ X \subset Z $ and $ x \in X ${.} 
	Let $ J \subset R $ be the ideal in $ R = \cO_{Z,x} $ determining the local situation at $ x $.
	\begin{itemize}
		\item[(1)]	A system of element{s} $ ( f ) = ( f_1, \ldots, f_m ) $ in $ J $ is called a {\em standard basis for $ J $} if {the} system of their initial forms $ ( in_\maxIdeal ( f) ) = ( in_\maxIdeal ( f_1 ) , \ldots, in_\maxIdeal(f_m) ) $ is a standard basis for $ In_\maxIdeal ( J )  \subset gr_\maxIdeal ( R ) $, i.e.,
		
			\begin{itemize}
			\item[(a)]	{$ In_\maxIdeal ( J )  = \langle in_\maxIdeal ( f_1 ) , \ldots, in_\maxIdeal(f_m) \rangle $},
			\item[(b)]	$ in_\maxIdeal ( f_i ) \notin \langle in_\maxIdeal ( f_1 ) , \ldots, in_\maxIdeal(f_{i - 1} ) $, for $ 2 \leq i \leq m $, and
			\item[(c)]	$ \nu_1 \leq \nu_2 \leq \ldots {\leq} \nu_r $, where $ \nu_i = \ord_\maxIdeal ( f_i ) $.
		\end{itemize}	
	
	\item[(2)] Let $ ( f ) $ be a standard basis of $ J $.
			 We define the $ \nu^\ast $-invariant by
			 $$
			 	\nu^\ast_x ( X,Z ) := \nu^\ast ( J, R ) := ( \nu_1, \nu_2, \ldots, \nu_r , \infty, \infty, \ldots ) \in \INinfty^\IN, 
			 	\hspace{15pt} (\INinfty := \IN \cup \{ \infty \}).
			 $$
	\end{itemize}
\end{Def}

In fact, a standard basis of $ J $ also generates the ideal $ J $, see \cite{HiroCharPoly} Corollary (2.21.d).

The definition of $ \nu^\ast_x ( X,Z ) $ as it is stated here looks as if it depends on the choice of a standard basis.
By \cite{CJS} Lemma 1.2 this is not the case and $ \nu^\ast_x ( X,Z )$ is really an invariant of the singularity.

\smallskip

A candidate for an invariant for the strategy in \cite{CJS} should behave well under the blow-ups which are performed during their process.
These are so called permissible blow-ups.

\begin{Def}[\cite{CJS} Definition 2.1]
	Let $ X $ be an excellent Noetherian scheme and $ D \subset X $ a closed reduced scheme.
	Denote by $ I_D \subset \cO_X $ the ideal sheaf of $ D $ in $ X $, $ \cO_D = \cO_X / I_D $ and $ gr_{ I_D } ( \cO_X ) = \bigoplus_{ t \geq 0} I_D^t / I_D^ {t + 1 } $.
	
	\begin{itemize}
		\item[(0)]	The normal cone of $ D \subset X $ is defined as $ C_{D}(X) := \Spec (gr_{ I_D } ( \cO_X )) $.
						
		
		\item[(1)]	
					$ X $ is {\em normally flat along $ D $ at $ x \in D $ } if $ gr_{I_D} (  \cO_X )_x $ is a flat $ {\cO}_{D,x} $-module.
					$ X $ is normally flat along $ D $ if $ X $ is normally flat along all points of $ D $.
						
	
		\item[(2)]
					$ D \subset X $ is {\em permissible at $ x \in D $} if $ D $ is regular at $ x $ and $ X $ is normally flat along $ D $ at $ x $, and if $ D$ contains no irreducible component of $ X $ containing $ x $.
					$ D \subset X $ is permissible if $ D $ is permissible at all points of $ D $.
						
	
		\item[(3)]	The blow-up $ \pi_D : Bl_D ( X ) \to X $ with a permissible center $ D \subset X $ is called a {\em permissible blow-up}.
	\end{itemize}	 
\end{Def}

In order to give the reader a feeling for the notion of normal flatness, we recall the following two results.

\begin{Thm}[\cite{CJS} Theorem 2.2(2)]
\label{Thm:2.2.(2)_CJS}
	 Let $ X $ and $ D $ {be} as in the previous definition and
	 suppose $ X $ is a closed subscheme of a regular Noetherian scheme $ Z $.
	 Further, $ D $ is assumed to be irreducible.
	 Let $ x \in D $.
	 Denote by $ ( R = \cO_{Z,x}, \maxIdeal, \resfield ) $ the local ring of $ Z $ at $ x $, and let $ J \subset R $ (resp.~$ \fp \subset R $) be the ideal which defines $ X \subset Z $ (resp.~$ D \subset Z $) locally at $ x $.
	 Then the following are equivalent
	 \begin{itemize}
	 	\item[(1)] 	$ X $ is normally flat along $ D $ at $ x $. 
	 		
	
	 	\item[(2)]	$ T_x ( D ) \subset \Dir_x ( X ) $ and the natural map $ C_x(X) \to C_{D}(X) \times_D x $ induces an isomorphism $ C_x (X)/T_x(D) \to C_D(X) \times x $, where $ T_x (D) $ acts on $ C_x(X) $ by the addition in $ T_x (X) $.

	 	
	 	\item[(3)] Let $ u : gr_\fp ( R ) \otimes_{R/\fp} \resfield \to gr_\maxIdeal ( R ) $ be the natural map. 
	 	Then $ In_\maxIdeal ( J ) $ is generated in $ gr_\maxIdeal ( R ) $ by $ u ( In_\fp ( J ) ) $.
	 		
	
	 	\item[(4)]	There exists a standard basis $ ( f )= ( f_1, \ldots, f_m ) $ of $ J $ such that $ \ord_\maxIdeal ( f_i ) = \ord_\fp ( f_i ) $, for all $ i \in \{ 1, \ldots, m \} $.
	 \end{itemize}
\end{Thm}

\begin{Thm}[\cite{CJS} Theorem 2.3]
	 Let $ X $ and $ D $ {be} as in the previous definition.
	 Assume $ D $ is regular.
	 Let $ x \in D $ and let $ y \in D $ be the generic point of the irreducible component of $ D $ containing $ x $.
	 Then the following are equivalent
	 \begin{itemize}
	 	\item[(1)] 	$ X $ is normally flat along $ D $ at $ x $. 
	 		
	
	 	\item[(2)]	$ H_{ \cO_{X,x} }^{ (0) } = H_{ \cO_{X,y} }^{ (\mathrm{codim}_Y (x) ) } $, where $ Y = \overline{ \{ y \} } $ is the closure of $ y $ in $ X $ .
	 		
	
	 	\item[(3)]	$ H_X ( x ) = H_X ( y ) $.
	 \end{itemize}
\end{Thm}

 
Bennett, Hironaka and Singh showed the following result on the behavior of the $ H_X $ under the permissible blow-ups:

\begin{Thm}[\cite{CJS} Theorem 2.10]
\label{Thm:BennHiroSingh}
Let $ X $ be an excellent Noetherian scheme and $ D \subset X $ a permissible closed subscheme, and let $ \pi_X : X' = Bl_D ( X ) \to X $ be the blow-up with center $ D $.
Consider $ x \in D $ and $ x' \in \pi_X^{ - 1 }(x) $ and set $ \delta_{x'/x} = \mathrm{trdeg}_{\resfield} (\resfielt) $, where $ \resfield $ (resp.~$ \resfielt $) denotes the residue field of $ x $ (resp.~$ x' $).
Then
\begin{enumerate}
	\item[(1)] $ H_{X'} ( x' ) \leq H_X ( x ) $.
		
	
	\item[(2)] If $ H_{X'} ( x' ) = H_X ( x ) $ holds, then for any field extension $ K/\resfielt $ we have $ e_{X'} ( x' )_K \leq e_X ( x )_K - \delta_{x'/x} $.
\end{enumerate}

Suppose $ X $ is embedded in a regular scheme $ Z $ and let $ \pi_Z : Z' = Bl_D ( Z ) \to Z $ be the blow-up of $ Z $ with center $ D $.
Then 
\begin{enumerate}
	\item[(3)]	$ \nus_{ x'} ( X', Z' ) \leq \nus_{ x } ( X, Z ) $, and
		
	
	\item[(4)]	$ \nus_{ x'} ( X', Z' ) = \nus_{ x } ( X, Z ) $ holds if and only if $ H_{ X' } ( x') = H_X ( x ) $.
\end{enumerate}
\end{Thm}

{In \cite{CJS} Theorem 1.15 it is shown that the set of all Hilbert-Samuel functions is well-founded, which means that every strictly descending sequence has to be finite.
The proof for this relies on a result by Maclagan \cite{Maclagan}.}

{Hence by the previous theorem we only need to study those points where the Hilbert-Samuel function did not drop and, moreover, one can also use the $ \nus $-invariant in order to detect them.}

\smallskip

The dimension of the directrix is not stable under field extensions $ k' / k(x) $.
In general, one only has $ e_X(x)_{k'}\leq e_X(x)$ according to Lemma~1.10 in \cite{CJS}.

\begin{Def}
	Let the situation be as in the previous theorem.
	Consider {$ x \in D $ and} $ x' \in \pi_X^{ - 1 }(x) ${.}
	\begin{itemize}
			\item[(1)] If $  H_{X'} ( x' ) = H_X ( x ) $ then $ x '$ is called {\em near to $  x $}.
				
	
			\item[(2)] If $ x' $ is near to $ x $ and if we have additionally $ e_{X'} ( x' ) + \delta_{x'/x} = e_X ( x )_{k'} = e_X ( x ) $, for $ k' = k(x')$, then $ x' $ is said to be {\em very near to $ x $.}
	\end{itemize}
\end{Def}


In order to achieve an improvement of the singularities one needs a precise knowledge of the locus of near points.
The following result due to Hironaka (and later improved by Mizutani) gives some control on this locus.

\begin{Thm}[\cite{CJS} Theorem 2.14]
\label{Thm:HiroMizutani}
Let $ x' $ be near to $ x $.
Assume that either $ \car ( k (x ) ) = 0 $ or $ \car ( k(x)) \geq \dim (X)/2 + 1 $.
Then $ x ' $ lies in the projective space associated to the vector space determined by the directrix and the tangent space of the center,
$$
	x' \in \IP ( \Dir_x ( X ) / T_x ( D ) ) \subset \pi^{ -1 }_X ( x ).
$$
\end{Thm}

Recall: $ D $ being normally flat along $ D $ at $ x \in D $ implies $ T_x (D) \subset \Dir_x ( X ) $ 
(Theorem \ref{Thm:2.2.(2)_CJS} (1)$ \Rightarrow$(2)).

If $ \dim ( X) = 2 $, then $ \dim(X)/2 + 1 = 2 $ and the assumptions on {the characteristic of the residue field is always fulfilled.}

\smallskip

During the resolution process we have to take care of the arising exceptional divisor.
This is important in order to obtain a canonical process to resolve singularities.

\begin{Ex}
	Consider the singularity given by $h:= z^3 + x^2 y^2 z + x^3y^3 \in k[x,y,z]$, where $ k $ is a field of any characteristic.
	The singular locus is the union of the two lines $L_1 :=V(z,x)$, $L_2:=V(z,y)$ and this is also the highest order locus.
	Localizing at the generic points of these lines, it is clear that you have to blow them up. 
	But, when, resp.~in which order? 
	
	If this is the beginning of the process, then these two lines are indistinguishable.
	Thus to obtain a canonical resolution process via blow-ups in regular centers one has to blow up the origin $ V(x,y,z) $.
	Then, at the point of parameters $(z',x,y'):=(\frac{z}{x},x,\frac{y}{x})$ in the corresponding affine chart, the strict transform has the equation $ h'={z'}^3 + x^2 {y'}^2 z' + x^3{y'}^3$. 
	Apparently, nothing changed. 
	
	But the resolution process provides us more data.
	Namely, the component $ L_3 := V (z', x) $ is contained in the exceptional divisor of the last blow-up which is given by $ V ( x ) $.
	The strict transform of $ L_1 $ is empty at this point and $ L_3 $ is a newly created curve only lying above the center of the last blow-up.
	Therefore we can distinguish the two irreducible curves $ L_2' = V (z', y' ) $ and $ L_3 $ of the singular locus of $ h' $ since one is ``older" than the other.  
	
	The philosophy of \cite{CJS} is to consider always the oldest irreducible components. 
	Thus following their algorithm the next center would be $ L_2' $.

\end{Ex}

Another reason to take the exceptional divisors into account is that we only want to blow up in centers having at most normal crossings (short n.c.) with the exceptional divisors.
Also it might be important for applications to start not only with $ X $ but also with a n.c.~divisor on $ Z $.
In \cite{CJS}, this is done via the notion of boundaries.

In order to simplify the presentation we suppose in the following that $ X \subset Z $ is embedded in a regular scheme $ Z $.
The invariants are stable under localization and completion at $x$, \rot{and} they may be defined with the only datum $ R/J $. 
\rot{This fact is well-known} for the Hilbert-Samuel function, the  directrix;
\rot{For} the specific $\alpha,\beta,\gamma,\delta$, see Theorems \ref{Thm:alphabetainvariant_1_deltaonly} and \ref{Thm:alphabetainvariant_part2_nodelta} below.

\begin{Def}[\cite{CJS} Definition 3.3]
	\label{Def:boundary}
A {\em boundary $ \cB $ on $ Z $} is a collection of regular divisors on $ Z $, $ \cB = \{ B_1, \ldots, B_d \} $, such that they have at most n.c.~singularities, i.e., each of them is regular and they are intersecting transversally.
\end{Def}

Since the divisors in $ \cB $ have n.c.~we may choose the coordinates $ ( z_1, \ldots, z_n) $ at a point $ x $ such that $ (B_i)_x = V ( z_i) $, for $ 1 \leq i \leq t \leq n $ and $ \cB (x) := \{ B \in \cB \mid x \in B \} = \{ B_1, \ldots, B_t \} $.
 
\begin{Def}
	A closed subscheme $ D \subset X $ is called {\em $ \cB $-permissible} if it is permissible and has at most n.c.~with the boundary $ \cB $.
	The corresponding blow-up with center $ D $ is called a {\em $ \cB $-permissible blow-up}. 
	
	After a $ \cB $-permissible blow-up the {\em transform $ \cB' $ of a boundary $ \cB $} is defined to be the union of the strict transform of the components of $ \cB $ and the exceptional divisor of the blow-up.
\end{Def}

 The notion of transversality which is crystal clear in the embedded case is very subtle in the non embedded case. 
 This is the difficult point which is overcome in \cite{CJS}~section~4.    
 In the non embedded case, a boundary $\mathcal B$ is defined as a multiset of locally principal closed subschemes of $X$, \cite{CJS} Definition~4.1.
 See \cite{CJS} Definition~4.2 for the definition of normal crossing of a regular subscheme $D\subset X$ with $\mathcal B$. 
 Of course, $D\subset X$ is permissible for $\mathcal B$ at $x\in X$ if it is normal crossing with $\mathcal B$ and permissible for $X$ at $x$. 
 The total transform of $\mathcal B$ after a permissible blowing up is the union of the strict transforms of the components of $\mathcal B$ and of the exceptional divisor. 
 Remark~4.22 in \cite{CJS} clarifies this problem: 
 it says that we may assume  that $X$ can be embedded in a regular excellent scheme $Z$ with simple normal crossings boundary $\mathcal {B}_Z$ and that $\mathcal {B}$ is the pull back of  $\mathcal {B}_Z$ to $X$. 
 Let us recall this remark.

 \begin{Rk}[\cite{CJS}~Remark~4.22]\label{Rk:nonembedded}
 Let $X$ be an excellent scheme, let $\mathcal B$ be a boundary on $X$, and let $x\in X$. 
 Assume a property concerning $(X,\mathcal B,x)$ can be shown by passing to the local ring $\mathcal{O}:=\mathcal{O}_{X,x}$
 and its completion $\widehat{\mathcal{O}}$. 
 Then the following construction is useful. 
 The ring $\widehat{\mathcal{O}}$ is the quotient of a regular excellent ring $R$. 
 Let $\mathcal {B}(x):=\{\{\mathcal {B}_1,\ldots,\mathcal {B}_r \}\}$ be the multiset of components of $\mathcal B$ passing through $x$, and let $f_1,\ldots,f_r$ the local functions defining them in  $\mathcal{O}:=\mathcal{O}_{X,x}$. 
 We denote by $ \widehat{\mathcal{B}(x)} $ the corresponding boundary on $ \widehat{\mathcal{O}} $.
 Then we get a surjection:
 $$
	 R[X_1,\ldots,X_r] \twoheadrightarrow \widehat{\mathcal{O}}
 $$
 mapping $X_i$ to $f_i$, and the functions $X_i$ define a simple normal crossing boundary $\mathcal {B}_Z$ on $ Z:= \Spec ( R[X_1,\ldots,X_r]) $ such that $\widehat{\mathcal{B}(x)}$ is its pull-back under 
 $ \Spec ( \widehat{\mathcal{O}}) \hookrightarrow Z$.
 \end{Rk}


It is important to keep track of the resolution process and to remember at which step a boundary component occurred and how it is related to the maximal value of the Hilbert-Samuel function.
(We explain this more precisely after the next definition).
In order to get hands on this issue, one can use the following notion:
 
\begin{Def}[\cite{CJS} Definition 3.6]
	\label{Def:history_function}
	Let $ \cB $ be a boundary on $ Z $.
	A {\em history function} for a boundary $ \cB $ on $ X $ is a function 
	$$ 
		O_X: X \to \{ \mbox{ subsets of } \cB \, \}; \; \; x \mapsto O_X(x) 
	$$ 
	which satisfies the following conditions:
	\begin{itemize}
		\item[(1)]	For any $ x \in X $, $ O_X(x) \subset \cB (x) $.
			
	
		\item[(2)]	For any $ x,y \in X $ with $ x \in \overline{ \{ y \} } $ and $ H_X ( x ) = H_X ( y ) $, we have $ O_X(y) \subset O_X(x) $.
			
	
		\item[(3)]	For any $ y \in X $, there exists a non-empty open subset $ U \subset \overline{ \{ y \} } $ such that $ O_X(x) = O_X(y) $ for all $ x \in U $ such that $ H_X ( x ) = H_X ( y ) $.
\end{itemize}
	For such a function, we set, for $ x \in X $, 
	$$
		N_X ( x ) := \cB ( x ) \setminus O_X ( x ).
	$$
	A component of $ \cB $ is called {\em old for $ x $} if it is a component of $ O_X ( x) $,
	while a component of $ N_X ( x ) $ is called {\em new for $ x $.}
\end{Def}

\noindent
If no confusion is possible, we set $ O(x):= O_X (x) $ and $ N(x) := N_X(x) $.
 
Let us recall the history function which \cite{CJS} \rot{uses} \rot{(}see loc.~cit.~Lemma 3.7 and Lemma 3.14 and the part before\rot{)}:
If we are at the beginning of the resolution process or if the maximal value of the Hilbert-Samuel function just dropped to $ \nu $, then all boundary components are old, $ O ( x ) = \cB ( x ) $.

Let $ \pi: Z' \to Z $ be a blow-up with center $ D \subset X $, $ x \in D $.
Consider a point $ x' \in X' $ which is near to $ x $, $ \pi ( x' ) = x $ and  $ H_{X'} ( x' ) = H_X ( x ) = \nu $.
The boundary $ \cB'(x') $ on $ Z' $ consists of the strict transforms of the components in $ \cB(x) $ and the exceptional divisor of the last blow-up. 
The strict transforms of old components stay old, whereas the new exceptional divisor is defined to be new.

This continues for further blow-ups and near points; 
the strict transforms of old components remain to be old, and the new exceptional divisors and their strict transforms which arose after $ \nu $ appeared for the first time are defined to be new.
As soon as the Hilbert-Samuel function drops, say to $ \nu' < \nu $, this process starts over and all the boundary components are old (since they appeared before $ \nu' $ became a maximal value of $ H_X $).

Note that in the notation of \cite{CJS} Definition 3.15 we consider the complete transform  $ ( \cB', O' ) $ of $ ( \cB, O ) $.

\begin{Def}
\label{Def:logHS_function}
	\begin{itemize}
	\item[(1)]
	The {\em log-Hilbert-Samuel function} $ H^O_X : X \to \INN \times \IN $ is defined by 
	$$ 
		H^O_X ( x ) = ( H_X ( x ) , \, | O(x) | ) .
	$$
	Here $ \INN \times \IN $ is equipped with the lexicographical order.
	(But still we use the product order on $ \INN $).
		
	
	\item[(2)]
	For $ \tnu \in \INN \times \IN $, we define
	$
		X ( \geq \tnu ) = X^O ( \geq \tnu ) := \{ x \in X \mid H_X^O ( x ) \geq \tnu \} 
	$
	and
	$
		X (  \tnu ) = X^O (  \tnu ) := \{ x \in X \mid H_X^O ( x ) = \tnu \}  
	$
	and
	$ \Sigma_X^O := \{ H_X^O ( x ) \mid x \in X \} $ and $ \Sigma^{max,O}_X $ denotes the set of maximal elements in $ \Sigma_X^O $ and $ X_{max}^O := \bigcup_{ \tnu \in \Sigma^{max,O}_X } X^O ( \tnu ) $.
		
	
	\item[(3)]	
	In the local situation at a point $ x \in X $ ($ R = \cO_{Z,x}, J \subset R $ defining  ideal of $ X $) we set 
	$$ 
		J^O := J \cdot I_{O(x)} ,
	$$
	where $ I_{O(x)} $ denotes the ideal defining the divisor given by the old components of $ \cB ( x ) $, i.e., $ I_{O(x)} = \langle \, \vp \, \rangle $ with $ \vp = \vp_1 \cdots \vp_d $ and the regular elements $ \vp_1, \ldots, \vp_d \in R $ are generators of the old boundary components $ O(x) $.
	
	We denote by $ \Dir^O_x ( X ) $ the directrix associated to $ J^O $, by $ e^O_X ( x ) $ its dimension and by $ I (\, \Dir^O_x ( X ) \, ) $ the corresponding ideal in the graded ring.
	\end{itemize}
\end{Def}

\noindent
We use the notation $ e ( J )  := e_X ( x )$ and $ e ( J^O ) := e^O_X ( x ) $.

\smallskip

Analogous to Theorem \ref{Thm:specializationHS} one has

\begin{Thm}[\cite{CJS} Lemma 1.34 and Theorem 3.8]
\label{Thm:HS^O_usc}
	Let $ X $ be an excellent Noetherian scheme, $ \cB $ be a boundary, and $ O $ a history function.
	\begin{itemize}
		\item[(1)]
				If $ x \in X $ is a specialization of $ y \in X $, i.e., $ x \in \overline{ \{ y \} } $, then $ H_X^O ( x ) \geq H_x^O ( y ) $.
			
	
		\item[(2)]
				For any $ y \in X $, there exists a dense open subset $ U \subset  \overline{ \{ y \} } $ such that $ H_X^O ( x ) = H_X^O ( y ) $ for every $ x \in U $.
			
	
		\item[(3)]
				The function $ H_X^O  : X \to \INN \times \IN $ is upper semi-continuous and takes only finitely many values. 
				
	
		\item[(4)]
				For every $ \tnu \in \INN \times \IN $, $ X^O ( \tnu ) $ is locally closed, with closure contained in $ X^O (\geq \tnu ) $.
				In particular, $ X^O ( \tnu ) $ is closed if $ \tnu \in \Sigma^{max,O}_X $ is a maximal element, and $ X_{max}^O $ is closed.
	\end{itemize}
\end{Thm}


\begin{Lem}
	Let the situation be as in the previous Definition part (3);
	$ J \subset R = \cO_{Z,x} $ locally defining $ X $ at $ x $, and $ \vp_1, \ldots, \vp_d \in R $ are generators of the old boundary components $ O(x) $.
	Denote by $ \maxIdeal \subset R $ the maximal ideal of $ R $.
	Then we have
	$$ 
		I(\, \Dir_x^{O} ( X ) \,) = I( \, \Dir_x ( X ) \, ) +  \langle \; in_\maxIdeal ( \vp_i) \mid 1\leq i \leq d \; \rangle .
	$$
\end{Lem}

Hence $ \Dir^O_x ( X ) $ (resp.~$ e^O_X ( x ) $) is exactly as the $ \Dir^O_x ( X ) $ (resp.~$ e^O_X ( x ) $) as defined in \cite{CJS} Definition 3.9.

\begin{proof}
Take any $\Psi \in In_\maxIdeal (J^O ) $ ($ J^O = J \cdot I_{O(x)} $), then $ \Psi = \left( \prod_{1\leq i \leq d}\Phi_i^{a(i)} \right) \cdot G $, with integers $ a(i) > 0 $ and 
$ \Phi_i=in_\maxIdeal (\vp_i) $, for $ 1 \leq i \leq d $, and $ G $ not divisible by any of the $\Phi_i$.
There exists a Hasse-Schmidt derivation $ D_i $ such that $ D_i(\Psi)= \Phi_i^{a(i)} $, so by a result of Giraud (see \cite{GiraudMaxPos} Lemma 1.6 or \cite{ComputeRidge} Corollary 2.3),
$ \Phi_i^{a(i)} \in I(\, \Rid_x^{O} ( X ) \, ) $, $ 1 \leq i \leq d$.
(Here $ \Rid_x^{O} ( X ) $ denotes the ridge associated to $ J^ O $, the latter is a generalization of the directrix, where the linear forms $ (Y_1, \ldots, Y_r ) $ in Definition \ref{Def:dir} have to be replaced by additive polynomials $ ( \sigma_1, \ldots \sigma_s ) $, see also Remark \ref{Rk:ridge}).
This leads to $ \Phi_i \in I(\, \Dir_x^{O} ( X ) \,)$, and
$  \langle \; in_\maxIdeal ( \vp_i) \mid 1\leq i \leq d \; \rangle \subset I(\, \Dir_x^{O} ( X ) \,) $.

A similar argument shows 
$ I( \, \Dir_x ( X ) \, ) \subset I(\, \Dir_x^{O} ( X ) \,) $.
Hence we get the inclusion ``$\supset$" of the desired equality.
On the other hand, the opposite inclusion 
$I(\, \Dir_x^{O} ( X ) \,) \subset I( \, \Dir_x ( X ) \, ) +  \langle \; in_\maxIdeal ( \vp_i) \mid 1\leq i \leq d \; \rangle $
is obvious and we obtain the result.
\end{proof}

\begin{Lem}
	Let $ ( R, \maxIdeal) $ be an excellent regular local ring and $ J \subset R $ a non-zero ideal. 
	For an element $ \vp \in R $, we put $ \mu_\vp  := \ord_\maxIdeal ( \vp ) $ and
	$$
			J^\vp := J \cdot \langle \, \vp \, \rangle .
	$$
	Let $ ( f ) = ( f_1, \ldots, f_m ) $ be a standard basis of $ J $ (Definition \ref{Def:standardbasis}) and $ \nu_i := \ord_\maxIdeal ( f_i ) $, for $ 1 \leq i \leq m $.
	We have
	\begin{enumerate}
			\item[(1)]
				$ (\, \vp \cdot f_1, \, \ldots,\, \vp \cdot f_m \,) $ is a standard basis of $ J^\vp $, and

			
			\item[(2)]	
				$ \nus ( J^\vp, R ) = (\, \mu_\vp  + \nu_1,\, \mu_\vp  + \nu_2, \ldots, \, \mu_\vp  + \nu_m ,\, \infty, \infty, \ldots \,)$.
	\end{enumerate}
\end{Lem}

In the lemma $ \vp $ can be \emph{any element} in $ R $ and is not forced to be a monomial $ \vp_1 \cdots \vp_d $ of regular elements in $ R $.
If $ \vp $ is the monomial given by the old boundary components as above, then $ J^\vp = J^O $.

\begin{proof}
An element in $ h \in J^\vp $ can be written as
$$
	h = A_1 \vp f_1 + \ldots + A_m \vp f_m = \vp \cdot ( A_1 f_1 + \ldots + A_m f_m  ),
$$ 
for certain $ A_i \in R $, $ 1 \leq i \leq m $.
Clearly, $ A_1 f_1 + \ldots + A_m f_m \in J $ and 
$$
	in_\maxIdeal ( h ) = in_\maxIdeal ( \vp ) \cdot in_\maxIdeal ( A_1 f_1 + \ldots + A_m f_m ),
$$
imply $ In_\maxIdeal ( J^\vp ) = \langle in_\maxIdeal ( \vp ) \rangle \cdot In_\maxIdeal ( J) \subset gr_\maxIdeal ( R ) $.
Now, we only have to go through the definitions of a standard basis and the $ \nus $-invariant (both Definition \ref{Def:standardbasis}) and obtain by using that $ ( f ) $ is a standard basis of $ J $ that $ (\, \vp \cdot f_1, \, \ldots,\, \vp \cdot f_m \,) $ is a standard basis of $ J^\vp $ and 
$ 
	\nus ( J^\vp, R ) = (\, \mu_\vp  + \nu_1,\, \mu_\vp  + \nu_2, \ldots, \, \mu_\vp  + \nu_m ,\, \infty, \infty, \ldots \,).
$
\end{proof}

\begin{Rk}
In fact, we even have that a center $ D = V ( \fp ) $, with an ideal $ \fp \subset R $, is permissible for $ J^\vp $ if and only if it is permissible for $ J $ as well as $ \vp $.

Moreover, if the $ \nus $-invariant of $ J^\vp $ drops after a permissible blow-up, then the $ \nus $-invariant of $ J $ or the one of $ \vp $ also drops, i.e., $
	\nus ( (J^\vp)', R')  < \nus ( J^\vp, R) $
	implies 
	$ \nus ( J', R')  < \nus ( J, R) 
	\mbox{ or }
	\nus ( \langle \vp \rangle', R')  < \nus ( \langle \vp \rangle \, R) ,
$
where $ R '$ is the local ring after the blow-up and $ (.)'$ denotes the strict transform of $ (.) $ in $ R' $.

If $ \vp = \vp_1 \cdots \vp_d $ is product of regular elements $ ( \vp_1, \ldots , \vp_d ) $ in $ R $ which can be extended to a regular system of parameters ({\RSP}) for $ R $, then the system $ ( \vp_1, \ldots , \vp_d ) $ can be extended to a system defining the directrix of $ J^\vp $

This can all be seen by using an interpretation of these things in the language of the idealistic exponents.
(For an introduction to the notion of the idealistic exponents we refer to the first sections of \cite{BerndThesis}, where also references to the original literature are given).
The problem of lowering the $ \nus$-invariant of $ J $ (and thus the problem of lowering the Hilbert-Samuel function, see Theorem \ref{Thm:BennHiroSingh}) is equivalent to the problem of resolving the idealistic exponent 
$$
	\IE = ( f_1, \nu_1 ) \cap \ldots \cap ( f_m, \nu_m ) .
$$
By passing to $ J^O = J \cdot I_{O(x)} $, $ I_{O(x)} = \langle \, \vp \, \rangle $, we consider the idealistic exponent
$$
	\IE^O = ( \vp \cdot f_1, \nu_1 + \mu_\vp ) \cap \ldots \cap ( \vp \cdot f_m, \nu_m + \mu_\vp ),
$$
for which we have the following equivalence (roughly speaking, two idealistic exponents are equivalent if they undergo the same resolution process)
$$
	\IE^O \sim  ( f_1, \nu_1 ) \cap \ldots \cap ( f_m, \nu_m ) \cap ( \vp, \mu_\vp) = \IE \cap  ( \vp, \mu_\vp).
$$
From this one can deduce the first two statements of the remark.

If $ \vp = \vp_1 \cdots \vp_d $ with regular elements as above, then $ ( \vp, \mu_\vp) \sim ( \vp_1, 1 ) \cap \ldots \cap ( \vp_d, 1 ) $.
(Note that $ \mu_\vp = \sum_{ i = 1 }^d 1 = d $).
This implies the third assertion.

The process of intersecting with old boundary components (or say exceptional divisors) is a well-known technique and is already used by others, e.g.~in the proof for constructive resolution of singularities over the fields of characteristic zero by Bierstone and Milman \cite{BM} (see also \rot{\cite{BerndThesis} Construction~3.3.2}, where this is explained in the setting of idealistic exponents), or the proof for resolution of singularities in dimension three by Piltant and the first author in \cite{CPmixed} and  \cite{CPmixed2}.
\end{Rk}


As we mentioned before we do not have a control on the old components;
e.g.~for $ J = \langle \, y^p + u_1^p \cdot f ( u_1, u_2 ) \,\rangle $, with some $ f ( u_1, u_2 ) \in R $, $ D := V ( y, u_1 ) $ is a permissible center, but $ V (\,y + u_2^2\, ) $ may appear as an old boundary component in which case $ D $ is not $ \cB $-permissible.
Therefore we have to consider the singularity defined by $ J^O $. 
This corresponds to the process of transforming the maximal Hilbert-Samuel locus at the beginning into a $ \cB $-permissible center, which is in particular n.c.~with the boundary.
The aim is to lower the Hilbert-Samuel function or at least to make the number of appearing old boundary components strictly decrease.
(In fact, this means we want to lower the $ \nu^\ast$-invariant of $ J^O $).
Thus the points which have to be considered are the following:

\begin{Def}
\label{Def:very_O-near}
Let $ \pi : Z' = Bl_D ( Z ) \to Z $ be a $ \cB $-permissible blow-up of $ Z $ with center $ D \subset X $.
Consider {$ x \in D $ and} $ x' \in \pi_Z^{ - 1 }(x) $ .
\begin{itemize}
	\item[(1)] $ x' $ is called {\em $ O $-near to $ x $} if $ H_{X'}^O ( x') = H_{X}^O ( x) $.
	(Note that we denote the transform of the history function by $ O $ and not by $ O'$).
		
	
	\item[(2)] $ x' $ is called {\em very $ O $-near to $ x $} if it is $ O $-near, very near, and $$  e_{X'}^O ( x' ) = e_X^O ( x )_{k(x')} - \delta_{x'/x} .$$
\end{itemize} 
\end{Def}

Clearly, $ x ' $ is $ O $-near to $ x $ if and only if $ x' $ is near to $ x $ and contained in the strict transform of all $ B \in O ( x ) $.


The analogous results as in Theorem \ref{Thm:HiroMizutani} and Theorem \ref{Thm:BennHiroSingh}(1) and (2) are true if we replace all the definitions by log-definitions with the history function:

\begin{Thm}[\cite{CJS} Theorem 3.18]
\label{Thm:HiroMizutani$^O$}
Let {$ x' \in \pi_Z^{ - 1 }(x) $} be $ O $-near to {$ x \in D $}.
Assume that we have either $ \car ( \resfield ) = 0 $ or $ \car ( \resfield ) \geq \dim (X)/2 + 1 $ 
($ \resfield $ denotes the residue field of $ x $).
Then
$$
	x' \in \IP ( \Dir_x^O ( X ) / T_x ( D ) ) \subset \pi^{ -1 }_Z ( x ).
$$
\end{Thm}

\smallskip

\begin{Thm}[\cite{CJS} Theorem 3.23(1)]
\label{Thm:BennHiroSinghO}
	Let {$ x \in D $ and} $ x ' \in \pi_Z^{ -1 } ( x ) $.
	Suppose that either $ \car ( k (x) ) = 0 $ or $ \car ( k (x) ) \geq \dim ( X )/ 2 + 1 $.
	If $ x' $ is $ O $-near and very near to $ x $, then 
	$$  
		e_{X'}^O ( x' ) \leq e_X^O ( x )_{k(x')} - \delta_{x'/x} .
	$$
\end{Thm}
 

After recalling the necessary definitions and result of \cite{CJS} we can now define the first part of the invariant $ \io = ( \ioo, \ioc, \iop ) $ and prove that it does not increase under permissible blow-ups:

\begin{Def}[\bf First part of the invariant]
	Let $ x \in X  \subset Z $ and $ \cB $ be a boundary on $ Z $. 
	We define
	$$ 
	\begin{array}{c}
	\ioo := \ioo ( X, \cB, x )  := (\, H_X ( x ),\, | O_X (x) |,\, e_X (x) ,\, e^O_X ( x ) \,) =
	\\[9pt]
	\hspace{16pt} 
		 = (\, H_X^O ( x ),\, e_X (x) ,\, e^O_X ( x ) \,) 
	\end{array}
	$$
	Note that $ \ioo \in 
	\INN \times \IN \times \{ 0, 1, \ldots, \dim ( X ) \} \times \{ 0, 1, \ldots, \dim ( X ) \} \subset
	\INN \times \IN^3 $.
	Further, all entries are independent of the embedding of $ X $ in $ Z $.
\end{Def}
 
\begin{Prop}
\label{Prop:ioo_non_increase}
Let $ \pi : Z' = Bl_D ( Z ) \to Z $ be a $ \cB $-permissible blow-up of $ Z $ with center $ D \subset X $.
Suppose that either $ \car ( k (x) ) = 0 $ or $ \car ( k (x) ) \geq \dim ( X )/ 2 + 1 $.
Consider {$ x \in D $ and} $ x' \in \pi^{ - 1 }(x) $ .
Then we have
$$ \ioo ( X', \cB', x' ) \leq \ioo ( X, \cB, x ) .$$
\end{Prop}

\begin{proof}
This is an immediate consequence of Theorem \ref{Thm:BennHiroSingh} and Theorem \ref{Thm:BennHiroSinghO}.
\end{proof}
 

\begin{Rk}
\label{Rk:ridge}
	We remark that, in higher dimensions, it might be important to take not only the dimension of the directrix but also the dimension of the ridge into account when constructing a first invariant for singularities.
	The latter is a generalization of the directrix such that Theorem \ref{Thm:HiroMizutani} is true without the assumption on the characteristic of the residue field.
	The drawback of this is that the ridge is no longer generated only by the linear forms, and we have to take the additive polynomials as generators.
	{But in contrast to the dimension of the directrix the dimension of the ridge is upper semi-continuous, see Example \ref{Ex:e_not_usc}.}
\end{Rk}

%
%
%
%
%
%
%
%
%
%
%
%
%
%
%
%
%
%



\section{The \cite{CJS} algorithm}

After introducing the first measure for the complexity of a singularity, we recall the algorithm by \cite{CJS} in detail in this section.
Note that $ \ioo $ and its good behavior under $ \cB $-permissible blowing ups (Proposition \ref{Prop:ioo_non_increase}) 
is independent of a specific strategy to resolve singularities.

\begin{Algorithm}[\cite{CJS} Remark 5.29]
\label{Algorithm_CJS}
Let $X_0 $ be a reduced excellent Noetherian scheme of arbitrary dimension, embedded in a regular Noetherian scheme $ Z_0 $, $ X_0 \subset Z_0 $. 
Let $ \cB_0 $ a boundary on $ Z_0 $.
By Theorem \ref{Thm:HS^O_usc}(3) the log-Hilbert-Samuel function takes only finitely many maximal values and, by \ref{Thm:HS^O_usc}(1), the maximal strata are disjoint.

Let $ \tnu \in \Sigma^{max,O}_X $ be a maximal value and set $ Y_0 := X_0(\tnu ) $. 
Since $ X_0 $ is reduced, we have $ \dim ( Y_0 ) < \dim (X_0 ) $.
By the induction on dimension we obtain a finite sequence of $ \cB $-permissible blowing ups,
$$
	Z_0 \longleftarrow Bl_{D_0} (Z_0) =: Z_1 \longleftarrow \ldots \longleftarrow Bl_{D_{m_0-1}} (Z_{m_0 - 1}) =: Z_{m_0},
$$
such that the strict transform $ Y_{m_0}^{(0)} $ of $ Y_0 $ is $ \cB_{m_0} $-permissible, where $ \cB_{m_0} $ denotes the boundary at this step.
Therefore $ Y_{m_0}^{(0)} $ is an appropriate choice for the center of the next blowing up.

We denote by $ X_{m_0} $ the strict transform of $ X_0 $ in $ Z_{m_0} $.
During the previous sequence of blowing ups we possibly created new irreducible components in the $ \tnu $ stratum $ Y_{m_0} := X_{m_0} ( \tnu ) $.
In order to be able to handle them, the algorithm assigns to the appearing irreducible components certain labels which are determined by the first occurrence of these irreducible components,
$$
	Y_{m_0} = Y_{m_0}^{(0)} \cup Y_{m_0}^{(1)} \cup \ldots \cup Y_{m_0}^{(m_0)}.  
$$
Here $ Y_{m_0}^{(i)} $ are those irreducible components (resp.~their strict transforms) which occurred after the $ i $-th blowing up for the first time.
See also Remark \ref{Rk:label} below for more details in the case that $ X $ is two-dimensional.

The next center in the \cite{CJS} algorithm is $ D_{m_0} = Y_{m_0}^{(0)}  $,
$$
	Z_{m_0} \longleftarrow Bl_{D_{m_0}} ( Z_{m_0} ) =: Z_{m_0 + 1}.
$$
If $ Y_{m_0 + 1} := X_{m_0 +1} ( \tnu ) = \emptyset $ then the algorithm starts from the very beginning. 
Suppose $ Y_{m_0 + 1} \neq \emptyset $.
The labeling of the irreducible components in $ Y_{m_0 +1} $ goes as follows:
Let $ W \subset Y_{m_0 + 1} $ an irreducible component lying above the center $ D_{m_0} $.
If $ W $ dominates \rot{an irreducible component of}
$ D_{m_0} $ then $ W $ inherits the label of $ D_{m_0} $.
Otherwise, $ W $ gets the label $ m_0 + 1 $.
Hence we obtain
$$
	Y_{m_0 + 1} = Y_\rot{m_0+1}^{(0)} \cup Y_{m_0 + 1}^{(1)} \cup \ldots \cup Y_{m_0 + 1}^{(m_0)} \cup Y_{m_0 + 1}^{(m_0+ 1)},  
$$
where $ Y_{m_0 + 1}^{(i)} $ is the strict transform of $ Y_{m_0}^{(i)} $, for $ 1 \leq i \leq m_0 $.

In the next step the \cite{CJS} algorithm repeats the previous for the irreducible components of smallest label, say $ Y_{m_0 + 1}^{(i_1)} \neq \emptyset $, for some $ \rot{0} \leq i_1  \leq m_0 + 1 $ such that $ Y_{m_0 + 1}^{(j)} = \emptyset $, for all $ j < i_1 $.
That means, first we use induction in order to make the strict transform $ Y_{m_1}^{(i_1)} $ of $ Y_{m_0+1}^{(i_1)} $ $ \cB_{m_1} $-permissible\rot{,} and once this \rot{is} achieved\rot{,} we blow up with center $ Y_{m_1}^{(i_1)} $.
\end{Algorithm}


\begin{Rk} 	
	\label{Rk:global}
\rot{In other approaches to resolution of singularities, the centers for blowing up are constructed locally in an affine neighborhood. 
	A crucial part is then to show that the local resolutions actually provide a global one,
	i.e., to show that the locally defined centers glue together to a global center.}
	
\rot{In contrast to this, there are no local constructions involved in the \cite{CJS}-algorithm.
	Therefore, the gluing problem does not appear
	and the above sequence of permissible blow-ups is global.}

\rot{Precisely, in dimension $\leq 2$, \cite{CJS} looks at  the log-Hilbert-Samuel stratum (the union of the maximal strata for $ H_X^O$)  in its entirety. 
It is the union of isolated points and irreducible curves.}

\rot{\cite{CJS} blows up the isolated points and if above one of them appears an exceptional curve,  then applies a ``fundamental sequence'' by \cite{CJS} Theorem 5.35, this will terminate.}

\rot{If there are components of dimension \blau{one} of the log-Hilbert-Samuel stratum, \cite{CJS} makes these curves,  and the new created by the blowing ups, normal crossing with the boundary by a sequence of blowing ups centered at closed points. When this is achieved,  \cite{CJS} destroys all dimension \blau{one} components by blowing up along all of them systematically.}

\rot{In a few words, once $ H_X^O$ is defined,  \cite{CJS} algorithm is given by the geometry of the log-Hilbert-Samuel stratum.}
\end{Rk} 

Following \cite{CJS} Remark 5.29 we do not make any assumptions on the dimension.
But we point out that the finiteness of the sequence of blowing ups constructed is proven only in dimension smaller than or equal to two \rot{(\cite{CJS} Theorems 5.35, 5.40)}.
In Examples \ref{Ex:dimensionthree} and \ref{Ex:dimensionthree_BM} below, we discuss some phenomena that appear in dimension three.

\smallskip

We like to point out that the algorithm of \cite{CJS} differs in characteristic zero from the usual one.

\begin{Ex}
\label{Ex:easy_difference}
Consider the variety $ X = V ( x^2 - y^2 z ) $ over a field of characteristic zero and with no boundary given.
Then \cite{CJS} chooses the singular locus $ V ( x, y ) $ as the center of the first blow-up.
On the other hand, $ V ( x ) $ has maximal contact with $ X $ and the coefficient ideal is defined by the monomial $ y^2 z $ and following the strategy of Bierstone and Milman (see \cite{BM}, \cite{BerndThesis}) we have to blow up with center $ V ( x, y ,z ) $.
Indeed the two algorithms are different.
\end{Ex}

{After a blow-up with center $ D $ the interesting points above $ x \in D $ are the $ O $-near points}, where the log-Hilbert-Samuel function does not drop.
If $ \dim ( X ) = 2 $, the assumption of Theorem \ref{Thm:HiroMizutani$^O$} on the characteristic of the residue field does always hold and the $ O $-near points are characterized by this theorem.
Therefore we get

\begin{Obs}(\cite{CJS}, Proof of Theorem 5.28, Step 1 and Step 2, Lemmas 3.20, 4.25, and 4.26){\bf .}
\label{Obs:GooOnear}
Let $ X $ be a Noetherian excellent scheme of dimension at most two embedded in a regular Noetherian scheme $ Z $ and let $ \cB $ be a boundary on $ Z $.
Let $ \pi : Z' \to Z $ be a $ \cB $-permissible blow-up with center $ D \subset X ( \tnu ) $, for some $ \tnu \in \Sigma_X^{max,O} $.
The center $ D \subset X ( \tnu ) $ is either a closed point or $ D $ is regular irreducible of dimension \rot{one} and n.c.~with $ \cB $.

If $ D = x \in X $ is a closed point, the good behavior of the near points is a consequence of \cite{CJS} Lemmas 3.20 and 4.25. 
Let us recall the situation:
$$ 
	D' := X' ( \tnu ) \cap \pi^{-1} ( x ) \subset \IP ( \Dir^O_x ( X ) ) \cong \IP^t_\resfield, 
$$
for $ t = e^O_X ( x ) - 1 \leq \dim ( X ) - 1 = 2 - 1 = 1 $.
If $ D' $ is empty, then there are no near points lying above $ x $.
If $ D' $ is non-empty, then it is either a union of closed points or a projective line over $ \resfield $.
In the latter case it is the whole reduced exceptional divisor of the blow-up $X'\longrightarrow X$.
Since the center is $ \cB $-permissible  we conclude that $ D' $ is $ \cB' $-permissible by \cite{CJS} Lemmas 3.20 and 4.25.

Now suppose $ D $ is regular irreducible of dimension \rot{one} and n.c.~with $ \cB $.
Denote by $ \eta $ its generic point and let $ x \in D $ be a closed point.
Then 
$$
	X' ( \tnu ) \cap \pi^{-1} ( x ) \subset \IP ( \Dir^O_x ( X )/ T_x ( D ) ) \cong \IP^s_\resfield, 
$$
for $ s = e^O_X ( x ) - 2 \leq 0 $.
Moreover, 
$$
	X' ( \tnu ) \cap \pi^{-1} ( \eta ) \subset \IP ( \Dir^O_\eta ( X ) ) \cong \IP^r_{k(\eta)}, 
$$
for $ r = e^O_X ( \eta ) - 1 \leq 0 $.
If $ X' ( \tnu ) \cap \pi^{-1} ( \eta ) \neq \emptyset $ then it consists of a unique point $ \eta' $ with $ k ( \eta' ) \cong k ( \eta ) $.
Thus $ \pi $ induces an isomorphism
$$
 X' ( \tnu ) \cap \pi^{-1} ( D ) =: D'  \stackrel{\cong}{\longrightarrow} D  .
$$
Therefore, we conclude that $ D' $ is either empty or a union of closed points or a curve isomorphic to $ D $ which is $ \cB' $-permissible by \cite{CJS} Lemma 4.26.
\end{Obs}

Having the previous in mind, let us have another look at the labeling procedure if $ \dim (X) = 2 $.

\begin{Rk}[\em Labels of the irreducible components of $ X ( \widetilde{ \nu } ) $ in dimension \rot{two}]
\label{Rk:label}
	\textcolor{white}{.}\\
	Let us assume that $ \dim (X_0) = 2 $ in Algorithm \ref{Algorithm_CJS}.
	We denote the strict transform of $ X_0 $ after $ m \geq 0 $ steps by $ X_m $.
	Suppose $ X_m ( \widetilde{ \nu } ) \neq \emptyset $.
	To get control on the new components in $ X_m ( \widetilde{ \nu } ) $, we label the irreducible component of $ X_m ( \widetilde{ \nu } ) $ according to \cite{CJS} as follows:
	\begin{itemize}
		\item At the beginning every irreducible component of $ X_0 ( \widetilde{ \nu } ) $ gets label 0.
	Hence we can assume in the remaining part $ m \geq 1 $.
		
		
		\item	If the algorithm blows up an isolated closed point $ x $ in $ X_{m-1} ( \tnu ) $ with label $ j \in \IN $, then all the irreducible components in $ X_m ( \widetilde{ \nu } ) $ lying above $ x $ inherits the label $ j $ since they are dominating $ x $.
	
	
		\item	If the center of the blow-up is a $ \cB $-permissible curve $ C \subset X_{m-1} ( \tnu ) $ with label $ j \in \IN $, and suppose there exists a curve $ C' \subset X_{m} ( \widetilde{ \nu } ) $ with $ C' \cong C $, then $ C ' $ inherits label $ j $.
				If there is an isolated closed point lying above $ C $ in $ X_m ( \widetilde{ \nu } ) $, then the point cannot dominate the whole irreducible component $ C $.
				Therefore we give this point a new label, namely the current year (number of steps) in the resolution process.
		
				
		\item	Suppose the center is a closed point strictly contained in a bigger irreducible component of $ X_{m-1} ( \widetilde{ \nu } ) $.
				Whatever is lying above the center it cannot dominate a whole irreducible component of $ X_{m-1} ( \widetilde{ \nu } ) $ and the created components in $ X_m ( \widetilde{ \nu } ) $ get a new label as before.
	\end{itemize}
			
	This means once we have blown up the $ \cB_{m_0} $-permissible strict transform $ Y_{m_0}^{(0)} $ of $ Y_0 $ in Algorithm \ref{Algorithm_CJS}	the upcoming centers are uniquely determined by the lowest label if $ \dim (X_0)  = 2 $.

	\rot{When computing a resolution by hand, one will usually cover $ X_0 $ and its strict transforms under blow-ups by local affine charts.
		The labeling as described above is then rather complicated since one has to consider the global resolution process at every step.}
\end{Rk}
	
	Note that in higher dimensions the components of lowest label do not necessarily give a $ \cB$-permissible center.
	Already in dimension three there might appear singular curves in $ X' ( \tnu) $ after blowing up a closed point:

\begin{Ex}
\label{Ex:dimensionthree}
Consider the three-dimensional variety given by the polynomial $ t^2 + xy^2 + z^3 + x^5 y $ over a field of characteristic two without any boundary given.
The singular locus is the origin $ V ( t, x, y, z ) $.
Hence \cite{CJS} take the origin as the center for the first blowing up.
At the origin of the $ X $-chart, we have coordinates $ (t',x,y',z') = (\frac{t}{x}, x, \frac{y}{x}, \frac{z}{x}) $ and the strict transform of the polynomial is
$$
	t'^2 + x y'^2 + x z'^3 + x^4 y'. 
$$
Its singular locus is $ V ( t', x , y'^2 + z'^3 ) $ which is a singular curve.
\end{Ex}

	Moreover, even if the center of the blow-up is a curve $ C $, after the blow-up there might appear some curve $ C' $ which does not dominate $ C $.
	This also needs to get a new label.
	
\begin{Ex}
\label{Ex:dimensionthree_BM}
Consider the threefold $ X = V ( t^2 + x^4 + y^2 z^5 + x^2 z^3 + y^7 z ) $ over a field of characteristic zero and with no boundary given.
Show that the singular locus is given by $ V ( t, x, y ) $.
This is then the center of the \cite{CJS}-strategy.
After the blow-up, we consider the origin of the $ Y $-chart with coordinates $(t',x',y, z) = (\frac{t}{y}, \frac{x}{y}, y , z )$.
The strict transform of $ X $ is 
$$
 X' = V ( t'^2 + x'^4 y^2 + z^5 + x'^2 z^3 + y^5 z ) 
$$ 
and the singular locus is $ V(t', y, z ) $.
The latter is a curve which does not dominate the center $ V(t,x,y) $. 

Moreover, the reader who is familiar with the invariant introduced by Bierstone and Milman (\cite{BM},\cite{BerndThesis}) realizes that it increases if we compare the value at the origin before the blow up, $ ( 2, 0; 2, \ldots ) $, with the value at the origin of the $ Y$-chart, $ ( 2, 0; \frac{5}{2}, \ldots ) $.
Hence this is another example showing the difference between the usual characteristic zero algorithm for resolving singularities and \cite{CJS}.
\end{Ex}

Therefore the higher dimensional case is more delicate and so far it is not known if the procedure of \cite{CJS} is finite or not.


\begin{Ex}
	Consider the variety over a field $ k $ given by
	$$
		f = x^2 + y^9 z^{{10}} = 0 .
	$$
	The maximal Hilbert-Samuel locus coincides with the locus of order two.
	There are two irreducible components $ L_1 := V ( x,y ) $ and $ L_2 := V ( x, z) $.
	Since we are at the beginning of the resolution process both get label 0.
	
	We have to blow up the origin in order to separate these two components.
	At the origin of the $ Z $-chart of the blow-up we have the coordinates $ (x',y',z) = (\frac{x}{z}, \frac{y}{z},z) $.
	The strict transform of $ f $ is $ f ' = x'^2 + y'^9 z^{17} $.
	The label of $ L_1' = V ( x', y' ) $ is still 0 and since $ L_3 := V ( x', z ) $ is lying on the exceptional divisor it gets the label 1.
	Note that the strict transform of $ L_2 $ is empty in this chart.
	Hence the component of the maximal Hilbert-Samuel locus with minimal label and thus the center for the next blow-up is $ V ( x', y' ) $.
	
	At the origin of the $ Y' $-chart we obtain the coordinates
	$ ( x'', y', z) = ( \frac{x'}{y'}, y', z)  $ and
	 the singularity is given by $ f'' = x''^2 + y'^7 z^{17} $.
	The line $ L_4 := V ( x'',y' ) $ lies in the singular locus and also on the exceptional divisor of the last blow-up.
	Moreover, it is dominating the center $ L_1' = V (x',y') $ of the last blow-up. 
	Thus we assign the label 0 to it.
	Further, the strict transform $ L_3' = V(x'', z) $ of $ L_3 $ still has the label 1.
	So, the center for the next blowing up following the algorithm of \cite{CJS} is $ L_4 = V ( x'', y' ) $.
\end{Ex}

If we skip the condition on the inheritance of the label when an irreducible component after a blow-up is dominating one before, then every new component in $ X' ( \widetilde{ \nu } ) $ would get a new label.
This means in the previous example that $ L_4 = V(x'', y') $ gets the label 2 (instead of 0) and the center for the next blowing up would be $ L_3' = V ( x'' ,z ) $ and not $ L_4 = V ( x'', y' ) $.

From a practical point of view this variant does not look very elegant.
For instance we have above, using the labels without inheritance,
$$
	f'' = x''^2 + y'^7 z^{17},
	\hspace{10pt}
	\mathrm{Sing}(f'') = V ( x'' ,z )^{(1)} \cup V ( x'', y' )^{(2)},
$$
where the $ (\ldots)^{(i)} $ indicates that the assigned label is $ i $.
Hence the center is $ V(x'', z) $.
At the origin of the $ Z $-chart, the coordinates are $ (x''',y',z') = (\frac{x''}{z}, y', z') $ and
$$
	f''' = x'''^2 + y'^7 z^{15},
	\hspace{10pt}
	\mathrm{Sing}(f''') = V ( x''' ,z )^{(3)} \cup V ( x''', y' )^{(2)}.
$$
Now, the next center is  $ V ( x''', y' ) $.
At the origin of the $ Y' $-chart, the coordinates are $ (x'''',y',z') = (\frac{x'''}{y'}, y', z') $ and
$$
	f'''' = x''''^2 + y'^5 z^{15},
	\hspace{10pt}
	\mathrm{Sing}(f'''') = V ( x'''' ,z )^{(3)} \cup V ( x'''', y' )^{(4)}.
$$
The next center is  $ V ( x'''', z ) $ and so on.
In other words, we alternately blow up the $ z $-axis and the $ y' $-axis of the corresponding charts until the exponent of $ y ' $ is smaller than two.

If we follow the original labeling, we would first make the exponent of $ y' $ decrease to one and after that the remaining blowing ups make the exponent of $ z $ decrease.


\begin{Rk}[\textsc{Caution:} \em Different notions of old component] 
On one hand, we have the old {\em boundary components}, which do not necessarily have anything to do with the maximal Hilbert-Samuel locus. 
Their importance lies in the requirement that the center has to be n.c.~with the boundary.

On the other hand, we labeled the {\em irreducible components of $ X ( \widetilde{ \nu } ) $}.
But as we have seen in the third step of the example, these two labellings are different;
in fact, $ L_4 = V ( x'', y') $ is an irreducible component of the maximal Hilbert-Samuel locus with label zero,
while it is contained in the exceptional divisor $ V(y') $ of the last blow-up, which is new.  
\end{Rk}

%
%
%
%
%
%
%
%
%
%
%
%
%
%
%
%
%
%



\section{Hironaka's characteristic polyhedron}
\label{sec:HiroPoly} 

One of the main tools in the construction of our local invariant is the characteristic polyhedron, which reflects a refined nature of the singularities. 
Hironaka \cite{HiroCharPoly} introduces the notion of the characteristic polyhedron under the following situation (see also \cite{CJS} section 7): 


{\em Let $ ( R, \maxIdeal, \Resfield = R / \maxIdeal ) $ be an arbitrary regular local ring with maximal ideal $ \maxIdeal $ and let $ (0) \neq \ideal \subset R $ be an ideal.}

\smallskip

As will be shown later in Example \ref{Ex:e_not_usc}, 
our invariant is {\it not} upper semicontinuous.
Its strata may not even be constructible. 
We define the invariant at each point $ x \in X $. 
We will show that it is stable by localization and passing to completion, and {\it does not depend on the embedding}. 
This allow us to pass to the completion of the local ring $ \cO_{X,x} $.
By the Cohen structure theorem and Remark~\ref{Rk:nonembedded}, we can reduce \rot{our analysis} to the local situation mentioned above, with $ R $ even being complete.
Nonetheless, we do {\it not} require $ R $ to be complete in this section.

The characteristic polyhedron in which we are interested is the one associated to one of the ideal\rot{s} $ I \in \{ J^O = J \cdot I_{O_{{X}}} , \, I_C,\, I_C^O = I_C \cdot I_{O_C} \} $, where $ J \subset R $ is the ideal which defines $ X $ locally at some point $ x $,  $ I_{O_{{X}}}:= I_{O_{{X}}(x)} $ denotes the ideal defining the divisor given by the old boundary components of $ \cB ( x ) $, $ I_C \subset R $ is the ideal of the union of some components of the Hilbert-Samuel locus, and $ I_{O_C}:= I_{O_C(x)} $ denotes the ideal defining the divisor given by the boundary components of $ \cB ( x ) $ which are old for $ C $ (i.e., old \wrt the value $ H_C(x) $).

\smallskip

Let $ ( u , y ) = ( u_1, \ldots, u_e; y_1, \ldots, y_r ) $ be a {\RSP}~of $ R $.
In fact, Hironaka defines the characteristic polyhedron in a slightly more general setting, but for us the interesting case is the one where  the system $ ( y ) $ is chosen such that their initial forms $ Y_j := in_{\maxIdeal}  ( y_j ) = y_j \mod \maxIdeal^2 $ generate the ideal of the directrix $ \Dir_\maxIdeal ( {\ideal} ) $ of $ {\ideal} $ at $ \maxIdeal $ (Observation~\ref{Obs:Directrix}, Definition \ref{Def:dir}).
Therefore we restrict our attention to this case. 

For  $ g \in R $, we have an expansion in a {\em finite} sum
\begin{equation}
\label{eq:expansion}
	g = \sum_{(A,B) \in \IZ^{e + r }_{\geq 0 } } C_{A,B} \, u^A \, y^B
\end{equation}
with coefficients $ C_{A,B} \in R^\times \cup \{ 0 \} $.
Due to \cite{HiroCharPoly} at the beginning of \S 2, we observe that, if we fix the \rot{coordinate} system $ ( u, y ) $ and if we require the {number of} appearing exponents to be minimal, then the set $ \{ (A,B) \mid C_{A,B} \neq 0 \} $ is uniquely determined 
(even though the set $ \{ C_{A,B} \}  $ may vary).

\begin{Def}
	\label{Def:Poly} 
\begin{itemize}
	\item[(1)]	A {\em $ F $-subset} of $ \IR^e_{\geq 0 } $ is a closed convex subset $ \Delta \subset \IR^e_{\geq 0 } $ such that $ v \in \Delta $ implies $ v + w \in \Delta $ for every $ w \in \IR^e_{\geq 0 } $.
		
	
	\item[(2)]	Let $ g \in R $ be an element in $ R $ with $ g \notin \langle u \rangle $.
				Then we can expand $ g $ as in \eqref{eq:expansion} into a finite sum
				$
					g = \sum C_{A,B} \, u^A \, y^B
				$
				with coefficients $ C_{A,B} \in R^\times \cup \{ 0 \} $.
				Set \rot{$ \nu := \ord_{\overline{\maxIdeal}} ( \overline{g} ) $,
					where $ \overline{g} := g \mod \langle u \rangle \in R/\langle u \rangle $ and 
					$ \overline{\maxIdeal} := \maxIdeal \cdot R/ \langle u \rangle $.} 
				The {\em polyhedron associated to $ ( g ,u, y ) $} is then defined as the smallest $ F $-subset $ \poly guy $ containing the points
	$$
		\left\{
			\;	\frac{A}{\nu - |B|} \;\Big| \; C_{A,B} \neq 0 \, \wedge \, |B| < \nu
			\;
		\right\} .
	$$	
			
	
	\item[(3)]	Let $ ( f ) = ( f_1, \ldots, f_m ) $ be a system of elements in $ R $ with $ f_i \notin \langle u \rangle $.
				Then the {\em polyhedron {$ \poly fuy $} associated to $ ( f, u, y ) $} is defined to be the smallest $ F$-subset containing $ \bigcup_{ i = 1 }^m \poly{f_i}{u}{y} $.
		
	
	\item[(4)]	For an ideal $ {\ideal} \subset R $ we set:
				$$
					\poly {\ideal}uy := \bigcap_{ (f) }  \poly fuy,
 				$$
				where the intersection runs over all possible standard bases $ ( f ) = ( f_1, \ldots, f_m ) $ of $ {\ideal} $ (Definition \ref{Def:standardbasis}).

				Finally, the {\em characteristic polyhedron of $ ( {\ideal} ; u ) $} is defined by
				$$
					\cpoly {\ideal}u := \bigcap_{(y)} \poly {\ideal}uy,
				$$
				where the intersection ranges over all systems $ ( y ) $ extending $ ( u ) $ to a {\RSP}~of $ R $ and such that their initial forms generate the directrix $ \Dir_\maxIdeal ( {\ideal} ) $. 
\end{itemize}
\end{Def}

In \cite{BerndThesis} the second author introduced  the characteristic polyhedron for idealistic exponents and used it to show that the invariant of Bierstone and Milman for resolution of singularities in characteristic zero can be purely determined by these polyhedra.	

\begin{Rk}
\label{Rk:Only_std_not_u-std}
In general, one has to consider the so called $ ( u ) $-standard bases (\cite{CJS} Definition 6.7).
However, one can deduce \rot{the following} easily from its definition: 
If $ (u,y) $ above is such that $ ( y ) $ determines the directrix of $ {\ideal} $ then any standard-basis is already a $ ( u ) $-standard basis.
Therefore we can avoid recalling the lengthy and technical definition of a $ ( u ) $-standard basis and work only with standard bases.
\end{Rk}

By introducing the procedure of vertex preparation Hironaka was able to prove the following result.
(Again, this is also valid in the general case, but we state it here in our special case).

\begin{Thm}
\label{Thm:Hironaka} Let $ ( R, \maxIdeal, \Resfield = R / \maxIdeal ) $ be a 
regular local ring with maximal ideal $ \maxIdeal $ and let
		 $ {\ideal} \subset R $ be a non-zero ideal and $ ( u, y ) $ be a {\RSP}~for $ R $ such that $ ( y ) $ yields the ideal generating the directrix of $ {\ideal} $.
		 
		 Then there exist a standard basis $ ( \hf ) = ( \hf_1, \ldots, \hf_m ) \in \hR^m$ of $ \widehat{\ideal} = {\ideal}\cdot \hR $ and a system of elements $ ( \hy ) = ( \hy_1, \ldots, \hy_r) $ in $ \hR$ such that $ ( u, \hy ) $ is a {\RSP}~of $ \hR $, $ ( \hy ) $ determines the directrix of $ {\ideal} $, 
		 \rot{$ (\hf; u; \hy) $ is well-prepared}
		 and
		$$
				\poly{\hf}{u}{\hy} = \cpoly {\ideal}u.
		$$ 
\end{Thm}

\rot{This result relies on \cite{HiroCharPoly} Theorem (4.8) (Theorem \ref{Thm:Hiro4.8} below).}
In order to recall \rot{the latter and} the construction of $ (\widehat  f , \widehat y ) $ of Theorem \ref{Thm:Hironaka} briefly we need to introduce some more notions.

\begin{Def}
\label{Def:initial_vertex}
	Let $ I \subset R $ be an ideal as before, $ (u,y) $ a {\RSP} for $ R $ such that $ (y) $ determines the directrix of $ I $ at $ \maxIdeal $.
	Let $ (f) = (f_1, \ldots, f_m  ) $ be a standard basis for $ I $ and $ v \in \poly fuy $ be a vertex of the polyhedron associated to $ (f,u,y) $.
	
	Set $ \nu_i := \ord_\maxIdeal (f_i) $, and consider finite expansions $ f_i = \sum C_{A,B,i} \, u^A y^B $ as in \eqref{eq:expansion}, with $ C_{A,B,i} \in R^\times \cup \{ 0 \} $, $ 1 \leq i \leq m $. 
	
	\begin{enumerate}
		\item[(1)]
	The {\em initial forms of $ (f) $ at the vertex $ v $}, is defined as
	$$ 
	 in_v (f) :=  in_v (f)_{(u,y)} := ( \, in_v(f_1), \ldots, in_v(f_m) \,) ,
	$$
	where, for $ 1 \leq i \leq m $,
	$$
	in_v (f_i) := in_v (f_i)_{(u,y)} := 
	F_i(Y) +  \sum_{(A,B) : \frac{A}{\nu_i - |B|} = v } \overline{C_{A,B,i}} \, U^A \,Y^B   
	,
	$$
	and $ F_i(Y) = in_\maxIdeal (f_i) $ is the initial form of $ f_i $ at $ M $, $ \overline{C_{A,B,i}} \in k = R/\maxIdeal $ denotes the image in the residue field, and $ U_i := in_M (u_i) $ resp.~$ Y_j := in_M(y_j) $. 
	
		\item[(2)]	
	The vertex $ v $ is called {\em solvable} if \rot{$ v \in \IZ^e_{\geq 0} $ and} there exist $ \lambda_1, \ldots, \lambda_r \in k $ 
	such that
	$$
		in_v (f_i) = F_i ( Y_1 + \lambda_1 U^v , \ldots, Y_r + \lambda_r U^v),
	$$ 
	for all $ i \in \{ 1, \ldots, m \} $.
	We sometimes also abbreviate the right hand side of the previous equation by $ F_i (Y + \lambda U^v) $.
	
		\item[(3)]
	For a homogeneous element $ G = \sum_{D} \mu_D Y^D  \in k[Y]=k[Y_1, \ldots, Y_r ] $ the {\em leading exponent of $ G $} is defined as $ {LE}(G) := \max_{lex} \{ D \in \IZ^r_{\geq 0} \mid \mu_D \neq 0 \} $, where the maximum is taken \wrt the lexicographical order on $ \IZ^r_{\geq 0} $. 
	
	For a homogeneous ideal $ I \subset k[Y] $, we set
	$$
		 E(I) :=  \{ {LE}(G) \mid G \in I \mbox{ homogeneous}   \}
	$$
	
	\item[(4)]
		Let $ P_{B,i}(U) \in k[U] $ such that 
		$ in_v (f_i) = \sum\limits_{B : |B|= \nu_i } \alpha_{B,i} Y^B  + \sum\limits_{B : |B|< \nu_i } P_{B,i} (U) \, Y^B $, for $ 1 \leq i \leq m $ ($ \alpha_{B,i} \in k $). 
		The system $ ( f ) $ is called {\em normalized at the vertex $ v $} if 
		\begin{equation}
		\label{eq:normalized_condition}
			 \alpha_{B,i} \equiv 0
			 \hspace{5pt}\mbox{ and }\hspace{5pt}
			 P_{B,i}(U) \equiv 0 \,,
			 \hspace{10pt}
			 \mbox{ for all }
			 B \in E( \langle F_1, \ldots, F_{i-1} \rangle ),
		\end{equation}
		with $ F_j = F_j (Y) $ the initial forms at $ M $ above.
	\end{enumerate}
	If $ (f) $ is normalized at $ v $ and $ v $ is not solvable, then $ ( f;u;y ) $ is called {\em prepared at $ v $}.
	If $ (f;u;y) $ is prepared at every vertex of $ \poly fuy $ then $ (f;u;y) $ is called {\em well-prepared.}
\end{Def}


From the definition, it is immediate that we have
\begin{equation}
\label{eq:initial_vertex_with_U^v}
	in_v (f_i) = 	F_i (Y) + \sum_{B : |B|< \nu_i } \lambda_{B,i}  \, U^{(\nu_i -\vert B\vert)\cdot v} \, Y^B ,
\end{equation}
for certain certain coefficients $ \lambda_{B,i} \in k $ that are non-zero only if $ (\nu-\vert B\vert)\cdot v \in \IZ^e_{\geq 0} $.
\rot{Note that the above notations can be defined for any point $ v \in \poly fuy $ that is contained in a face of the polyhedron.}

\smallskip 

\rot{Now, we are able to recall (for our special case)}

\begin{Thm}[\cite{HiroCharPoly} Theorem (4.8)]
	\label{Thm:Hiro4.8} 
	\rot{Let $ ( R, \maxIdeal, \Resfield = R / \maxIdeal ) $ be a 
	regular local ring with maximal ideal $ \maxIdeal $ and let
	$ {\ideal} \subset R $ be a non-zero ideal and $ ( u, y ) $ be a {\RSP}~for $ R $ such that $ ( y ) $ yields the ideal generating the directrix of $ {\ideal} $.
	Let $ (f) = (f_1, \ldots, f_m ) $ be a standard basis for $ I $.}
	
	\rot{If $ v $ is any vertex of $ \poly fuy $ such that $ (f;u;y) $ is prepared at $ v $,
		then $ v $ is also a vertex of $ \cpoly Iu $.
		In particular, if $ (f;u;y) $ is well-prepared,
		then}
		$$
			\rot{\poly{f}{u}{y} = \cpoly {\ideal}u.}
		$$
\end{Thm}

\rot{In fact, Hironaka's result is slightly more general:
	The hypothesis on $ ( y ) $ is not on the directrix of $ I  \subset R $, but on the directrix of $ I \cdot R/\langle u \rangle \subset R/ \langle u \rangle $.
	(Further, recall Remark \ref{Rk:Only_std_not_u-std}).
	The tentative reader may observe that in the implication for the equality, Hironaka originally requires $ (f;u;y) $ to be {\em totally prepared}.
	The latter means that $ (f;u;y) $ is well-prepared and normalized along all bounded faces (see e.g.~\cite{CJS} Definition 7.15(5)).
	Since the normalization at a point, that is not a vertex, does not change the polyhedron, this stronger assumption is not needed.}

\smallskip

\rot{In the following remark, we explain how to achieve the desired elements $ (\hf; \hy) $ of Theorem \ref{Thm:Hironaka}.}

\begin{Rk}[\em Vertex preparation]
Suppose a given standard basis $ ( f ) $ of $ I $ is not normalized at some vertex $ v $.
Then some exponent $ B $ of a monomial $ U^A Y^B $ appearing in $ in_v( f_i) $ with non-zero coefficient coincides with the leading exponent of an element in the ideal generated by the previous elements $ (f_1, \ldots, f_{i-1} ) $, say $ h = a_1 f_1 + \ldots + a_{i-1} f_{i-1} $, $ a_j \in R $.
If we replace $ f_i $ by $ f_i' := f_i - C_{A,B,i} u^A h $ then the monomial $ U^A Y^B $ does not appear in $ in_v(f_i' ) $ with non-zero coefficient. 
In fact, in \cite{HiroCharPoly} Lemma (3.15) Hironaka proves that there exist $ a_{ij} \in R $ such that, for $ g_i := f_i - \sum_{j=1}^{i-1} a_{ij} f_j $, the system
$(g) =  ( g_1, \ldots, g_m ) $ is normalized at $ v $,
$ \poly guy \subset \poly fuy $, and 
any vertex of $ \poly fuy $ other than $ v $ is also a vertex of $ \poly guy $.
\rot{Note that, in contrast to the process of solving vertices, we do not necessarily eliminate $ v $ when normalizing, i.e., it is possible that $ v $ is a vertex of $ \poly guy $.}

If $ v $ is solvable then we must have $ v \in \IZ^e_{\geq 0} $ and at least one of the coefficients $ \lambda_i \neq 0  $ has to be non-zero.
By choosing any lift $ a_i \in R $ of $ \lambda_i \in R/\maxIdeal $ and replacing $ ( y ) $ by 
$$
	(z ) = (z_1, \ldots, z_r) := ( y_1 + a_1 u^v , \ldots, y_r + a_r u^v),
$$
we eliminate the vertex $ v $ in the polyhedron, i.e., $ v \notin \poly fuz $.
In fact, in \cite{HiroCharPoly} Lemma (3.10) it is shown that $ \poly fuz \subset \poly fuy $, $ v \notin \poly fuz $, and any vertex different from $ v $ is also a vertex of $ \poly fuz $.

Suppose $ ( f ) $ and $ ( y ) $ are given.
For the vertex preparation process we need to equip $ \IR^e_{\geq 0 } $ with a total order, e.g.~the one given by the lexicographical order of the entries of $ v = ( v_1 , \ldots , v_e ) \in \IR^e_{\geq 0} $.
Let $ v \in \poly fuy $ be the minimal vertex of $ \poly fuy $ \wrt the chosen total order.
First, we change $ (f) $ such that \rot{it} is normalized at $ v $.
If $ v $ is still a vertex after this then we 
\rot{follow the following procedure.}
\rot{If $ v $ is solvable,} 
we eliminate it as described before and consider the minimal vertex of the resulting polyhedron.
If $ v $ is not solvable, then we consider the next smallest vertex of the associated polyhedron and repeat the previous.

Unfortunately, this process is not necessarily finite.
For example, consider the hypersurface $ f = y^4 + y^2 + u_1^6 + u_2^5 $ over a field of characteristic two.
The vertex $ v = (3,0) $ of $ \poly fuy $ is solvable since $ in_v(f) = (Y + U_1^3)^2 $, but if we replace $ y $ by $ z_1 := y + u_1^3 $ we obtain $ z_1^4 + z_1^2 + u_1^{12} + u_2^5 $ and the vertex $ (6,0) $ is {also} solvable.
Thus the process is not finite in general.

Nevertheless the vertex preparation process yields elements $ (\widehat{f}, \widehat{y} ) $
 in $ \widehat{R} $ such that $ (\widehat{f}, u, \widehat{y} ) $ is well-prepared.
Then the proof of \cite{HiroCharPoly} Theorem (4.8) shows that the corresponding associated polyhedron coincides with the characteristic polyhedron of $ I $, i.e., $ \poly{\hf}{u}{\hy} = \cpoly {\ideal}u $.
Further, $ ( \widehat f ) $ is a standard basis for $ I\widehat{R} $ by \cite{HiroCharPoly} Corollary (3.17.4).
\end{Rk}

In \cite{CPcompl} it is shown that, in the case $ R $ a local $ G $-ring, $ m = 1 $, and $ r = 1 $ one can find the system $(\hy)$ in $ R $ such that the characteristic polyhedron is achieved $ \poly fuy = \cpoly {\ideal}u $, $ {\ideal} = \langle f \rangle $.

For our aim to prove the invariance of the numbers that we deduce from the characteristic polyhedron the following result is sufficient.
 
 \begin{Prop}\label{polyvide} 
 Let $I\subset R$ be an ideal of an {\em excellent} regular local ring $ (R,M) $ 
 such that $ \dim (R / I) \leq 2$.
 Let $\cX:=V(I)\subset \Spec (R) $ and $ x $ the closed point. 
 Suppose  that, with the notations of Theorem \ref{Thm:Hironaka} 
 $$
 	\Delta(I;u_1,\ldots,u_e)=\emptyset,\ e:=\dim( \Dir_M( I ) ).
 $$
 Then there exists a sequence of parameters $(y):=(y_1,\ldots,y_r)$ in $R$ such that:
\begin{enumerate}
		\item[(i)]	V$(y)\subset V(I) $.
		\item[(ii)]	V$(y)$ is the Hilbert-Samuel stratum of $\cX = V(I) $ in a neighborhood of $x$ and V$(y)$ is permissible for $\cX$ at $x$.
					Moreover, if ${\cX}'\longrightarrow {\cX}$ denotes the blowing up along V$(y)$, then there is no point $x'\in {\cX}'$ near to $x$.
		\item[(iii)]		For any system $(\widehat{z}) := (\widehat z_1,\ldots,\widehat z_r)$ in the completion $\widehat{R}$ such
 that 
 $$	
 	\Delta(I \hR;u_1,\ldots,u_e;\widehat z_1,\ldots,\widehat z_r) = 
 	\Delta(I;u_1,\ldots,u_e) = \emptyset,
 $$ 
 we have:
 $$
 	\langle y \rangle \cdot \widehat{R} = \langle \widehat z \rangle .
 $$
\end{enumerate}	 
\end{Prop}
 
 \begin{proof}
 First, we show (iii) and explain the construction of $(y)$. 
	Let $ ( \widehat z ) $ be a system of elements in $ \widehat{R} $ as in the statement.
	Note: In particular, $ ( \widehat{z} ) $ determines the directrix of V$(I\widehat{R})=:\widehat{{\cX}}$ at $ x $.
	
	By Theorem \ref{Thm:2.2.(2)_CJS}(4), 
	V$(\widehat z)$ is permissible for $\widehat{{\cX}}$ at the origin $ x $.
	Hence it is contained in the Hilbert-Samuel stratum of $ \widehat{{\cX}}$.  
	Let $\widehat{\pi}:\widehat{{\cX}'}\longrightarrow \widehat{{\cX}}$ be the blow up of $\widehat{{\cX}}$ along V$(\widehat{z})$.
	By Theorem \ref{Thm:HiroMizutani}, 
	there is no point $x'\in \widehat{{\cX}'}$ near to $x$.
	(At this point we need that $ \dim ( R/I) \leq 2 $, because otherwise the assumptions of Theorem \ref{Thm:HiroMizutani} do not necessarily hold). 
	
	As a conclusion, V$(\widehat z)$ must be the whole Hilbert-Samuel stratum of $\widehat{{\cX}}$:
	Suppose there is another component $ Y $ with $ x \in Y $.
	Then the intersection of V$(\widehat z)$ and $ Y $ is non-empty. 
	The strict transform of $ Y $ under $ \widehat{\pi} $ is still contained in the Hilbert-Samuel locus and must have non-trivial intersection with the exceptional divisor. 
	Therefore a point $ x' $ contained in the intersection is near point to $ \widehat{\pi}(x') \in V(\widehat{z}) $. 
	This is a contradiction.
 	
 	The excellence of $ R $ implies that the Hilbert-Samuel stratum in the completion is the pull back of the Hilbert-Samuel stratum at $x$, see \cite{CJS} Lemma~1.37. 
 	So we can choose a sequence of parameters $ (y):=(y_1,\ldots,y_r) $ in $R$ such that $\langle y \rangle \cdot\widehat{R}=\langle \widehat z \rangle $.
 	This shows (iii).
 	Moreover, V$(y)$ is the Hilbert-Samuel stratum of $ \cX = V(I)$ in a neighborhood of $x$. 
 	Thus (i) and (ii) hold.
 \end{proof}

 \begin{Rk}\label{Rk:polyvide} 
 In case $e = \dim (R/I)$, under the hypotheses of Proposition~\ref{polyvide}, part (i) gives: 
 V$(y)= $V$(I)$ which is a regular scheme.
 \end{Rk}

 \begin{Prop}\label{Cor:blowuppolyvide} 
 Let $I\subset R$ be an ideal in an excellent regular local ring $R$  
 such that $ \dim (R / I) \leq 2$. 	
 Set $ \cX:=$V$(I)\subset \Spec(R) $ and let $x$ be the closed point. 
 
 Then, with the notations of Theorem \ref{Thm:Hironaka}, 
 there exists a sequence of parameters $(y):=(y_1,\ldots,y_r)$ in $R$ and a standard basis $(f):=(f_1,\ldots,f_m) $ of $I$ in $ R $ such that 
 \rot{$ (f;u;y) $ is well-prepared and}
 $$\Delta(I;u)=\Delta(I;u;y)=\Delta(f;u;y).$$
 \end{Prop}

\begin{Rk} 
	\rot{Once we have the equality $ \poly guy = \cpoly Iu $, for some $ ( y ) $ as above and some standard basis $ ( g ) = (g_1, \ldots, g_m) $ of $ I $, none of the vertices is solvable.
		But $ (g; u; y) $ is not necessarily normalized at every vertex.
		After normalizing finitely many vertices, we obtain the desired well-prepared $ (f;u;y) $.}
\end{Rk}

\begin{proof}[Proof of Proposition \ref{Cor:blowuppolyvide}]
 In the case $e=0$, there is no $(u)$.
 Hence, for every $(y)$ and any standard basis $ ( f ) $ of $ I $, the polyhedron $\Delta(f;\emptyset;y)$ is empty, by definition.
 Thus there is nothing to prove in this extreme case.
 \rot{Note that $ ( f) $ is automatically well-prepared since there is no vertex.} 

 If $e=1$ then $(u)=(u_1)$.
 Let $ ( y) $ any system of elements in $ R $ extending $ (u_1) $ to a {\RSP} for $ R $. 
 If $ \Delta(I;u_1;y) = \emptyset $ then we are done.
 
 Let $ (f) $ be any standard basis for $ I $. 
 Suppose $\Delta(f;u_1;y)=\delta(f;u_1;y)+\R_{\geq 0}$, for some $ \delta(f;u_1;y) \in \IQ $. 
 Either Hironaka's process to construct the characteristic polyhedron (which is either a semi-line or $\emptyset$ in this case) is finite and we get the result, or Hironaka's process is infinite.
 But in the latter case we get $\Delta(I,u_1)=\emptyset$ and we can apply Proposition~\ref{polyvide}. 
 	
 For the case $e=2$, we refer to \cite{CSCcompl} Theorem~A.	
 \end{proof}

The following is an important invariant coming from the characteristic polyhedron and 
one of the key ingredients for the study of singularities.
 
\begin{Def}
\label{Def:delta}
	With notations of Theorem \ref{Thm:Hironaka}, we define
	$$\delta(I;u):=\mathrm{inf}\{\vert v \vert; v\in \Delta(I;u)\} \in \IQ_{\geq 0} \cup \{ \infty \} =:\IQ_\infty.$$
	
	Let $ ( f ) = (f_1, \ldots, f_m ) $ be generators of $ I $ and 	
	\rot{$ (w,z) = (w_1, \ldots, w_d, z_1, \ldots, z_s ) $ be any regular system of parameters for $ R $ such that $ f_i \notin \langle z \rangle $, for all $ i \in \{1, \ldots, m\} $. 
	Then $ \poly fwz $ is defined (Definition \ref{Def:Poly}(2)) and we introduce} 
	$$
		\rot{\delta(f;w;z):= \inf\{\vert v \vert; v\in \Delta(f;w;z)\} \in \IQ_\infty.}
	$$		
\end{Def}

\rot{Clearly, we have $ \delta(f;u;y) \leq \delta(I;u) $,
	if $ ( y) = (y_1, \ldots, y_r ) $ are elements in $ R $ extending $ ( u) $ to a {\RSP} for $ R$.
	Furthermore, we always have $ \delta(I;u) > 1 $ since we assume that $ ( y ) $ yields the ideal generating the directrix of $ I $.
	In contrast to this, we may have $ \delta(f;u;y) \leq 1 $, as the following example shows:}

\begin{Ex}
	\rot{Suppose $ (u_1, u_2, y_1, y_2) $ is a {\RSP} for $ R $.
		Consider
		$$
		I := \langle f_1, f_2 \rangle \subset R, \ \ \
		\mbox{ for } \
			f_1 := y_1^a + u_1^b,
			\ \
			f_2 := y_2^c + u_2^d,
		$$
		where $ a, b, c, d \in \IZ_+ $ with $ a < \min\{ b, c \} $ and $ c < d $. 
		Then $ (y_1, y_2) $ determines the directrix of $ I $, $ \poly fuy = \cpoly Iu $
		and $ \delta(I;u) = \min\{ \frac{b}{a}, \frac{d}{c} \} > 1 $.
		Let
		$$ 
			g_1 := f_1, \ \ \
			g_2 := f_2 + u_1 f_1 = y_2^c  + u_2^d + u_1 y_1^a + u_1^{b+1}.
		$$
		Then $ (g) = (g_1, g_2) $ are generators for $ I $ and
		$$
			\delta(g;u;y) =
			\min 
			\left\{ 
				\frac{b}{a}, \frac{d}{c}, \frac{1}{c-a}, 
				\frac{b+1}{c}
			\right\}
			\leq \frac{1}{c-a} \leq 1.
		$$
		}
\end{Ex}

In fact, Theorem \ref{Thm:Hironaka} provides already some useful information on $ \delta(I;u) $.
Namely, we get the following corollary from it.

\begin{Cor}
\label{Cor:delta_in_1/N_Z^e}
	Let $ I \subset R $ be a non-zero ideal in a regular local ring $ (R, M) $
	and let $ (u,y) $ be a {\RSP} for $ R $ such that $ (y) $ determines the directrix of $ I $ at $ M $.
	Suppose $ ( \hf ) = ( \hf_1, \ldots, \hf_m ) $ is the standard basis for $ \widehat{I} = I \cdot \hR $ and $ (\hy ) $ is the system of regular elements in $ \hR $ of the statement of Theorem \ref{Thm:Hironaka}.
	In particular, $ \poly{\hf}{u}{\hy} = \cpoly {\ideal}u $.
	
	Suppose 	$ \cpoly {\ideal}u \neq \emptyset $.
	If we set $ N := \ord_\maxIdeal (\hf_m) $, then
	$$
		\delta (I; u ) \in \frac{1}{N!} \cdot \IZ_{\geq 0}.
	$$ 
\end{Cor}

\begin{proof}
	Since $ \poly{\hf}{u}{\hy} = \cpoly {\ideal}u $ we have $ \delta(\hf;u;\hy) = \delta (\ideal;u) $.
	Further, the assumption $ \cpoly Iu \neq \emptyset $ implies $ \delta(I;u) < \infty $.
	 
	The intersection $ \poly{\hf}u{\hy} \cap \{ v \in \IR^e_{\geq 0 } \mid  |v| = \delta(\hf;u;\hy) \}$ is non-empty and defines a face of the polyhedron.
	In particular, it contains at least one vertex of $ \poly{\hf}{u}{\hy} $.
	The vertices of $ \poly{\hf}{u}{\hy} $ are of the form $ \frac{A}{\nu_i - |B|} $, for $ A \in \IZ^e_{\geq 0 }$, $ B \in \IZ^r_{\geq 0 } $, $ |B| \leq \nu_i $, and $ \nu_i := \ord_M (\hf_i ) $, $ 1 \leq i \leq m $.
	Since $ ( \hf ) $ is a standard basis we have $ \nu_1 \leq \ldots \leq \nu_m = N $.
	This implies the assertion.	
\end{proof}


The transformation laws of the characteristic polyhedron are difficult to apprehend. 
This is one of the major difficulties of mixed or positive characteristics.
In Proposition \ref{Prop:originblowup} we study the easy case when we look at the origin of a chart after a permissible blowing up.
But first let us fix some notations. 

\begin{Setup}
\label{Setup:blowup}
 Let $(R, \maxIdeal , k ) $ be an excellent regular local ring and ${\ideal}\subset R$ an ideal.
 Suppose there exists a standard basis $( f ) = ( f_1, \ldots, f_m ) $ of $ {\ideal}  $ and a system of elements $ ( y ) = ( y_1, \ldots, y_r) $ in $ R $ such that 
 \begin{itemize}
 	\item	$ ( u, y ) :=(u_1,\ldots,u_e,y_1, \ldots, y_r) $ is a {\RSP}~of $ R $, 
 	\item	$ ( y ) $ determines the directrix of $ {\ideal} $, and 
 	\item	$( f\rot{; u}; y )$ is well prepared (Definition \ref{Def:initial_vertex}).
 \end{itemize}
 \rot{By Theorem~\ref{Thm:Hiro4.8}, the last property implies 
 $
 \poly{f}{u}{y} = \cpoly {\ideal}u
 $.}
	\rot{For $1\leq i \leq m$, we set  $ \nu_i:=\ord_\maxIdeal(f_i)$.}
 
 Assume that, $D:=$V$(u_1,\dots,u_d,y)$, $1\leq d \leq e$, is a permissible center for $ \mathcal X := \Spec (R/ \ideal) $ at the origin $ x $.
 Let $ \mathcal X' \longrightarrow \mathcal X $ be the blowing up along $D$ and $x'$ be the point in $ \mathcal X'$ of parameters 
 $$
 	(u',y'):=\left(u_1,\frac{u_2}{u_1},\ldots,\frac{u_d}{u_1},u_{d+1},\ldots,u_e, \frac{y}{u_1}\right).
 $$ 
 Further, we set $( f' ) := \left( \dfrac{f_1}{u_1^{\nu_1}}, \ldots, \dfrac{f_m}{u_1^{\nu_m}} \right) $ and we denote the strict transform of $ I $ by $ I' $, 
$$
	I'= \langle f'  \rangle \subset R\left[\frac{u_2}{u_1},\ldots,\frac{u_d}{u_1}, \frac{y}{u_1}\right]_{\langle u' , y' \rangle }.
$$ 
\end{Setup}

 Note: The assumption on the existence of such elements $ (f, y ) $ is always true if $ R $ is complete (Theorem \ref{Thm:Hironaka}), or if $ \dim (R/J) \leq 2 $ (Proposition \ref{Cor:blowuppolyvide}).

\begin{Prop}
\label{Prop:originblowup}
	Let the situation be as in Setup \ref{Setup:blowup}.
	{\em Assume that $x'$ is very near to $x$}.
 	\begin{enumerate}
 	\item[(1)] 	 Then  $( f' ) $ is a standard basis of the strict transform $ I' $ of $ I $ (at $ x' $) and $( f', \rot{u'} ,y')$  is well-prepared. 
 	In particular,
 	$$
 		\poly{f'}{u'}{y'} = \Delta (I'; u').
 	$$
 	\item[(2)]
 	Suppose the center of the blow up is the origin (i.e., $ d = e $ in Setup \ref{Setup:blowup}) and that $ e = 1 $.
 	Then we have
 	$$
 		\delta(I';u') = \delta (I;u) - 1 .
 	$$ 
\end{enumerate} 	
 \end{Prop}

\begin{Rk}[\cite{CJS} Proposition 8.6]
\label{Rk:delta_and_(very)near}
	It is worth to point out that:
	\begin{enumerate}
		\item[(1)]	If $ \delta(f',u',y') \geq 1 $ then $ x' $ is near to $ x $.
		\item[(2)]	If $ \delta(f',u',y') > 1 $ then $ x' $ is very near to $ x $.
	\end{enumerate}
	Moreover, the converse of both implications is true if $ (f',u',y') $ is prepared at any vertex contained in $ \{v \in \IR^e_{\geq 0} \mid |v| \leq 1 \} $.
	
	In fact, \cite{CJS} show a more general result.
	Namely, they also consider points different from the origin of the $ U_1 $-chart.
	But then, one has to add in (2) the extra assumption $ e_{\cX}(x)=e_{\cX}(x)_{k(x')} $ (which is clearly true if $ x' $ is the origin of the $ U_1 $-chart). 
\end{Rk} 
 
 \begin{proof}[Proof of Proposition \ref{Prop:originblowup}]
 Theorem 8.1 in \cite{CJS} provides that $ (f') $ is a standard basis for $ I' $.
 (Remember also Remark \ref{Rk:Only_std_not_u-std}).
 Let us determine $ \poly{f'}{u'}{y'} $:
Let $ h \in \{ f_1, \ldots, f_m \} $ and set $ n = ord_\maxIdeal ( h ) $.
Consider an expansion as in \eqref{eq:expansion},
$$
	h = \sum_{A,B} C_{A,B} \, u^A \, y^B.
$$
In the coordinates of $ x '$, 
$ (u',y'):=(u_1,\frac{u_2}{u_1},\ldots,\frac{u_d}{u_1},u_{d+1},\ldots,u_e, \frac{y}{u_1}) $, 
the strict transform of $ h $ is given by
$$
	h' = \sum_{(A,B) } C_{A,B} \, (u_1')^{ A_1 + \ldots+  A_d + |B| - n } \,  (u_2')^{A_2} \ldots (u_d')^{A_d} \, (u_{d+1}')^{A_{d+1}} \ldots  (u_e')^{A_e}\, (y')^B.
$$
We observe that
\begin{equation}
\label{eq:Exponent_u_1_after_blowuporigin}
	\frac{A_1 + \ldots + A_d + |B| - n}{n - |B|} = \frac{A_1 + \ldots + A_d }{n - |B|} - 1.
\end{equation} 
 
 Let $v':=(v'_1,\ldots,v'_e)$ be a vertex of $  \poly{f'}{u'}{y'}$.
 The previous computation shows that there is a vertex $v=(v_1,\ldots,v_e)\in \poly{f}{u}{y}$ such that
 \begin{equation} 
 \label{eq:pointafterblowup}
 	v'_1=v_1+\ldots + v_d - 1,\ v'_i=v_i,\ 2\leq i \leq e.
\end{equation}
 By \eqref{eq:initial_vertex_with_U^v}, we have
 $in_{v}(f_i):=F_i(Y)+\sum_{B, \vert B\vert< \nu_i} \lambda_{B,i} \, U^{(\nu_i - \vert B\vert)v}\, Y^B$, where   $\lambda_{B,i} \in k(x)=k(x')$ and $1\leq i \leq m$.
 We get 
 \begin{equation}
 \label{eq:in_v'_easycompute}
   in_{v'}(f'_i)= F_i(Y')+\sum_{B}  \lambda_{B,i} \, U^{(\nu_i-\vert B\vert)v'}\,{Y'}^B.
\end{equation}
 As $(f)$ is normalized at $ v $ (Definition \ref{Def:initial_vertex}(4)), $(f')$ is normalized at the vertex $ v' $. 
 
 Furthermore: $v\in \IZ_{\geq 0}^e\Leftrightarrow v'\in \IZ_{\geq 0}^e$.
 In this case, $v$ is not solvable (Definition \ref{Def:initial_vertex}(2)), i.e., there does not exist $ (\lambda) = (\lambda_1, \ldots, \lambda_r) \in k(x)^r=k(x')^r $ such that $in_{v}(f_i) = F_i(Y+\lambda U^v)$, for all $ 1 \leq i \leq m $.
 Then \eqref{eq:in_v'_easycompute} implies that $v'$ is also not solvable.
 
 As a conclusion we get that $ (f',u',y') $ is prepared at $ v '$.
 This implies the first part of the proposition.
 
 For (2), suppose $ d = e = 1 $. 
 Using $ \poly{f'}{u'}{y'} = \Delta(I'; u') $, \eqref{eq:Exponent_u_1_after_blowuporigin}, and $ \poly{f}{u}{y} = \Delta(I; u) $, we obtain
$$ 
	\delta(I';u') = \delta ( f';  u ; z' ) =  \delta ( f;  u ; y ) - 1 = \delta(I;u) - 1.
$$
\end{proof}


In the remaining part of this section we study the information on the singularity $ R/I $ that we can obtain from the characteristic polyhedron.
In Definition~\ref{Def:delta}, we have introduced
$  \delta(I;u) =\mathrm{inf}\{\vert v \vert; v\in \Delta(I;u)\} \in \IQ_{ \infty } $.

Note that $ \delta (  {\ideal};u ) = \infty $, implies $ \cpoly{\ideal}u = \emptyset $.
Hence we already have enough information in order to improve the singularity if $ \dim (R/I) \leq 2 $, see Proposition~\ref{polyvide}.
For the case $ \delta ( {\ideal};u ) < \infty $ we need the following notion in order to state the theorem:

\begin{Def}
\label{Def:In_delta}
Let $ {\ideal} \subset R $ be a non-zero ideal in an excellent 
regular local ring $ R $ (not necessarily complete) 
and let $ ( u, y ) = ( u_1, \ldots, u_e; y_1, \ldots, y_r ) $ be a {\RSP}~for $ R $ 
such that $ ( y ) $ yields the directrix of $ {\ideal} $. 

Set $ \delta :=  \delta ( {\ideal}; u  ) > 1  $. 
Assume $\delta \neq \infty$.
For $ \lambda \in \IQ _{\geq 0} $ we define 
$$
	\cI_\lambda := \left\langle \, u^A y^B \;\Big|\; |B| + \frac{|A|}{\delta} \geq \lambda \, \right\rangle,
	\hspace{10pt}
	\cI_\lambda^+ := \left\langle \, u^A y^B \;\Big|\; |B| + \frac{|A|}{\delta} > \lambda \, \right\rangle,
$$ 
which yields a graded ring which we denote by $ gr_\delta ( R ) $,
$$ 
gr_\delta ( R ):=\bigoplus_{\lambda \in \IQ^+} {\cI_\lambda  / \cI_\lambda^+}.
$$
(Note that $ \delta $ is a positive number).
The quotients $ {\cI_\lambda  / \cI_\lambda^+}$ are zero except for a discrete set of $\lambda$.

In fact, we can define a ``monomial valuation'' $v_{\delta}$ on $R$ by 
$$
	v_{\delta}(u^A y^B):=     |B| + \frac{|A|}{\delta} ,
$$ 
and, for $ 0\not=g = \sum C_{A,B} \, u^A  y^B $ (finite sum with $ C_{A,B} \in R^\times \cup \{ 0\} $, as in \eqref{eq:expansion}),
$$
	v_{\delta}(g):=\mathrm{inf}\left\{|B| + \frac{|A|}{\delta} \mathrel{\Big|} C_{A,B} \in R^\times \right\}.
$$
Then $ gr_\delta ( R )$ is the graded ring associated to the filtration defined by $v_{\delta}$.

We denote the initial of $ u_i $ and $ y_j $ by the corresponding capital letters.
If we put weight $ 1 $ on the variables $ Y $ and $ \frac{1}{\delta} $ on the variables $ U $ then we have an isomorphism of quasi-homogeneous rings:
$$
	gr_\delta ( R ) \cong \Resfield [ U, Y],
$$
where $ \Resfield = R / \maxIdeal $ is the residue field of $ R $.
Furthermore, we define, for an element $ 0\not=g = \sum C_{A,B} \, u^A  y^B \in R $ (as above) 
\begin{equation}
\label{eq:in_delta}
 	in_\delta ( g ) := \sum\limits_{ |B| + \frac{|A|}{\delta} \, =\, v_\delta ( g ) } \overline{C_{A,B}} \; U^A \,Y^B, 
\end{equation}	 	
$$ 
	 	In_\delta ( \ideal ) := \langle \, in_\delta ( g ) \mid g \in \ideal \, \rangle,	
$$
where the sum ranges over those $ ( A,B ) \in \IZ^e_{\geq 0} \times \IZ^r_{\geq 0}$ fulfilling $ |B| + \frac{|A|}{\delta} = v_\delta ( g ) $ and $ \overline{C_{A,B}} = C_{A,B} \mod M \in \Resfield $.
\end{Def}

\begin{Thm}
\label{Thm:alphabetainvariant_1_deltaonly}
Let $ {\ideal} \subset R $ be a non-zero ideal in an excellent regular local ring $ R $ with $ \dim (R/I) \leq 2 $.
Let $ ( u, y ) = ( u_1, \ldots, u_e ; y_1, \ldots, y_r ) $ be a {\RSP}~for $ R $ such that $ ( y ) $ yields the directrix of $ {\ideal} $ at $ M $. 

Then the number $ \delta (  {\ideal}; u  ) \in \  ] 1,\infty]  $ is an invariant of the singularity $ R / {\ideal} $, i.e., the only datum of $ R / {\ideal} $ is enough to compute $ \delta ( \ideal; u ) $.
The same is true for the quasi-homogeneous ring $ gr_\delta ( R ) /  In_\delta ( \ideal )$ when $\Delta ( {\ideal}, u )\not= \emptyset$, $ \delta := \delta (  {\ideal}; u  ) $.
	%
%
%
In particular, these notions are independent of the embedding.
\end{Thm}

\begin{Obs} \label{Obs:nocompletion}
The invariance of $ \delta := \delta ( {\ideal};u ) $ can be shown by using Hironaka's trick:  
add a variable $T$ and define $\delta$ with the only datum ${(R/J)}[T]$.
This shows that $\delta$ may be defined by ${R/J}$ only, 
i.e., $\delta$ does not depend on the embedding. 
This 
\rot{argument is carried out}
by the first author \rot{in} \cite{CosRevista}.
Unfortunately, \cite{CosRevista} is written in the case over a field of characteristic zero; 
though this hypothesis is unnecessary. 
In fact, the characteristic zero assumption implies the good behavior of the ``distinguished data'' under permissible blowing ups and 
\rot{avoids the process of} 
Hironaka's preparation after each blowing~up.
This is mainly used in the proof of \cite{CosRevista}~Lemma~B.2.7.

In the general case, without any assumption on $ \dim(R/I)$, 
Proposition~\ref{Prop:originblowup} is enough for the proof of the invariance of $ \delta $ to go through
(after eventually replacing $R$ by its completion). 
Indeed, in the process described in \cite{CosRevista}, all the points considered after some permissible blowing ups are at the origin of a chart. 
By Proposition \ref{Cor:blowuppolyvide}, the assumption $ \dim(R/I)\leq 2$ 
implies that we can find a standard basis $(f_1,\dots,f_m)$ of $I$ and $(y)$ a sequence of parameters in ${R}$ such that $(f,\rot{u},y)$ is well-prepared: 
$\Delta(f;u;y) =\Delta(I;u)$. That is all we need to make the computations of \cite{CosRevista}. 
\end{Obs}


 \begin{proof}[Proof of Theorem \ref{Thm:alphabetainvariant_1_deltaonly}]
Let us recall Hironaka's trick and the main arguments of  \cite{CosRevista}.
We introduce a new variable, say $ T $, and set $ \cZ_0 := \Spec ((R/I)[T]) $.
Let $ L_0 $  be the line of ideal $M_{R/I}\cdot (R/I)[T]\subset (R/I)[T]$.
First, we blow up the closed point $ z_0 \in \cZ_0 $ corresponding to the maximal ideal $ M_{R/I} + \langle T \rangle $.
Let $ \cZ_1 $ be the affine $ T $-chart of the blow-up and $ z_1 \in \cZ_1 $ its origin.
We observe that $ z_1 $ is the unique closed point given by the intersection of the strict transform of the line $ L_0 $ and the exceptional divisor of the last blowing up.
Note that $ z_1 $ is very near to $ z_0 $.

We repeat this process $ l $ times for some $ l \in \IZ_+ $; 
i.e., in the next step we blow up with center $ z_1 $, consider the affine chart $ \cZ_2 $ complementary to the strict transform of the preceding exceptional divisor (= the affine $T$-chart), and let $ z_2 \in \cZ_2 $ be the intersection of the strict transform of the line $ L_0 $ and the exceptional divisor of the last blowing up.
Hence we obtain a sequence of permissible blowing ups
$$
	\cZ_l  \longrightarrow \cZ_{l-1} \longrightarrow \ldots \longrightarrow \cZ_{1}   \longrightarrow \cZ_0
	\eqno(S1)
$$ 
and a sequence of closed points $ z_0, z_1, \ldots, z_l $, where $ z_j \in \cZ_j $ is very near to $ z_{j-1} $, $ 1 \leq j \leq l $.
In the exceptional locus of $\mathcal{Z}_l$, there appears the closed set $D$ corresponding to  $\IP(\Dir _{z_{l-1}}  (\mathcal{Z}_{l-1} ))$, see Theorem \ref{Thm:HiroMizutani}.
(For more details, see the canonical isomorphism at the beginning of the proof of \cite{CJS} Theorem~2.14, or \cite{GiraudEtude}~1.1.2, p.~II-1).  

Let us do the computations: 
we take any embedding of $R/I$.
By Proposition~\ref{Cor:blowuppolyvide}, there exist %
a {\RSP} $(u,y)$ as in the statement and a standard basis $ (f) $ of $ {\ideal} $ such that $ (f, \rot{u}, y ) $ is well-prepared (Definition \ref{Def:initial_vertex}).
In particular, $ \poly fuy = \cpoly Iu $.
By computing 
the transformation under the blow-ups (S1) and using  Proposition~\ref{Prop:originblowup} ($z_i$ is  the  origin of the $ T $-chart), we obtain $D= \mathrm{V} ( \frac{y}{T^l}, T ) $.
Further, $ D $ becomes permissible if $l$ is large enough 
(in fact, if $l\cdot (\delta ( {\ideal};u )-1)\geq 1$). 
Let $\eta$ be the generic point of $D$.

We blow up with center $ D $,
$$
 \pi_{l,1}:\mathcal{Z}_{l,1}  \longrightarrow \mathcal{Z}_{l,0}:=\mathcal{Z}_l  ,
$$
where $\mathcal{Z}_{l,1}$,  is the affine chart complementary of the strict transform of the preceding blowing-up (affine $T$-chart).
We define $z_{l,1}\in \mathcal{Z}_{l,1}$ as the point corresponding to $\IP(\Dir_{z_l}(\mathcal{Z}_l )/T_{z_l}(D))$ by the canonical isomorphism. 
(Again $ z_{1,l} $ is the unique closed point which is the intersection of the strict transform of the line $ L_0 $ and the exceptional divisor of the last blowing up).
It \rot{turns out} that in  $\mathcal{Z}_{l,1}  $  there is the point $\eta_1$ corresponding to   $\IP(\Dir _{\eta}  (\mathcal{Z}_{\eta} ))$ by the canonical isomorphism.
Set $ D_1 := \overline{\eta_1}$.

If  the  restriction map  $ \pi_{l,1} |_{D_1} : D_1\longrightarrow D $ is an isomorphism and $D_1$ is permissible at $z_{l,1}$, we blow  up along $D_1$.
Otherwise we stop.

We repeat this process and obtain after $ i \geq 1 $ steps a sequence of blowing ups
$$
	\cZ_{l,i}  \longrightarrow  \ldots \longrightarrow \cZ_{l,1}  \longrightarrow \cZ_{l,0}.  
	\eqno(S2)
	$$
We define  
$\eta_i\in \mathcal{Z}_{l,i}$ to be the point corresponding to $\IP(\Dir _{\eta_{i-1}}  ({\mathcal{Z}_{l,i-1}}))$, 
$D_i \subset \cZ_{l,i} $ to be the closure of $\eta_i$, 
and $z_{l,i}\in \mathcal{Z}_{l,i} $ to be the point corresponding to $\IP(\Dir_{z_{l,i-1}}(\mathcal{Z}_{l,i-1} )/T_{z_{l,i-1}}(D_{i-1}))$ by the canonical isomorphism.

If  $ D_i=\overline{\eta_i} $ is isomorphic to $ D_{i-1}\cong D $ and permissible for $ \mathcal{Z}_{l,i} $ at $ z_{l,i} $, then we  blow  up along $ D_i $.
Otherwise we stop.

This process is infinite if and only if $\delta(I;u) = \infty $ ($ \Leftrightarrow \cpoly Iu = \emptyset $). 
When the process is finite, set $a(l)$ as the number of the blowing ups in (S2).
It appears that
$$
	a(l)=\lfloor{l \cdot(\delta(I;u)-1)}\rfloor,\hspace{10pt} \delta(I;u) = 1+ \lim_{l\rightarrow \infty} \frac{a(l)}{l}.
$$
The last equality proves the invariance of $ \delta(I;u) $.
We write $ \delta := \delta(I;u) $ for simplicity.
In Corollary \ref{Cor:delta_in_1/N_Z^e}, we have seen that $ \delta \in \frac{1}{N!} \cdot \IZ_{\geq 0} $, for some $ N \in \IZ_+ $, if $ \cpoly Iu \neq \emptyset $.
Therefore $ \delta = 1 + \frac{a(N!)}{N!} $.

When $l\cdot \delta \in \IZ_{\geq 0 } $, let $E_{a(l)}$ be the last exceptional locus in the affine $T$-chart $ \cZ_{a(l),l} $.
Its ring of function is naturally isomorphic to $ { gr_\delta ( R ) / In_{\delta}(J)}$, by \cite{CosRevista}~Th{\'e}or{\`e}me~A.4(iii). 
Moreover, take any $\phi \in {R/I}$.
In the affine ring $\mathcal{O}_{\cZ_{a(l),l}}$ of the last chart, $T$ is a divisor of $ \phi $ and the order in $T$ defines a filtration on 
$R/I$
(\cite{CosRevista}~Th{\'e}or{\`e}me~A.4(ii)): 
for $n\in \N$, set $(R/I)_n:=\{\phi \in {R/I}\mid \mathrm{ord}_T(\phi)\geq n\}$. 
The graded ring associated to this filtration is naturally isomorphic to $ {gr_\delta ( R )/ In_{\delta}(I)}$. 
\end{proof}


Let us now introduce further invariants of the polyhedron that we will use in the case $ e = 2 $ in order to get refined data on the singularity.
These and $ \delta(I;u) $ will be the key ingredients for the constructions of $ \iop $.

\begin{Def}
	\label{Def:alphabetagammas}
Consider a $ F $-subset $ \Delta \subset \IR^2_{ \geq 0} $  with finitely many vertices, say
$$ 
	v^{(1)}, v^{(2)}, \ldots, v^{(t)} \in \IR^2_{ \geq 0 } .
$$
Suppose the vertices are ordered by the lexicographical order of the coordinates $ v = ( v_1, v_2 ) $.
(In fact, since $ \Delta $ is living in dimension two it suffices to order them \wrt their first entry, i.e., $ v^{(1)}_1 < v^{(2)}_1 < \ldots < v^{(t)}_1 $;
this yields the same ordering).
Then we define

\begin{enumerate}
%
\item[(1)]
	$ \alpha_1 ( \Delta )  := \inf \{ v_1 \mid  v = ( v_1, v_2  ) \in \Delta \} $

\item[(2)]
	$ \beta_1 ( \Delta )  := \inf \{ v_2 \mid  v = ( \alpha_1 ( \Delta ), v_2 ) \in \Delta \} $

\item[(3)]
	$ \gamma_1 ( \Delta ) := \sup \{ v_2 \mid ( \delta ( \Delta ) - v_2, v_2 ) \in \Delta \} $

\item[(4)] 
	If $ t \geq 2 $ then we define $ s_1 ( \Delta ) \in \IR_+ $, 
	to be the number such that the line passing through $ v^{(1)} $ and $ v^{(2)} $ has slope $ \frac{-1}{s_1 ( \Delta )} $.
	\\
	If $ t \leq  1 $ we put $ s_1 ( \Delta ) := \infty $.
\end{enumerate}
Note that $ ( \alpha_1 ( \Delta ), \beta_1 ( \Delta ) ) = v^{(1)} $ and $ ( \delta ( \Delta ) - \gamma_1 ( \Delta ), \gamma_1 ( \Delta ) ) $ are vertices of $ \Delta $.

Analogously we define $ \alpha_2 ( \Delta ) $, $ \beta_2 ( \Delta ) $, $ \gamma_2 ( \Delta ) $, and $ s_2 ( \Delta ) $ by interchanging the role of $ v_1 $ and $ v_2 $.
\end{Def}

For $ \Delta = \cpoly {\ideal}u $ we abbreviate
\begin{equation}
\label{eq:Definition_alphabetagamma} \left\{\hspace{15pt}
\begin{array}{ccc}
\alpha_1 ( I;u) := \alpha_1 (  \cpoly {\ideal}u ), 
& \hspace{10pt}&
\beta_1 (I;u) := \beta_1 (  \cpoly {\ideal}u ),
\\[5pt] 
\gamma_1 (I;u) := \gamma_1 (  \cpoly {\ideal}u ),
& &
s_1(I;u) := s_1 (  \cpoly Iu ) ,
\end{array}\right.
\end{equation}
and analogously if the index is $ 2 $.
If $ \Delta(I;u ) = \emptyset $ then all this numbers are infinity. 
On the other hand, for $ \Delta(I;u ) \neq \emptyset $ we have $ \alpha_1(I;u), \beta_1(I;u), \gamma_1(I;u) \in \frac{1}{N!} \cdot \IZ^2_{\geq 0 } $ (with $ N $ as in Corollary \ref{Cor:delta_in_1/N_Z^e}) and $ s_1 (I;u) \in \IQ_{\geq 1} \cup \{ \infty \}  $.

One sees easily that $ s_1 ({\ideal};u ) $ depends on the choice of the coordinates. 
In particular{, it depends on the choice} of $ u_2 $ (e.g.~$ {\ideal} = \langle y^2 + u_1^3 u_2 \rangle = \langle y^2 + u_1^3 v_2 - u_1^5 \rangle $ for $ v_2 = u_2 + u_1^2 $, then $ s_1 (I;u_1,u_2) = \infty $, whereas $ s_1 (I;u_1,v_2) = 2 $). 

Nevertheless, analogous to Theorem \ref{Thm:alphabetainvariant_1_deltaonly} we have the following unpublished result due to Uwe Jannsen and the first author.

\begin{Thm}
\label{Thm:alphabetainvariant_part2_nodelta}
Let $ {\ideal} \subset R $ be a non-zero ideal in an excellent regular local ring $ (R,M) $, $ \dim (R/I) = 2 = \dim(\Dir_M(I)) $.
Let $ ( u, y ) = ( u_1, u_2; y_1, \ldots, y_r ) $ be a {\RSP}~for $ R $ such that $ ( y ) $ yields the directrix of $ {\ideal} $. 
Suppose $ V ( u_1 ) $ defines a boundary component and that $\Delta ( {\ideal}, u )\not= \emptyset$.

Then the numbers $ \alpha_1 ( {\ideal}; u ) $, $ \beta_1 ({\ideal}; u )  $, and $ \gamma_1 (   {\ideal}; u  ) $ are invariants of the singularity $ R / {\ideal} $ and the boundary component $ V( u_1 ) $, i.e., the only datum of $ R / {\ideal} , \langle u_1 \rangle R / {\ideal}$ is enough to compute these three numbers.
In particular, $ \alpha_1 (  {\ideal};u ) $, $ \beta_1 ( {\ideal}; u  ) $, and $ \gamma_1 (  {\ideal}; u  ) $  are independent of the embedding.
\end{Thm}

\begin{proof}
{\it Invariance of $\alpha_1(I;u)$.} 
Take a standard basis $(f_1,\dots,f_m)$ of $I$ and ${(y)}$ a sequence of parameters in ${R}$ such that $(f\rot{, u}, y)$ is well-prepared: 
$\Delta(f;u;y) =\Delta(I;u)$ (Proposition \ref{Cor:blowuppolyvide}). 

When $ \alpha_1(I;u) \geq 1 $,  $ {D}:=V ( u_1, y ) $ is a permissible center in Spec$ {(R/I)}=:{\mathcal{X}}$ and we denote its generic point by 
$ \eta $.  
By Theorem \ref{Thm:2.2.(2)_CJS}(3) and Theorem \ref{Thm:HiroMizutani},
we obtain that there is no near point to  $x$ on the strict transform of div$(u_1)$,
if we blow up along this center $ {\mathcal{X}}'\longrightarrow {\mathcal{X}}$: 
$ {D}$ is exactly the intersection of the Hilbert-Samuel stratum and div$(u_1)$. 
We note that 
$$
	\alpha_1(I;u)=1 \Rightarrow \mathrm{dim}(\Dir_\eta ( \mathcal X ) )=0.
$$
Set $ \fp := \langle u_1, y_1,\ldots,y_r \rangle $. 
By Theorem \ref{Thm:2.2.(2)_CJS}(3), 
$\alpha_1(I;u)>1$ implies that {the ideal generated by the} initial {forms} of $(y_1,\ldots,y_r)$ in gr$_{\fp}(R_\fp)$ contains the ideal of the directrix 
of {$ {\mathcal X}_\eta$} in $C_{\eta} ( \mathcal X ) $. 
As dim$( \mathcal X_\eta )  = 1 $, 
the initial {forms} of $(y_1,\ldots,y_r)$ 
{determine}
the directrix.
When $\alpha_1(\rot{I};u)=1$, suppose that $ \mathrm{dim}(\Dir_\eta ( \mathcal X ) ) \neq 0 $, i.e.,
$$
\mathrm{IDir}_\eta(\mathcal X)=(Y_1+\lambda_1U_1,\ldots,Y_r+\lambda_rU_1)\subset gr_{\fp}(R_\fp ),
$$ 
for $\lambda_i\in k(\eta) \cong R_\fp / \fp $. 
As $ R/\fp $ is a discrete valuation ring and $\overline{u_2}$ is an uniformizing parameter, $\lambda_i=\mu_i \overline{u_2}^{a_i}$, $\mu_i \in R$, $\mu_i=0$ or the residue of $\mu_i$ is invertible in  $ R/\fp $ and $a_i\in \IZ $, $1\leq i \leq r$. 
Then 
$({\mu_1} u_2^{a_1},\ldots,{\mu_r} u_2^{a_r})$ would be a solution for the vertex of abscissa $\alpha_1(I;u)$ (Definition \ref{Def:initial_vertex}(2)):
This is a contradiction as this vertex is not solvable.

When $ \alpha_1(I;u) > 1 $, we compute $\delta((R/I)_{{{\fp}}})$: 
As in the proof of Theorem \ref{Thm:alphabetainvariant_1_deltaonly}, we construct sequences of blowing ups (S1) (S2). 
Instead of doing the computations in the ring $R_{{\fp}}$, we do them in the ring $R\subset R_{{\fp}}$. 
We look at the  first sequence of  blowing ups (S1).
The first blow-up is made along the permissible center of ideal $(T,I(D))\subset (R/I)[T]$,
the second one is made along the permissible center of ideal $(T,I(D)/T)\subset (R/I)[T, I(D)/T]$, etc.:
$$
	\mathcal{Y}_l \longrightarrow \ldots \longrightarrow \mathcal{Y}_1 \longrightarrow 
 \mathcal{Y}_0:=\mathrm{Spec}((R /I)[T]) ,
$$
where $\mathcal{Y}_i$, $l\geq i \geq 1$ is the affine chart complementary of the strict transform of the preceding blowing-up (= the affine $T$-chart), 
$x$ is the closed point of $\mathrm{Spec}( (R /I)[T])$, 
and let $x_i\in  \mathcal{Y}_i$  be very near to $x$ on the strict transform of div$(u_2)$. 
Then, after computations, $x_l$ is the closed point of Spec$(R [T,\frac{y}{T^l},\frac{u_1}{T^l}]_{\langle T,\frac{y}{T^l},\frac{u_1}{T^l},u_2\rangle})$.
The strict transform of $I\cdot R[T]$ is  generated by $( \frac{f_1}{T^{l \nu_1}},\ldots, \frac{f_m}{T^{l \nu_m}})$ ($ \nu_i = \ord_M(f_i) $) which is a well-prepared standard basis 
\wrt $ ( \frac{y}{T^l} ) $ by Proposition~\ref{Prop:originblowup}.

We look at the second sequence of blowing ups (S2):
$$
\mathcal{Y}_{l,a(l)}  \longrightarrow  \ldots \longrightarrow \mathcal{Y}_{l,1}  \longrightarrow \mathcal{Y}_{l,0}:=\mathcal{Y}_l, \ i\geq 1,  
$$
where  $\mathcal{Y}_{l,i}$  is the affine chart complementary of the strict transform of the preceding blowing-up (= the affine $T$-chart), and $x_{l,i}\in \mathcal{Y}_{l,i}$ is very near to $x$ on the strict transform of div$(u_2)$. 
Then, after computations, $x_{l,i}$ is the closed point of Spec$ (\, R [T, \frac{y}{T^{l+i}},\frac{u_1}{T^l}]_{\langle T,\frac{y}{T^{l+i}},\frac{u_1}{T^l},u_2 \rangle }\, )$, 
and the strict transform of $I\cdot R[T]$ is generated by $(\frac{f_1}{T^{(l+i) \nu_1}},\ldots, \frac{f_m}{T^{(l+i) \nu_m}})$ which is a well-prepared standard basis. 
Finally, we get 
$$ 
	\alpha_1(I;u) = 1+\lim_{l\rightarrow \infty} {a(l)\over l}=\delta((R/I)_{{\fp}}).
$$
 
\rot{If} $ 0 < \alpha_1(I;u) \leq 1 $, we introduce a new variable $ W $ and pass from $ R $ to
\begin{equation}
\label{eq:UwesTrick}
	S := ( R[W]/ \langle W^d - u_1 \rangle )_{\langle u, y, W \rangle },
\end{equation}
where we choose $ d \in \IZ_{+} $ large enough such that
\rot{$d\cdot\alpha_1(I;u)>1$
	and such that $ d $ is prime to the characteristic of the residue field $ k = R/\maxIdeal $ if $ \car(k) > 0 $.
	Further, if $ \alpha_1(I;u)< 1 $, we additionally may assume 
	that $d\cdot\alpha_1(I;u)\not\in \IZ_{\geq 0} $.}
Take a standard basis $(f_1,\dots,f_m)$ of $I$ and ${(y)}$ a sequence of parameters in ${R}$ such that $\Delta(f;u;y) =\Delta(I;u)$. 
It \rot{turns out} that $(f_1,\dots,f_m)$ is a standard basis of $I\cdot{S} \subset S $.
\rot{If $ \rot{v  := ( \alpha_1(I;u) , \beta)} $ denotes the vertex of smallest abscissa of $ \Delta(f;u_1,u_2;y) $,
	then
	the vertex of smallest abscissa of $\Delta(f;W,u_2;y)$ is} 
	$$ 
		\rot{{\tt v } := ( {\tt v}_1, {\tt v}_2) := ( d \cdot \alpha_1(I;u), \beta ).}
	$$ 
	
\rot{We claim that $ {\tt v} $ is not solvable and its initial is normalized (Definition~\ref{Def:initial_vertex}).
	Since we do not touch the elements $ ( y ) $, the second property is clear. 	
	If $ \alpha_1 (I;u) < 1 $ or $ \beta \notin \IZ_{\geq 0} $, then $ {\tt v} $ cannot be solvable
	(in the first case, use $d\cdot\alpha_1(I;u)\not\in \IZ_{\geq 0}$).
	Therefore, suppose that $ \alpha_1 (I;u) = 1 $ and that $ {\tt v} $ is solvable.
	Hence, there exist $ \lambda_1, \ldots \lambda_r \in k $ such that 
	\begin{equation}
	\label{eq:solv_tt_v}
	in_{{\tt v}} (f_i) = F_i (Y + \lambda (WU_2)^{{\tt v}}),
	\ \ \ \mbox{ for } i \in \{ 1, \ldots, m\}.
	\end{equation}
	This implies $ in_{{v}} (f_i) = F_i( Y + \lambda U^v) $, for $ 1 \leq j \leq m $.
	If we can show $ v \in \IZ^2_{\geq 0} $, the claim follows.
	Due to work of \rot{Giraud \cite{GiraudMaxPos} 3.3.4}, the ideal of the ridge of the homogeneous ideal $ \langle F_1, \ldots, F_m \rangle \subset k[Y_1, \ldots, Y_r] $ is generated by \rot{additive homogeneous} elements $ \sigma_1, \ldots, \sigma_s $ which are given by $ \sigma_j = P_j (D^Y_A(F_i)) $, for certain polynomials $ P_j $, $ 1 \leq j \leq s $,
	and Hasse-Schmidt derivatives $ D^Y_A $ \wrt the variables $ ( Y ) $ (see Observation \ref{Obs:Ridge}).
	Using \eqref{eq:solv_tt_v}, we get
	$$
		P_j (D^Y_A( in_{{v}} (f_i) )) = 
		\sigma_j(Y + \lambda U^v) \in k[U,\rot{Y}], 
		\ \ \mbox{ where } k=R/M.
	$$
	Since $ d $ is prime to the residue characteristic, if it is positive, and since $ {\tt v } \in \IZ^2_{\geq 0} $, we obtain that $ v \in \IZ^2_{\geq 0} $ which contradicts the fact that $ v $ is not solvable.
	This shows the claim.}

So, $ \alpha_1 (\ideal\cdot S ; W, u_2) = \alpha_1 ( \, \cpoly { \ideal\cdot S }{W, u_2} \, ) = d\cdot\alpha_1(I;u) > 1 $ and the previous result implies that $ \alpha_1(I;u) $ is an invariant of the singularity.

Since the non-negative number $ \alpha_1(I;u) $ is an invariant whenever $ \alpha_1(I;u) > 0 $ it also has to be an invariant of the singularity if $ \alpha_1(I;u) = 0 $. 
For simplicity, we may write $\alpha$ instead of $\alpha_1(I;u)$.

\smallskip
 
{\em Invariance of $\beta_1(I;u)$}. 
When $\alpha=0$, we use the extension (\ref{eq:UwesTrick}) above with $d>>0$. 
When $d$ is large enough,  
$ v:= (\alpha, \beta_1(I;u) )=(0, \beta_1(I;u))$ is the unique vertex of minimal modulus of 
$ \poly{{f}}{W,u_2}{y}$.
For every $i, \ 1\leq i \leq m$, we have $ in_v (f_i)\in k(x)[U_2,Y]$, and as $v$ is not solvable for $ (u_1,u_2,y) $, it is not for $ (W,u_2,y) $:
$\beta_1(I;u) =\delta(I\cdot S; W,u_2) $ is an invariant.

When $ 0< \alpha<1$, doing as in (\ref{eq:UwesTrick}) above, we \rot{are reduced to the case} $ \alpha\geq 1$.
In the case $\alpha\geq 1$, ${D} = V(u_1, y ) \subset \mathcal{X}$ is permissible at the origin of $ \mathcal{X} =\mathrm{Spec}(R/I) $ (Theorem \ref{Thm:2.2.(2)_CJS}(4)).

For $ n \in \IZ_+ $, let us make a sequence of blowing ups similar to (S1), where this time $ u_2 $ is taking the role of $ T $,
$$
 {\mathcal{X} _n} \longrightarrow \ldots \longrightarrow {\mathcal{X} _1} \longrightarrow {\mathcal{X} _0}:=\mathcal{X}.
 $$
 This means: 
 First, we blow up the closed point $ x_0 :=x $.
 Afterwards the center is the unique closed point $x_i\in {X_i}$ on the strict transform ${D_i}$ of ${D}\subset \mathcal{X}$ intersecting the exceptional divisor of the last blowing up (= origin of the $ U_2 $-chart).
 The strict transform of $I$ is generated by $(f^{(n)}) :=(\frac{f_1}{u_2^{n\nu_1}},\ldots,\frac{f_m}{u_2^{n\nu_m}})$ and the local ring at $x_n$ is  $R [\frac{y}{u_2^n},\frac{u_1}{u_2^n}] _{\langle \frac{y}{u_2^n},\frac{u_1}{u_2^n},u_2 \rangle }$.
 By Proposition~\ref{Prop:originblowup}, $(f^{(n)}, \frac{y}{u_2^n})$ is well prepared. 
 The transformation law is easy:    
$$
	\Delta\left( f^{(n)};{\frac{u_1}{u_2^n}, u_2};{\frac{y}{u_2^n}} \right) = \mathrm{Conv}\left\{ \bigcup_{(v_1,v_2)\in \poly{f}{u}{y}}(v_1,v_2+n(v_1-1))+\R^2_{\geq 0 }\right\}.
	$$ 
Then for $n>>0$, it \rot{turns out} that $\Delta\left( f^{(n)};{\frac{u_1}{u_2^n}, u_2};{\frac{y}{u_2^n}} \right)$ has only one vertex which is 
$(\alpha, \beta_1(I;u)+n(\alpha-1))$.
 Since $\alpha$ and $ \delta(\mathcal{O}_{\mathcal{X}_n,x_n})=\alpha+\beta_1(I;u)+n(\alpha-1)$ are invariants of $(R/I,\langle u_1 \rangle \cdot R/I)$, the same is true for $\beta_1(I;u)$. 


\smallskip
{\em Invariance of $\gamma_1(I;u)$}. 
By Theorem \ref{Thm:alphabetainvariant_1_deltaonly}, $ {gr_\delta ( R )/ in_{\delta}(J)}$ is invariant up to isomorphism of graded algebras.
 Suppose $(f,y)$ well-prepared. 
 By \cite{HiroCharPoly}(3.17.3), $ ( in_{\delta}(f_1),\ldots,in_{\delta}(f_m)) $ is a standard basis of $In_{\delta}(I)\subset gr_{\delta}(R)$.
	Further, with natural notations,
	$$
		\gamma_1(I;u)=\alpha(In_{\delta}(I)\cdot gr_{\delta}(R)_{\langle Y,U \rangle})=\alpha({(gr_\delta ( R )/ In_{\delta}(I))}_\mathfrak{M},U_1),$$
where $\mathfrak{M}$ is the ideal of elements of positive weight and $U_1$ the initial of $u_1$. 
\end{proof}

\begin{Obs}
The hypothesis  $\Delta ( {\ideal}, u )\not= \emptyset$ is fulfilled in what concerns us: when $  e_X^O ( x ) = e_X ( x ) = 2 $,  with the notations of Definition~\ref{Def:logHS_function}, $\cpoly {J^O}u\not= \emptyset$.
	Indeed,  by Remark~\ref{Rk:polyvide}, 
$\cpoly {J^O}u = \emptyset$ would imply that V$(J^O)$ is regular, so $X$ would be regular at $x$ and $O(x)$ empty. In our process, either $x$ is singular or $O(x)\not=\emptyset$: contradiction.

\end{Obs}

\begin{Rk}
\begin{enumerate}
	\item[(1)]	By Theorems \ref{Thm:alphabetainvariant_1_deltaonly} and \ref{Thm:alphabetainvariant_part2_nodelta}, the numbers $ \delta (  R / {\ideal} ):= \delta ({\ideal}; u )$, $ \alpha_1 (  R / {\ideal}, u_1 ) := \alpha_1 ({\ideal}; u ) $, $ \beta_1 (  R / {\ideal}, u_1 ) := \beta_1 ( {\ideal}; u  ) $, and $ \gamma_1 ( R / {\ideal}, u_1 ) := \gamma_1 ( {\ideal}, u  ) $ are invariants of the singularity $ R / {\ideal} $ and the boundary component $ V( u_1 ). $ 


	\item[(2)]
	As $ {gr_\delta ( R )/ In_{\delta}(I)}$ is invariant up to isomorphism of graded algebras, the ideal $ In_\delta ( \ideal ) \subset gr_\delta ( R ) $ defines a quasi-homogeneous tangent cone which is an invariant of the singularity $ R / \ideal $ and which is a refinement of the usual tangent cone given by the initial forms \wrt the maximal ideal.
	As we have seen in the proof above this carries crucial information on the singularity.
	Moreover, it appears also in other works,
	e.g.~it is used in the proof for resolution of threefold singularities by Piltant and the first author \cite{CPmixed} and \cite{CPmixed2}, and it is also hidden in the constructive proof for resolution of singularities over fields of characteristic zero.
	
\end{enumerate}
\end{Rk}

%
%
%
%
%
%
%
%
%
%
%
%
%
%
%
%
%
%



\section{Control on the old Hilbert-Samuel locus}
\label{sectionnumberControl}

\noindent
{\em Starting from this section we suppose that $ X $ is reduced and of dimension at most two.}
%
We consider the situation
\begin{equation}
\label{sequenceofblowups}
	X \to \ldots \to X_\ast \stackrel{\ast}{\to} X_{\ast -1} \to \ldots \to X_0 ,
\end{equation}
where $ X_0 $ is the scheme with which the resolution problem initially started.
Let $ \cB_0 $ be the initial boundary on $ X_0 $ and $ \cB_ , \ldots, \cB_\ast, \cB_{\ast - 1}, \ldots $ its respective transforms. 
\rot{Let $ x \in X, \ldots, x_\ast \in X_\ast, x_{\ast - 1} \in X_{\ast - 1} $ be the points (not necessarily closed) such that the one before lies above the one after.}
We assume that the value of the log-Hilbert-Samuel function strictly decreases at the $ \ast $-th stage (while it remains the same afterwards), i.e., we have
$$  
	\tnu  := H_{ X }^O ( x ) = \ldots =  H_{ X_\ast }^O ( x_\ast ) <  H_{ X_{\ast -1} }^O ( x_{\ast -1 } ) .
$$
(The decrease may in particular happen even if the Hilbert-Samuel function stays the same but the number of the old components decreases at $ x_\ast $).

Note that this assumption \rot{implies by Proposition \ref{Prop:ioo_non_increase}} 
$$
		\ioo(X, \cB, x) \leq \ldots \leq \ioo ( X_\ast, \cB_\ast, x_\ast ).
$$

In general, $ \tnu $ is not necessarily a maximal value of the log-Hilbert-Samuel function.
But, by Theorem \ref{Thm:HS^O_usc}(4) $ X_\ast (\tnu) $ is locally closed.
Thus let $ U_\ast \subset X_\ast $ be an open neighbourhood of $ x_\ast $ such that  
$$ 
	C_\ast := X_\ast^O (\tnu) \cap U_\ast \subset U_\ast
$$
is closed.
Therefore, if necessary, we pass to the local situation in $ U_\ast \subset X_\ast $  in the above sequence and may suppose that $ \tnu $ is a maximal value of the log-Hilbert-Samuel function.

Our goal in \S \ref{sectionnumberControl} is to make the strict transform of $ C_\ast $, which is denoted by $ C $ on $ X $, $ \cB $-permissible.
We construct an invariant $ \ioc $ measuring the improvement under the algorithm of \cite{CJS} for the one-dimensional excellent Noetherian schemes $ C $ with boundary $ \cB $.
In other words, we solve the problem in one dimension less (in the spirit of using the induction on dimension).

Note: During the resolution process for $ C $, we possibly create new irreducible components in the maximal Hilbert-Samuel locus of $ X $.
However, each of them is already $ \cB $-permissible (as long as the value of the log-Hilbert-Samuel function remains the same as $ \tnu $), see Observation \ref{Obs:GooOnear}.
How to measure the improvement of the singularities by blowing up these $ \cB $-permissible centers is the topic of \S \ref{sectionnumberRefine}.

\begin{Not}
We classify the situation of $ x \in X $ into the following four cases:
\begin{itemize}
	\item[(I)]	$ x = C \subset X $ a closed point.
	\item[(II)]	$ x \in C \subset X $, $C$ is a curve $ \cB $-permissible at $ x $ (hence irreducible and regular at $ x $).	
	\item[(III)]	 either $  C = \bigcup_{i=1}^t C_i \subset X $ (with $ t > 1 $) a family of irreducible curves (and hence reducible) at $ x $, or $ C = C_1 \subset X $ an irreducible curve which is not $ \cB $-permissible at $ x $.
	\item[(IV)]	$ x \notin C $.
	\item[(V)] {$ x \in C$ and $ \dim (C) =  2$}. 
\end{itemize}

{We remark that, if $ x \in C $ is not a closed point, i.e., $\overline{\{x\}}$ is an irreducible curve (resp.~an irreducible surface) then we are in Case (II) (resp.~(V))} 
\rot{or in Case (IV) (both cases).} 
\end{Not}

\begin{Rk}
	\label{Rk:(3)(V)done}
\begin{enumerate}
\item[(1)]		Case (IV) may happen,  
\begin{itemize}
	\item 	if $ X $ is singular, but $ x \in X  $ is a regular point of $ X $, or
	\item	if $ C = \emptyset $.
\end{itemize}

On the other hand, if $ x \in X $ is in Case (IV), then we see at most newly created irreducible components of the maximal Hilbert-Samuel locus of $ X $ which we treat in \S \ref{sectionnumberRefine}.

\item[(2)] 	Moreover, in Case (I) or (II), $ C $ is $ \cB $-permissible at $ x $ and the goal of \S \ref{sectionnumberControl} is achieved.
			
\item[(3)] {In Case (V) we have $ C = X $. 
	Hence $ X  $ is regular at $ x $ and $ O(x) = \emptyset $.
	We claim that the resolution process is finished:
	Let $ x_\star \in X_\star $, $ \star \leq \ast $, be the point in the sequence of blowing ups \eqref{sequenceofblowups} such that $ x $ lies above $ x_\star $ and 
	$ H_X (x) = H_{X_\star} (x_\star) < H_{X_{\star-1}} (x_{\star-1}) $.
	Then $ X_\star $ is regular at $ x_\star $ and all boundary components are old, $ \cB_\star (x_\star ) = O(x_\star) $. 
	By considering  the local situation, one sees that all exceptional divisor of the blowing ups $ X \to \ldots \to X_\star $ passing through $ x $ are transversal to $ X $, i.e., $ X $ and $ \cB(x) = N(x) $ intersect transversally.}
\end{enumerate}
\end{Rk}

\noindent
This motivates

\begin{Def}[\bf Control on the old Hilbert-Samuel locus]
\label{Def:ioc}
	Let $ x \in X  \subset Z $ and $ \cB $ be a boundary on $ Z $.
	We define the invariant $ \ioc $ as follows:
	\begin{itemize}
		\item	If we are in Case (IV) {or (V)}, we set 
	$$
		 \ioc (X, \cB, x) := (0, \ldots , 0 , 0 ,0 ,0 , 0) \in \INN \times \IN^3 \times {\IQinfty^2} 
	$$
	
	\item	If we are in Case (I) or (II), we set 
	$$
		 \ioc (X, \cB, x) := (0, \ldots , 0 , 0 ,0 ,0 , 1) \in \INN \times \IN^3 \times {\IQinfty^2} 
	$$

	\item	If we are in Case (III), we set 
	$$ 
		\begin{array}{c}
			\ioc := \ioc (X, \cB, x) := 
				( H_C ( x ),\, | O_C (x) |,\, e_C (x) ,\, e^O_C ( x ), \, \delta_C ( x ) ,\, \delta^O_C ( x ) \,) 	=
			\\[9pt]
			\hspace{26pt}
	 = ( H_C^O ( x ),\, e_C (x) ,\, e^O_C ( x ), \,\delta_C ( x ), \,\delta^O_C ( x ))
	 = ( \ioo (C, \cB, x), \,\delta_C ( x ), \,\delta^O_C ( x )  ) .
	\end{array}
	$$
	\end{itemize}	
	Here, the last two entries are defined by
	$$ 
		\delta_C ( x ) := \delta(I_C; v) ,
		\hspace{20pt}
		\delta^O_C ( x ):= \delta(I_C^O; v^O) ,
		\hspace{10pt}
		\mbox{ (Definition \ref{Def:delta}) }
	$$
	where $ I_C \subset R = \cO_{Z,x} $ is the ideal which defines $ C $ locally at $ x $,
	$ (v,z) $ is a {\RSP} such that $ (z) $ determines the directrix of $C $ at $ x $,
	$ I_C^O := I_C \cdot I_{O_C(x)} $
	with $ I_{O_C(x)} $ the ideal which defines $ O_C(x) $ locally at $ x $, 
	and
	$ (v^O,z^O) $ a {\RSP} for $ R $ such that $ (z^O) $ determines the directrix of $ I_C^O $ at the maximal ideal of $ R $.
	
	Note that $ \ioc \in 
	\INN \times \IN \times \{ 0, 1\} \times \{ 0, 1 \} \times \frac{1}{t} \cdot \INinfty \times \frac{1}{t'} \cdot \INinfty \subset
	\INN \times \IN^3 \times \IQinfty^2 $ 
	for certain fixed integers $ t, t' \in \IN $ and $ \IQinfty = \IQ_{\geq 0} \cup \{ \infty \} $.
	Further, Theorem \ref{Thm:alphabetainvariant_1_deltaonly} implies that $ \ioc $ is an invariant only depending on $ X $ and hence is independent of the embedding in $  Z $.
	This justifies the notation $ \ioc(X,\cB, x) $ instead of $ \ioc(X,Z,\cB,x) $.
\end{Def}


\begin{Rk}
\label{Rk:delta_infty}
\begin{enumerate}
	\item[(1)] 	In fact, in the first part of the definition we could have simply set $ \ioc (X, \cB,x) = 0 $, because all it is telling us is that $ x $ in not contained in $ C $ \rot{or} that the resolution process \rot{is over.}
	In particular, if $ x \notin C $ is a point in the maximal Hilbert-Samuel locus then it is contained in some newly created irreducible components which we study after this section.
	But in order to make things comparable we need to have an element in $ \INN \times \IN^3 \times \IQinfty^2 $.

	
	\item[(2)] 	In Case (I) and (II), the algorithm of \cite{CJS} (locally) picks $ C $ as the center for the next blowing up. 
	After performing this blow-up, the strict transform of $ C $ is empty.
	Thus we are in Case (IV) and $ ( \ioo, \ioc) $ strictly decreases at every point lying above $ C $.
	
	In principle, we could also set $ \ioc = (0, \ldots , 0 , 0 ,0 ,0 , 0)  $ if we are in Case (I) or (II).
	But then we need to use the refinement of $ ( \ioo, \ioc ) $ (see \S \ref{sectionnumberRefine}) for detecting the improvement after blowing up $ C $. 
	With our variant, $ \ioc $ emphasizes the difference between $ x \in  C $ and $ x \notin C $.
	Moreover, we use the final refinement solely for the new irreducible components.
			

	\item[(3)] 	If $ e_C (x) = 1 $ then the number $ \delta_C ( x ) $ measures how far $ C $ is away from being regular at $ x $.
	As soon as this is the case and if $ e^O_C (x) = 1 $, then $ \delta^O_C ( x ) $ measures how far $ C $ is away from being transversal to the old boundary.
	
%
\end{enumerate}
\end{Rk}


\begin{Prop}
\label{Prop:ioc_non_increase}
	Let $ \pi : Z' \to Z $ be a $ \cB $-permissible blow-up {with center $ D \subset X $} following the \cite{CJS}-algorithm.
	Consider {$ x \in D $ and} $ x ' \in \pi^{ -1} ( x ) $.
	
	Then the invariant $ (\ioo,\ioc) $ does not increase, 
	$$
		(\ioo,\ioc)(X',\cB',x') \leq (\ioo,\ioc)(X,\cB,x),
	$$
	and the inequality is strict if $ x $ is in Case (I), (II), or (III).
\end{Prop}

\noindent
Note: If $ x \notin D $ then the situation does not change and $ (\ioo,\ioc) $ stays the same.

Since $ \ioc $ cannot decrease infinitely many times, we get that after finitely many blowing ups the strict transform $ C' $ of $ C $ is $ \cB $-permissible.
In the next step \cite{CJS} will blow up with center $ C'$.
Hence the strict transform of $ C $ is then empty and $ \ioc (x'' ) = (0, 0, \ldots, 0) $, for every point $ x''$ on the strict transform $ X'' $ of the given surface $ X$.

\begin{proof}
	{Recall that by Remark \ref{Rk:(3)(V)done}(3), the resolution process ends if $ x $ is in Case (V).
		In particular, the \cite{CJS}-algorithm does no further blowing ups.}
		
By Proposition \ref{Prop:ioo_non_increase}, we have $ \ioo ( X', \cB', x' ) \leq \ioo ( X, \cB, x ) $.
If $ x $ is in Case (IV) then so is $ x' $ and $ \ioc(x) = \ioc(x') =  (0, \ldots , 0 , 0 ,0 ,0 , 0)  $, by definition.

As we explained in the previous remark, $ \ioc $ strictly drops if $ x $ is in Case (I) or (II) and $ H_{X'}^O(x') = H_X^O (x) $.
More precisely, we have $ D = C $ and thus 
$  \ioc(x') =  (0, \ldots , 0 , 0 ,0 ,0 , 0) <  (0, \ldots , 0 , 0 ,0 ,0 , 1) = \ioc(x) $.

The crucial case that needs to be understood is Case (III).
Therefore we may assume that $ C $ is a family of irreducible curves or a single curve which is not $ \cB $-permissible at $ x $.
{\em Then the center following the \cite{CJS} algorithm is the closed point $ D = x $.}

{If $ x' $ is not in Case (III) and $ \ioo(x') = \ioo (x) $, then \rot{$ \ioc(x') < \ioc (x) $, by Definition~\ref{Def:ioc}.} 
	Therefore without loss of generality, we may assume that $ x' $ is in Case (III), in particular, $ x' $ be a closed point}. 

Let $ I_C \subset R $ be the ideal which defines $ C $ locally at $ x $.
Further, let 
\rot{$ (v, z ) = ( v_1, \ldots, v_d, z_1, \ldots, z_s ) $} 
be a  r.s.p.~for $ R $  such that $ (z) $ defines the directrix $ \Dir_x ( C ) $ of $ C $ at $ x $ and such that 
$$
	\poly{I_C}{v}{z} = \Delta(I_C; v),
$$ 
as obtained by Proposition \ref{Cor:blowuppolyvide}.
We have $ d = e_C (x) $ and
$$
	 e_C^O (x) \leq  e_C (x) \leq 1 
$$

%
Suppose $ e_C^O(x) = 0 $.
Theorem \ref{Thm:HiroMizutani$^O$} implies $ H_{C'}^{O} (x') < H_{C}^O (x) $ and hence $ \ioc (x') < \ioc (x) $ at every point $ x'$ lying above the center $ x $.


{\em From now on, we may assume $ e^O_C(x) = 1 $ (which implies $ e_C (x) = 1  $).}
Suppose $ 1 < \delta_C ( x ) < \infty $.
After the blowing up with center \rot{$ D = x  = V(v_1, z_1, \ldots, z_s)$} 
(remember we are in Case (III)), the only point $x'$ where we possibly have
$$
 ( H_{C'}^{O} ( x' ),\,  e_{C'} (x') , \, e^{O}_{C'} ( x' ) )
 =  ( H_C^O ( x ), \,  e_C (x) , \, e^O_C ( x ) )
$$
is the origin of the $ V_1 $-chart (Theorem \ref{Thm:HiroMizutani$^O$}).
If the previous equality holds, then $ x ' $ is very near to $ x $ since it is very O-near to $ x $ (see Definition \ref{Def:very_O-near}).
Moreover, we have $ e_{C'} (x') = e^{O}_{C'} ( x' ) = 1 $.
Proposition \ref{Prop:originblowup}(2) and Remark \ref{Rk:delta_and_(very)near}(2) yield
$$ 
	1 < \delta_{C'} ( x' ) = \delta_{C} ( x ) - 1 < \delta_{C} ( x )  ,
	\hspace{10pt} \mbox{ i.e., }
	\ioc (x') < \ioc (x) .
$$

{\em Hence we may assume additionally $ \delta_{C} ( x )  = \infty $.}
Then the characteristic polyhedron is empty $ \Delta(I_C; v_1) = \emptyset $.
Proposition~\ref{polyvide} yields coordinates $ ( z ) = (z_1, \ldots, z_s )$ in $ R $ such that 
$ 	\poly{I_C}{v_1}{z} = \Delta(I_C; v_1) = \emptyset $,
V$(z)$ is the (maximal) Hilbert-Samuel stratum of V$(I_C)$ in a neighborhood of $x$, and V$(z)$ is permissible for V$(I_C)$ at $x$. 	
In fact, by Remark \ref{Rk:polyvide}, $ V (I_C ) = V(z) $ since $ e_C (x) = 1 = \dim (C )$.

But $ V(z) $ cannot $ \cB $-permissible since the designated \cite{CJS}-center is the closed point 
\rot{$ x = V(v_1,z) $.}
This happens if there are boundary components that are old for $ C $ and that are not 
\rot{normal crossing}
with $ V(z) $ (e.g.~$ O_C(x) = \{ V ( z_1 + v_1^2)\} $).

Again, the only point $x'$ that we need to consider is the origin of the $ V_1 $-chart.
By Proposition \ref{Prop:originblowup} (applied for $ I_C $), we get $ \delta_{C'} ( x' ) = \delta_{C} ( x ) = \infty $.

If $ 1 < \delta_C^O ( x ) < \infty $, then applying the same arguments as in the previous case ($ 1 < \delta_C ( x ) < \infty $), but this time for $ \cpoly{I_C^O}{v_1} $, yields
%
$
	\ioc (x') < \ioc (x) .
$	

{\em Therefore, the remaining case is
$ e_C(x) = e_C^O (x) = 1  $ and
$ \delta_{C} ( x ) =  \delta_{C}^O ( x )  = \infty $ and
the \cite{CJS}-center is the closed point $ x $.}
Let us collect what we already know. 
Since $ \delta_{C} ( x )  = \infty $ and $ e_C(x) = 1 $, we have that $ V(I_C) = V(z) $ is a regular curve in a neighborhood of $x$.
By the assumption $ e_C^O(x) = 1 $, the initial forms of each old component is determined by the coordinates in $ (z) $.
Finally, $ \delta_{C}^O ( x )  = \infty $ implies that $ V(z) $ and the old boundary components cannot be tangent and thus $ V(z) $ is contained in the divisor defined by $ O_C(x) $;
	in particular, a local generator of a component in $ O_C(x) $ is contained in the ideal $ \langle z \rangle $.
	
The only reason why \cite{CJS} do not choose $ V(z) $ but $ V(z,v_1 ) $ as the center for the next blowing up is that the $ V(z) $ is not n.c.~with the boundary.
	There can be only one boundary component which is new for $ C $, say, without loss of generality, $ N_C(x) = \{ V(v_1) \} $.
	Clearly, $ V(z) $ is n.c.~with $ V(v_1) $.
	But then the intersection of $ V(z) $ with some of the old boundary components must yield the center $ V(z,v_1 ) $.
	This implies that the local generator of at least one old boundary component is {\em not} contained in $ \langle z \rangle $.
	Therefore we get a contradiction to the above.
	
	As conclusion we get obtain: 
	if $ e_C(x) = e_C^O (x) = 1  $ and $ \delta_{C} ( x ) = \delta_C^{O }(x)  = \infty $ then the \cite{CJS}-center cannot be the closed point $ x $.
	Hence we cannot be in Case (III).
	This ends the proof on the decrease of $ \ioc $.
\end{proof}

In dimension two, in order to have a good control on what is above the center of the blow-up, i.e., on the points on the {newly} created irreducible components, we use Hironaka's characteristic polyhedron and define $ \iop $ in the next section.

%
%
%
%
%
%
%
%
%
%
%
%
%
%
%
%
%



\section{Refinement for the new irreducible components}
\label{sectionnumberRefine}

In this section, we show how our invariant is helpful to get hands on the newly created irreducible components in the Hilbert-Samuel locus,
and to end the process finally. 
The same arguments as in section \ref{sec:HiroPoly} allow us to reduce from $ x \in X $ with boundary $ \cB $ to the local embedded situation:

Let $ ( R, \maxIdeal, \Resfield = R / \maxIdeal ) $ be an excellent regular local ring with maximal ideal $ \maxIdeal $.
Note that we only assume $ R $ to be excellent but {\em not necessarily complete}. 
The characteristic polyhedron in which we are interested in is the one associated to the ideal $ J^O  = J \cdot I_O $, where $ J \subset R $ is the ideal which defines $ X $ locally at $ x $ and $ I_O:= I_{O(x)} $ denotes the ideal which defines the divisor given by the old boundary components of $ \cB ( x ) $ locally at $ x $.

By the previous section, $ (\ioo, \ioc ) $ drops strictly as long as there is an irreducible component 
 of $X^O(\tnu)$ (the log-Hilbert-Samuel stratum of $x\in X$)
 which is the strict transform of a component of  $X^O_*(\tnu)$ -- 
 for the notations we refer to the beginning of the last section.
Hence we may assume:
\begin{equation}
\label{eq:ALl_Irred_NEW}
\left\{ \hspace{10pt}
\parbox{280pt}{\em There exists no irreducible components of $X^O(\tnu)$  which is the strict transform of a component of  $X^O_*(\tnu)$.}
\right.
\end{equation}
This implies that each irreducible component of dimension \rot{one} is contained in some new boundary component (new for $ X $).


\begin{Def}[\bf Refinement for $ \bf e_X^O ( x ) \leq 1 $]\label{Def:iope=1}
	We set 
	$$
	\iop := \iop ( X,\cB, x) := 
	\left\{\hspace{5pt}
	\begin{array}{lc}
	(0,0,0,0), & \mbox{ if } e_X^O ( x ) = 0 , \\[3pt]
	(0, 0, 0, \delta_X^O ( x )), & \mbox{ if } e_X^O ( x ) = 1. 
	\end{array}  
	\right. 
	$$ 
\end{Def}
Recall that $ \delta_X^O(x) $ is an invariant of the singularity by Theorem \ref{Thm:alphabetainvariant_1_deltaonly}. 
Hence this justifies that we write $\iop ( X,\cB, x)$ instead of $\iop ( X,Z,\cB,x)$.

(In fact, it would suffice to take the values $ 0 $ \rot{(}resp. $ \delta_X^O ( x ) $\rot{)} instead of $ (0, 0, 0,0) $ \rot{(}resp. $ (0, 0, 0, \delta_X^O ( x )) $\rot{)}, but in order to make the invariant comparable with the one in the case $ e_X^O ( x ) = 2 $ we need to define an element in $ \IQinfty^4 $).

\begin{Prop}\label{Prop:e(x)=e(x_*)}
	Let $ \pi : Z' \to Z $ be a $ \cB $-permissible blow-up {with center $ D \subset X $} following the \cite{CJS}-strategy.
	Consider {$ x \in D $ and} $ x ' \in \pi^{ -1} ( x ) $.
	
	Suppose  $ e_X^O(x) \leq 1  $.
	Then $(\iota_0,\iota_{hs},\iop)(x')<(\iota_0,\iota_{hs},\iop)(x)$.
	
\end{Prop}

\begin{proof}
	By the previous section we may assume that \eqref{eq:ALl_Irred_NEW} holds.
	
	If $ e_X^O ( x ) = 0 $, then, following the strategy of \cite{CJS}, $x$ is the center of the next blowing up $ \pi $.  
	Theorem \ref{Thm:HiroMizutani$^O$} implies that there are no $ O $-near points, i.e., $ H_{ X' }^O ( x' ) < H_X^O ( x ) $, for every point $ x' \in \pi^{ - 1} ( x ) $ lying above $ x  $.
		
	If $ e_X^O ( x ) = 1 $, then the characteristic polyhedron of $ J^O $ is \rot{one}-dimensional or empty.
	Recall that $ \delta_X^O ( x ) \in  \frac{1}{t} \IZ_{\geq 0} \cup \{\infty\} $, for some fixed $ t \in \IZ_+ $ (Corollary \ref{Cor:delta_in_1/N_Z^e}). 
		
	Assume $ \delta_X^O ( x ) < \infty $ {and $ x $ is a closed point}.
	Let $X'\to X$ be the blow up along $x$: 
		by Theorem \ref{Thm:HiroMizutani$^O$}, there is at most one point $ x'$ which is $ O $-near to $ x $.
		Further, Proposition \ref{Prop:originblowup}(2) and Remark \ref{Rk:delta_and_(very)near} imply $ \delta_{X'}^O ( x' ) = \delta_X^O ( x ) - 1 $ or
		$e_{X'}(x')=0$.
		This proves that $x$ is isolated in its log-Hilbert-Samuel-stratum. Following the strategy of \cite{CJS}, $x$ is the center of the next blowing up. 
		
		{If $ \delta_X^O ( x ) < \infty $ and $ x $ is a non-closed point then locally at $ x $ the center is $ D = \overline{\{ x \}} $.
			With the same arguments as above, we get 
			$ \ioo(x') < \ioo(x) $ or 
			$ \delta_{X'}^O ( x' ) = \delta_X^O ( x ) - 1 $, for every point $ x' $ lying above $ x $.
			  }
		
		A descending induction on $ (e_X^O,\delta_X^O) $ proves the result when $ e_X^O (x ) \leq 1 $ except in the subcase
		$$
			e_X^O ( x ) = 1,\  \delta_X^O ( x )=\infty.
		$$
	
	Let $ e_X^O ( x ) = 1 $ and $  \delta_X^O ( x )=\infty $.
	Suppose the point $x$ is an isolated {closed point or a non-closed point} in the log-Hilbert-Samuel stratum of $ X $. 
	By Proposition~\ref{polyvide}(ii), we have $\delta_X^O ( x )<\infty$, 
	which is a contradiction.
	
	Hence we may assume that {$ x $ is a closed point and} there is an irreducible component $D$ of dimension \rot{one} in the log-Hilbert-Samuel stratum of $ X $ going through $x$.
	Since $\delta_X(x)=\infty$ Proposition \ref{polyvide} implies that $D$ is permissible at $ x $ and locally coincides with the log-Hilbert-Samuel stratum of $ X $.
	Moreover, we find a sequence of parameters $ ( y ) $ such that $ D = V(y) $ and $ ( y) $ determines the directrix of $ \Dir_x^O(X) $.
	
	By \eqref{eq:ALl_Irred_NEW}, we know that $ D $ is contained in a new boundary component.
	Then Observation \ref{Obs:GooOnear} implies that $ D $ is $ \cB$-permissible.
	In particular, it is transversal to the boundary components. 
	
	Therefore $ D $ is $ \cB$-permissible and it is the only irreducible component of the log-Hilbert-Samuel stratum locally at $ x $.
	But this means $ D $ is the next center of the \cite{CJS}-algorithm.
	Theorem~\ref{Thm:HiroMizutani$^O$} yields that the log-Hilbert-Samuel function decreases strictly at every point lying above $ x $ after blowing up $ D $.
\end{proof}

%
%
%
%
%
%
%
%
%
%
%
%
%
%
%
%
%
%
%
%
%


Hence it remains to consider the case $ e_X^O ( x ) = 2 $.
Since $ e_X^O(x) \leq e_X ( x) \leq 2 $, we get $ e_X(x) = 2 $. 
{Note that this implies that $ x $ is a closed point, whereas the situation of a non-closed point is included in our previous discussions.}
Furthermore, since 2 is the maximal value for the dimension of the directrix, it did not drop since $ \nu = H_X(x) $ became a maximal value for the first time. 

Thus, by \cite{CJS}~Theorem~3.23(2),
the new boundary components are automatically transversal to the directrix at $x$. 
This implies that we can choose the {\RSP} $ (u,y) = ( u_1, \ldots, u_e; y_1, \ldots, y_r ) $ of $ R $ ($ e := e_X^O (x) $)
such that $ ( y ) $ determines the directrix of $ J^O $ and each new boundary component $ D \in N(x) $ is (locally at $ x $) given by $ V( u_i ) $, for some $ i \in \{ 1· \ldots, e \} $. 

We make the following convention on the new boundary components that we will use from now on:
\begin{equation}
\label{eq:convention_new_boundary}
\left\{ \hspace{10pt}
\parbox{300pt}{ \em
	When $ |N(x)| = 1 $, then we choose $ (u_1, u_2 ) $ such that $ V(u_1) $ is the new boundary component.
	
	If $ |N(x)| = 2 $, we pick $ (u_1, u_2 ) $ such that the new boundary components are $  V(u_1) $ and $ V(u_2) $.		
}
\right.
\end{equation}



If $ V(J^O ) $ is regular at $ x $, the resolution process locally ends: 
$ X $ is regular at $ x $ and it transversal to the boundary at $ x $ 
{(see also Remark \ref{Rk:(3)(V)done}(3))}.
Note that we have in this case $ \delta_X^O(x) = \infty $.

\begin{Lem}
	\label{Lem:delta_not_infty_if_poly}
	We assume that $ x $ is a singular point of $ V(J^O ) $.
	Suppose \eqref{eq:ALl_Irred_NEW} holds and $ e_X^O ( x ) = 2 $.
	Then $ \delta_X^O ( x )<\infty$.
	In particular, $ \Delta ( J^O; u ) \neq \emptyset $.
\end{Lem}	

\begin{proof}
	Assume $ \delta_X^O ( x ) = \infty $.
	Propositions~\ref{polyvide}(ii) and \ref{Cor:blowuppolyvide} imply that $D = V(y) $ coincides with the log-Hilbert-Samuel stratum in a neighborhood of $ x $ and
	$(y) $ determines the directrix of $ J^O $.
	By the previous lines, new boundary components are transversal to the directrix.
	Hence we get a contradiction with assumption \eqref{eq:ALl_Irred_NEW}.
\end{proof}





The characteristic polyhedron of $ J^O $ is two-dimensional ($ e_X^O ( x ) = 2 $).
If $ V(u_1) $ is a new boundary component then, we set
$$ 
\alpha_1^O := \alpha_1 ( J^O;u), \hspace{10pt} 
\beta_1^O := \beta_1 (J^O;u), \hspace{10pt}
\gamma_1^O := \gamma_1 (J^O;u), \hspace{10pt}
s_1^O(u_2) :=  s_1(J^O;u).
$$
(See \eqref{eq:Definition_alphabetagamma} and Definition \ref{Def:alphabetagammas}).
We use the analogous notations if $ V(u_2) $ is a boundary component.
As we have seen in Theorem \ref{Thm:alphabetainvariant_part2_nodelta} $ \alpha_1^O, \beta_1^O, \gamma_1^O $ are invariants of the singularity and the boundary component.
On the other hand, $ s_1^O (u_2) $ depends on the choice of $ u_2 $.
But, we have

\begin{Obs}\label{Obs:s1} 
	Let $ (u_1, u_2; y_1, \ldots, y_r ) $ be a {\RSP} for $ R $ such that $ V(u_1) $ defines a boundary component and $ \poly{J^O}uy = \cpoly{J^O}{u} $.    
	Suppose that
	\begin{equation}
	\label{eq:beta_s_condition}
		\beta_1^O\geq 1 
		\hspace{10pt}
		\mbox{and} \hspace{10pt}
		s_1^O(u_2) \geq 1. 
	\end{equation}
	The face of $ \cpoly{J^O}{u} $ of slope $ \frac{-1}{s_1^O(u_2)}$ is defined by the equation:
	$$
		\frac{1}{s_1^O(u_2)} \cdot x_1 +x_2 
		= \frac{\alpha_1^O}{s_1^O(u_2)} + \beta_1^O.
	$$
	By \eqref{eq:beta_s_condition},
	the equation verifies condition $(\ast)$ of \cite{CosRevista} on the coefficients of defining equation for the face: 
	$0 < \frac{1}{s_1^O(u_2)} \leq \frac{\alpha_1^O}{s_1^O(u_2)} + \beta_1^O $ and 
	$ 0 < 1 \leq \frac{\alpha_1^O}{s_1^O(u_2)} + \beta_1^O$. Then \cite{CosRevista}~Theorem~B.2.2 states that this side is given by the only datum of the ideals $\langle \overline{u_1} \rangle,  \langle \overline{u_1} , \overline{u_2} \rangle \subset R/J^O$,
	where $ \overline{u_i} $ denotes the image of $ u_i $ in $ R/J^O $, $ i \in \{ 1, 2\} $.
	 
	This justifies the following definition: 
	\begin{equation}
	\label{eq:def_sigma}
		\sigma_1^O := \sigma_1^O (X,x) :=
		\left\{ 
		\hspace{5pt}
		\begin{array}{ccl}
			1, 
			& \hspace{5pt}& \mbox{when $\beta_1^O<1$}
			\\[2pt]
			\sup \{1, \sup_{\,\overline{u_2}\in R/I} \{ s_1^O(u_2)  \}\,\},
			&& \mbox{when $\beta_1^O \geq 1 $},
		\end{array}
		\right.
	\end{equation}
	where the second supremum in the second line is taken over all possible choices for $ \overline{u_2}$ and $ \mathrm{div}(\overline{u_1})\subset \Spec(R/I)$ fixed.
	
	Note: We do not exclude the case that $ V(u_2) $ is a boundary component in the previous considerations.
	We explain the need for this in Example~\ref{Ex:Why_sup_u2}.
	
	\noindent
	(Of course, we define $ \sigma_2^O $ in the analogous way).	
\end{Obs}

Recall convention \eqref{eq:convention_new_boundary}.

\begin{Def}[\bf Refinement for $ \bf e_X^O ( x ) = e_X ( x ) = 2 $]
	\label{Def:iotapoly}
	Let $ x \in X  \subset Z $ and $ \cB $ be a boundary on $ Z $.
	The definition of $ \iop := \iop ( X, \cB, x ) $ splits into the following cases:
	When \eqref{eq:ALl_Irred_NEW} does not hold or $ |N(x)| = 0 $, we put
	$$
	\iop := ( \infty, \infty, \infty, \infty ).
	$$
	If \eqref{eq:ALl_Irred_NEW} holds and $ |N(x)| > 0 $, we define
	$$
		\iop :=
		\left\{ 
		\hspace{5pt}
		\begin{array}{ccl}

			( \,\beta_1^O , \, \gamma_1^O  ,\, \sigma_1^O,\, \alpha_1^O  \, ), 
			& 
			& \mbox{if } |N(x)| = 1,
			\\[2pt]
			\inf_{lex} \left\{ 
			( \,\beta_1^O , \, \gamma_1^O  ,\, \sigma_1^O,\, \alpha_1^O  \, ),
			( \,\beta_2^O , \, \gamma_2^O  ,\, \sigma_2^O,\, \alpha_2^O  \, ) \right\},
			 
			& 
			& \mbox{if } |N(x)| = 2 .
				
		\end{array}
		\right.
	$$
	%
\end{Def}

\smallskip

	By Theorem \ref{Thm:alphabetainvariant_part2_nodelta} and Observation \ref{Obs:s1}, $ \iop = \iop(X,\cB,x) $ is an invariant of the singularity $ x\in X $.
	In particular, it is independent of the embedding into $ Z $.


	The role played by $\beta_1 $ appeared for the first time in \cite{HiroCharPoly} for the case of hypersurfaces 
	(see also \cite{DaleBook} Chapter 7, and in \cite{CJS}). Moreover $\gamma_1$ appears  inside the computations of these  authors when they control the behavior of $\beta_1$ under blowing ups of closed points.
%

\smallskip

Before coming to the precise statement, let us give an example where we illustrate the decrease of $ \iop $ and the need of taking the supremum over $ \overline{u_2} $ in the definition of $ \sigma_1^O $ (even if there are two new boundary components).

\begin{Ex}
	\label{Ex:Why_sup_u2}
	Consider the affine surface $ X $ defined by the polynomial
	$$
		 h := y^2 + (u_2 + \lambda u_1 )^3 + u_1^7 \in k [u_1, u_2, y],
	$$
	for any field $ k $ and $ \lambda \in k^\times $.
	Since $ X $ is a hypersurface it suffices to consider the order to determine the maximal Hilbert-Samuel locus.
	The maximal order appearing is $ 2 $ and the singular locus of $ X $ is the origin $ x = V(u_1, u_2, y) $.
	
	Suppose $ V (u_1) $ and $ V(u_2 ) $ are the only boundary components and both are new for $ X $.
	The polyhedron $ \poly{h}{u_1, u_2}{y} $  has vertices 
	$ v^{(1)} := ( \frac{3}{2}, 0) $ and $ v^{(2)} = (0, \frac{3}{2}) $.
	Moreover, it is minimal \wrt the choice of $ y $ and hence coincides with the characteristic polyhedron $ \cpoly{h}{u_1, u_2} $.
	From this we obtain
	\begin{equation}
	\label{eq:in_Ex_not_sup}
		(   \beta_1^O, \gamma_1^O, s_1^O(u_2), \alpha_1^O ) 
		= ( \beta_2^O, \gamma_2^O, s_2^O(u_1), \alpha_2^O )
		= \left(\frac{3}{2} , \frac{3}{2}, 1  , 0 \right)
	\end{equation}
			
	We blow up the origin $ x $ and consider the $ U_1 $-chart.
	Let $ x' $ be the point with parameters 
	$ (y', u_1, v_2' ) := ( \frac{y}{u_1}, u_1, \frac{u_2}{u_1} + \lambda ) $.
	Clearly, $ x' $ lies above $ x $, and the strict transform of $ h $ at this point is 
	$$
		h' = y'^2  + u_1 v_2'^3 + u_1^5.
	$$
	(Note that $ ( \ioo, \ioc) (x) = ( \ioo, \ioc) (x') = (2,0,  2, 2; 0, \ldots, 0 ) $). 
	We see only one boundary component at $ x' $, namely $ V(u_1) $, and it is not hard to show that
	$$ 
		\iop(x') 
		= 
		(   {\beta_1^O}', {\gamma_1^O}', {s_1^O}'(v_2), {\alpha_1^O}' ) 
		= \left(\frac{3}{2} , \frac{3}{2}, \frac{4}{3}  , \frac{1}{2} \right) >_{lex}
		\left(\frac{3}{2} , \frac{3}{2}, 1  , 0 \right).
	$$	
	(Here, we use the obvious notations $ {\alpha_1^O}' := \alpha_1^O(x') $, etc.).
	
	The explanation for this increase is that \eqref{eq:in_Ex_not_sup} does not compute $ \iop(x) $.
	We need to consider the supremum \wrt the choice of $ \overline{u_2} $ for $ \sigma_ 1 $ (resp.~$ \overline{u_1} $ for $ \sigma_2 $).
	The latter is achieved for $ v_2 := u_2 + \lambda u_1 $ (resp.~$ v_1 := u_2 + \lambda u_1 $);
	more precisely, $ h = y^2 + v_2^3 + u_1^7 $ and
	$$
		\iop (x) = \left(\frac{3}{2} , \frac{3}{2}, \frac{7}{3}  , 0 \right).
	$$
	Indeed, we get a decrease, $ \iop (x') <_{lex} \iop(x) $, and thus $ \io(x) < \io(x') $.
\end{Ex}


\begin{Prop}
	\label{Prop:iop_decrease}
	Let $ \pi : Z' \to Z $ be a $ \cB $-permissible blow-up {with center $ D \subset X $} following the \cite{CJS}-strategy.
	Consider {$ x \in D $ and} $ x ' \in \pi^{ -1} ( x ) $.
	
	Assume $ \ioo ( X', \cB',  x') = \ioo ( X,\cB, x ) $ and  $ \ioc ( X', \cB', x') = \ioc ( X, \cB, x ) $.
	Then
	$$
	\iop ( X',  \cB', x' ) < \iop ( X, \cB, x ).
	$$
\end{Prop}

\vspace{8pt}

In the proof we show that $ (\,\beta_1^O  , \,\gamma_1^O ,\, {\sigma_1^O}  , \, \alpha_1^O \,) $ strictly drops (\wrt the lexicographical order) in every chart of {the} blow-up {above.}  
This implies then the assertion and thus Theorem \ref{MainThm}.

As we will show, the difficult case is when the center in  \cite{CJS} process is the closed point.
In the case that the residue field extension is trivial {the} arguments for the proof for the decrease of $ \iop $ are the same as in the hypersurface case by Hironaka \cite{HiroCharPoly} (see \cite{DaleBook} chapter 7, in particular section 7.4).
Further, the missing part of a non-trivial residue extension is dealt with in \cite{CosToho}, and also in \cite{CJS} Theorem 13.4.
Nevertheless, we provide a complete proof for the decrease involving new arguments.

\begin{proof}
	As we mentioned at the beginning of the section, $ ( \ioo, \ioc) $ decreases strictly as long as there is an irreducible component with label 0 in the maximal Hilbert-Samuel stratum at $ x $ (see also Proposition \ref{Prop:ioc_non_increase}(2)).
	Moreover, we may assume that $ V(J^O ) $ is singular at $ x $.
	If this is not the case, $ X $ is regular at $ x $ and n.c.~with the boundary components containing $ x $, i.e., the resolution process ends at $ x $.
	
	Hence, we may assume that \eqref{eq:ALl_Irred_NEW} holds.
	In particular, there has to exist at least one new boundary component and $ \cpoly{J^O}{u} \neq \emptyset $ (Lemma \ref{Lem:delta_not_infty_if_poly});
	\WLOG $ V ( u_1 ) $ is a new boundary component for $ x $.

	{Since $ \ioo ( X', \cB',  x') = \ioo ( X,\cB, x ) $, we have $ e_{X'}^O(x') = e_X^O ( x ) = 2 $ and thus $ x' $ is a closed point.}

	By the good nature of the newly created singularities (Observation \ref{Obs:GooOnear}), 
	it remains to consider the case that the center of the blow-up is a whole irreducible component of the Hilbert Samuel locus (at $ x $) of label strictly bigger than 0.
	In particular, the center is either an irreducible curve or an isolated closed point.

	We abbreviate $ \alpha_1^O(x)  := \alpha_1^O(X,x) := \alpha_1 ({J^O};u ) $, and we define $ \beta_1^O(x), \gamma_1^O(x) $, $ \sigma_1^O (x) $ analogously.
	Similarly, we define
	$ \alpha_1^O (x'), \beta_1^O (x'), \gamma_1^O(x'), \sigma_1^O(x') $ for the situation at $ x ' $.
	
	Since $ \dim ( X ) = 2 $ Theorem \ref{Thm:HiroMizutani$^O$} implies that the $ O $-near points are contained in the projective space associated to the directrix: 
	Suppose $ V(u_i, y  \mid i \in \Lambda \subset \{ 1, 2\}) $ is the center of the blowing up and $ ( y) = ( y_1, \ldots, y_r ) $ determines the directrix of $ X $ at $ x $.  
	Then the previous implies that the $ O $-near points cannot lie in the $ Y_j $-charts.
	
	Let $ ( f) $ be a standard basis for $ J^O $ and $ ( y ) $ be part of a {\RSP} determining the directrix (of $ J^O $ at $ M $) such that the associated polyhedron is minimal 
	(i.e., it coincides with the characteristic polyhedron, $ \poly{f}{u}{y} = \cpoly{J^O}u $).
	Recall Proposition~\ref{Prop:originblowup}(1):
	If the origin of a chart of the blowing up is very near, then the polyhedron for the transformed data is still minimal.
	
	If the center is an irreducible curve, say $ V ( u_1, y ) $ (resp.~$ V ( u_2, y ) $), then the $ O $-near points can only be the origin of the $ U_1 $-chart (resp.~the origin of the $ U_2 $-chart).
	An easy computation (similar to the one in the proof of Proposition~\ref{Prop:originblowup}) shows that the point $ ( v_1, v_2 ) \in \cpoly {J^O}u $ is translated to $ ( v_1 - 1, v_2 ) $ under the blow-up in $ V ( u_1, y ) $ (resp.~to $ ( v_1, v_2 - 1 ) $ under the blow-up in $ V ( u_2, y ) $).
	In the first case we have 
	$
	\beta_1^O ( x') = \beta_1^O ( x ), \; \gamma_1^O ( x') = \gamma_1^O ( x ), \; \sigma_1^O ( x' ) = \sigma_1^O ( x) $ 
	(for the latter, recall Observation \ref{Obs:s1}),
	and
	$$
	\alpha_1^O ( x' ) = \alpha_1^O ( x ) - 1 < \alpha_1^O ( x ),
	$$
	and in the second we get $ \beta_1^O ( x') = \beta_1^O ( x) - 1 < \beta_1^O ( x) $.
	Therefore we obtain $ \iop ( x' ) < \iop (  x ) $ if the center of the blowing up is an irreducible curve.
	
	\smallskip
	
	For the remaining part of the proof we may assume that locally at $ x $ the maximal (log)-Hilbert-Samuel locus is an isolated closed point.
	The center of the considered blowing up is therefore the closed point $ x $.
	Clearly, once \eqref{eq:ALl_Irred_NEW} holds, we reach this situation in our resolution process after finitely many blowing ups of irreducible curves by the previous explanations.
	In particular, we get (using convention \eqref{eq:convention_new_boundary}): 
	\begin{itemize}
		\item	when $ |N(x)| = 1 $ then $\alpha_1^O (x) < 1 $, and
		\item	when $ |N(x)| = 2 $ then $\alpha_1^O (x) < 1 $ and $\alpha_2^O (x) < 1 $.
		
	\end{itemize}
%
	%
	First we observe that $\beta_1^O ( x ) + \alpha_1^O ( x) \geq \delta_X^O ( x)>1$ and $\alpha_1^O ( x) < 1$ imply
	\begin{equation}
	\label{eq:beta_1>0}
		\beta_1^O ( x )>0.
	\end{equation}
	
	Let $ ( f) = ( f_1, \ldots, f_m) $, $ (u, y ) $ be as in Setup \ref{Setup:blowup} (for $ I = J^O $), $ \nu_i = \ord_M (f_i) $, $ 1 \leq i \leq m$.
	Recall that the center of the blow-up is $ V(u,y) = V(u_1, u_2, y) $ and at least $ \mathrm{div}(u_1) $ defines a new boundary component.
	
	If $ x'$ is on the strict transform of div$(u_1)$, it is the origin of the $ U_2 $-chart (for any choice of $u_2$). 
	The strict transform of $f_i$ is $ f_i' = \frac{f_i}{u_2^{\nu_i}}$, $1\leq i \leq m$.
	If 
	$f_i =\sum C_{a_1, a_2,B,i }\,u_1^{a_1}u_2^{a_2}y^B$ is a finite expansion of $ f_i $ as in \eqref{eq:expansion}, $ C_{a_1, a_2,B,i} \in R^\times \cup \{0\} $, then
	we get 
	$$
	f'_i:={f_i \over u_2^{\nu_i}}=
	\sum C_{a_1 a_2,B} \,{u'}_1^{a_1}{u_2}^{a_1+a_2+\vert B \vert - \nu_i}{y'}^B,
	$$
	where $ u_1' = \frac{u_1}{u_2} $ and $ y' = \frac{y}{u_2} $.
	The transformation law of the polyhedron is easy:    
	\begin{equation}
	\label{eq:trafo_f_i}
	\poly{f'_i}{u_1', u_2}{y'} = 
	\mathrm{Conv}
	\left\{  
	\bigcup_{(v_1,v_2)\in \poly{f}{u}{y}}(v_1,v_2+v_1-1)+\R^2_{\geq 0}
	\right\}. 
	\end{equation}
	Since we have $ \alpha_1^O ( x ) < 1 $ this leads to
	$$ 
	\beta_1^O ( x') = \beta_1^O ( x ) + \alpha_1^O ( x) - 1 < \beta_1^O ( x ) 
	$$ 
	and hence we obtain the desired decrease $ \iop(x') < \iop(x) $.
	
	%
	Let us consider a point $ x '$ in the $ U_1 $-chart, i.e.,
	which is not on the strict transform of div$(u_1)$.
	In this chart, 
	$$
	f'_i:={f_i \over u_1^{\nu_i}}=
	\sum C_{a_1, a_2,B,i}\,{u}_1^{a_1+a_2+\vert B \vert - \nu_i}u_2'^{a_2}{y'}^B
	,
	$$
	with $ y' = \frac{y}{u_1} $, $ u_2' = \frac{u_2}{u_1} $.
	Note that $R[ \frac{y}{u_1},\frac{u_2}{u_1}]/\langle \frac{y}{u_1},u_1 \rangle \cong k(x)[\overline{\frac{u_2}{u_1}}]$ is a polynomial ring. 
	Since $ \ioo(x') = \ioo(x) $, $x'$ is a point very near to $x$ in the $ U_1 $-chart. 
	At $x'$ we have parameters 
	\begin{equation}
	\label{eq:parameters_x'}
	(y',u_1,\psi),\ \mathrm{ where}\ \psi \hbox{ has for residue } \phi(1,\overline{u'_2}) \ \mathrm{ modulo}\ \langle u_1,y' \rangle , 
	\end{equation}
	 and $\phi \in k(x)[U_1,U_2]$ is an irreducible homogeneous polynomial.
	 In order to make formulas easier readable, we set 
	 $$ 
		 \delta := \delta_X^O(x) , \ \
		 \alpha := \alpha_1^O(x), \ \
		 \beta := \beta_1^O(x).
	$$
	Up to multiplication by $ \frac{\delta}{\delta-1}$, the valuation $v_{\delta}$ (Definition~\ref{Def:In_delta}) extends to the monomial valuation $\nu$ on $ R[ \frac{y}{u_1},\frac{u_2}{u_1}] $ centered at $ \langle y',u_1 \rangle $ and defined by 
	\begin{equation}
	\label{eq:def_val_nu}
		\nu(y')=1, 
		\hspace{10pt} 
		\nu(u_1) = \frac{1}{\delta-1}.
	\end{equation}
	%
	%
	Set $ F_i ( Y) := in_M ( f_i) \in k(x)[Y] $.
	We can write the initial form of $ f_i $ \wrt the valuation $ v_\delta $ 
	(see \eqref{eq:in_delta})
	as
	$$
	in_{\delta}(f_i)
	=
	F_i(Y) +
	\sum_{\vert B \vert< \nu_i} 
	\Phi_{i,B}(U_1,U_2)
	\,
	U_1^{a(i,B)}
	\, 
	Y^B 
	\in k(x)[Y,U_1,U_2],
	$$
	where
	$ a(i,B) \geq ( \nu_i - \vert B \vert) \cdot \alpha $,
	and 
	$\Phi_{i,B}(U_1,U_2) \equiv 0 $ 
	or homogeneous of degree 
	$ d_i := (\nu_i - \vert B \vert) \cdot \delta - a(i,B) $.
	Note that the inequality for $ a(i,B) $ implies 
	$$ 
		d_i \leq (\nu_i - |B|)\cdot (\delta - \alpha) \leq (\nu_i - \vert B \vert) \cdot \beta .
	$$ 
	From this, we get a formula for the initial form 
	$ F'_{i,\nu} := in_{\nu}(f'_i) $ 
	of $ f_i' $ \wrt the extended valuation $ \nu $: 
	\begin{equation}
	\label{eq:F_i_nu}
		F'_{i,\nu}
		=
		F_i(Y') + \sum_{\vert B \vert< \nu_i} \Phi_{i,B}(1,\overline{u'_2}) \,{U_1}^{(\nu_i-\vert B \vert)(\delta-1)} \,{Y'}^B
		\in gr_{\nu}(R[y',u_2'])
	\end{equation}
	There are natural isomorphisms 
	$ gr_{\nu}(R[y',u_2'])
	\cong
	\frac{R[y',u_2']}{\langle y',u_1\rangle}[Y',U_1]
	\cong
	k(x)[Y',U_1,\overline{u'_2}].$
	
	We study first the easy case  where $\delta = \delta_X^O(x)  \not\in \IZ_+ $: 
	the vertex of smallest abscissa of $\poly{f'}{u'_1,\psi}{y'}$ is not solvable and, with natural notations, we get:
	$$
	\alpha_1^O (x') = \delta-1,
	\hspace{10pt} 
	\beta_1^O ( x' ) \leq \inf \left\{
	\rot{\frac{\mathrm{ord}_{x'}( \Phi_{i,B}(1,\overline{u'_2}) )}{\nu_i-|B|}}
	\mid  \Phi_{i,B}\not=0 \right\}.
	$$
	Then $\beta_1^O ( x' )\leq \frac{\beta_1^O ( x )}{d} $, 
	where $ d:= [k(x'):k(x)] $. 
	As ${\beta_1^O ( x )}>0$ (see \eqref{eq:beta_1>0}), 
	there holds strict inequality 
	$\beta_1^O ( x' ) < \beta_1^O(x) $ 
	when $ x' $ is not rational over $ x $.
	
	Hence, in order to attain equality, $\beta_1^O ( x' ) = \beta_1^O(x) $,  we must have  
	$$
	\Phi_{i,B}(U_1,U_2) = \lambda_{i,B}(U_1-\lambda U_2)^{(\rot{\nu_i}-\vert B \vert)\beta},
	\hspace{10pt} \mbox{for all possible } i, B,
	$$ 
	with $ \lambda, \lambda_{i,B}  \in k(x) $.
	Then $x'$ is the point of parameters $(y',u_1,u'_2-\lambda)$,
	which means that $x'$ is unique if it exists: 
	call it $x'_{bad}$.
	Moreover, if $ (u) = (u_1,u_2) $ is replaced by $(t) := (t_1, t_2) = (u_1, u_2-\lambda u_1)$, 
	the polyhedron $ \poly{f}{t}{y}$ has only one vertex of smallest module, 
	namely $ w :=(\alpha, \beta) =(\alpha_1^O(x), \beta^O_1(x))$, and  
	${\slope_1^O}(t_2) > 1 $. 
	The study of  $x'_{bad}$ is made below before and in Lemma~\ref{Lem:beta=}.
	
	We come to the case $\delta\in \IZ_+$. 
	Then, as 
	$ \alpha + \beta \geq \delta > 1 $ and 
	$ \alpha < 1$, we get 
	$$
		\delta\in \IZ_+ \,\Longrightarrow\, 
		\beta > 1.
	$$
	By computing the transform of the polyhedron at the origin $ x_0' $ of the $ U_1 $-chart (with parameters $ (y', u_1, u_2') = (\frac{y}{u_1}, u_1, \frac{u_2}{u_1})$), we get 
	$ \alpha_1 ( \poly{f'}{u_1, u_2'}{y'}) = \delta - 1 $.
	(Note: Since $ x_0' $ is not necessarily very near to $ x $ the number $ \alpha_1^O (x_0') $ is possibly not defined).
	This implies that 
	\begin{equation}
	\label{eq:alpha'_delta-1}	
		\alpha_1^O(x')=\delta-1.
	\end{equation} 
	
	To control the behavior of $\beta_1^O(x') $, we do not  quote  \cite{DaleBook} or \cite{CosToho}: we give a new argument using Giraud's machinery of \cite{GiraudMaxPos}: 
	\begin{Rk}
	Naively, the {\em ridge} ({\em fa\^\i te} in French) of an homogeneous Ideal  $I\subset k(x)[Y_1,\cdots,Y_r]$ is the biggest group  of translations of 
	$ \Spec ( k(x)[Y_1,\cdots,Y_r]) \cong \A^r_{k(x)}$ 
	leaving stable $ \Spec({k(x)[Y_1,\cdots,Y_r] / I})$, see the precise definition in  \cite{GiraudMaxPos}~1.5, or  \cite{GiraudEtude} D\'efinition~5.2 p.~I.24, or \cite{ComputeRidge}.
	\end{Rk}
	
	\begin{Obs}\label{Obs:Ridge}  
			Let us denote by $\mathrm{Rid}_x^O(X)$ 
			the ridge of the homogeneous ideal 
			$\mathcal{I}:= \langle F_1,\cdots,F_m \rangle \subset k(x)[Y_1,\cdots,Y_r]$. 
			(Recall that $ F_i = in_M (f_i) $, for $ 1 \leq i \leq m $).
			By \cite{GiraudMaxPos}~3.1 and 3.3, there exist polynomials 
			$$ 
				P_1,\ldots,P_s 
				\in k(x) \left[ \, X_{A(1),1}, \ldots, X_{A(m),m} \, \Big|\,  A(i) \in \IZ_{\geq 0}^r : |A(i)|< \nu_i, \, \mbox{for } 1 \leq i \leq m \, \right]
			$$ 
			which are homogeneous if we give to $X_{A(i),i}$ the degree $ \nu_i-\vert A(i) \vert$ and such that:
		\begin{enumerate}
			\item[(i)] 
			For simplicity, we abbreviate
			$$
				P_j(\mathrm{D}_A^{Y}(F_i)):= P_j \left( \, \mathrm{D}_{A(1)}^{Y}(F_1), \ldots, \mathrm{D}_{A(m)}^{Y}(F_m) \, \Big|\, A(i)  : |A(i)| < \nu_i,  1 \leq i \leq m \,\right)
			$$
			in the following.
			The elements
			$$	
			\sigma_j := P_j(\mathrm{D}_A^{Y}(F_i)), 
			\hspace{10pt}
			1\leq j \leq s, 
			$$
			generate the ideal of the ridge in $k(x)[Y_1,\cdots,Y_r]$, where $\mathrm{D}_A^{Y}$ is a differential operator defined by Taylor's formula \cite{GiraudMaxPos}~2.5 (nowadays, this is  called a Hasse-Schmidt derivation). For any $F\in S[Y_1,\cdots,Y_r]$, $S$ any commutative ring, any $A\in \N^r$, we define $ \mathrm{D}_A^{Y}(F)\in S[Y_1,\cdots,Y_r]$ by
			$$
			F(Y+Z)=:\sum_{A\in \N^r}   \mathrm{D}_A^{Y}(F)\, Z^A \in S[Y_1,\ldots,Y_r,Z_1,\ldots,Z_r].$$
				
			\item[(ii)] 
			Up to linear changes on the set of variables $ ( Y_1,\ldots,Y_r ) $, we have:
			$$ 
			\sigma_j = Y_j^{q_j}+\mu_{j+1,j}\, Y_{j+1}^{q_j}+\ldots+\mu_{r,j}\,Y_r^{q_j}, 
			$$
			$$ \mbox{for } \,  
			\mu_{i,j} \in k(x),\,\, 
			1\leq j \leq s, \,\, 
			j+1\leq i \leq r, \,\,
			\mbox{ and }
			 q_1\leq q_2 \leq \ldots \leq  q_s.
			 $$
			Moreover, the $\sigma_j$ are {\em additive} polynomials, i.e., each $q_j$ is a power of the characteristic when char$(k(x))=p>0$; 
			in characteristic $0$, the $\sigma_j$ are linear.
		\end{enumerate}
	\end{Obs}

	Let us note that 
	$$
	\Dir_x (X) \subset \mathrm{Rid}_x(X)\subset  C_x(X) \subset T_x(X) 
	$$
	(and, of course, the statement holds also in the $ O $-setting).
	In the case we are concerned, the directrix and the tangent cone have both dimension \rot{two}.
	Hence the ridge and the directrix have both dimension \rot{two}: the number of $\sigma_j$ is equal to $r$, $ s = r $. 
	(In general, it may be strictly smaller).
	
	Let 
	$ S_j := P_j(\mathrm{D}_A^{Y'}( F'_{i,\nu} )  )$, $1\leq j \leq r$  ($ F'_{i,\nu}= in_{\nu}(f'_i) $ \eqref{eq:F_i_nu}). 
	Then $ S_j $ is homogeneous for the valuation $\nu$ defined in \eqref{eq:def_val_nu}, $\nu(S_j)=q_j$.
	We have 
	$$
	S_j = 
	{Y_j'}^{q_j}
	+ \mu_{j+1,j} \,{Y'_{j+1}}^{q_j}
	+ \ldots
	+ \mu_{r,j}\, {Y'_r}^{q_j} 
	+ \sum_{\vert B \vert < q_j} 
	S_{j,B}(\overline{u'_2})
	\,
	U_1^{(q_j-\vert B \vert)(\delta -1)} 
	\, {Y'}^B ,
	$$
	for $ S_{j,B} : = S_{j,B}(\overline{u'_2}) \in k(x)[\overline{u'_2}], $
	$ \deg_{\overline{u'_2}}(S_{j,B})\leq (q_j-\vert B \vert) \cdot \beta $. 
	For simplicity, we write $q:=q_r$. 
	We consider the quasi-homogeneous ideal 
	$$
	\mathcal{J}:= 
	\big\langle \, S_1^{ \frac{q}{q_1} }, \, S_2^{ \frac{q}{q_2} }, \, \ldots, \, S_r \,\big\rangle 
	\in {gr}_{\nu}(R[y',u_2'])\cong 
	k(x)[Y',U_1,\overline{u'_2}].
	$$ 
	We can find generators $ ( G ) := ( G_1,\ldots, G_r)$ of this ideal with
	\begin{equation}
	\label{eq:def_G_j}
	G_j = 
	{Y'_j}^{q}+\sum_{\vert B \vert <\rot{q}} G_{j,B}(\overline{u'_2}) \, U_1^{(q-\vert B \vert)(\delta -1)} \, {Y'}^B ,
	\hspace{10pt}
	\ 1\leq j \leq r,
	\end{equation}
	where $ G_{j,B}\in k(x)[\overline{u'_2}] $ and $ \deg_{\overline{u'_2}}(G_{j,B})\leq (q-\vert B \vert)\cdot \beta $.
	
	Recall the notations of \eqref{eq:parameters_x'} and \eqref{eq:F_i_nu}.
	We set 
	$$ 
		\Psi := \phi(1,\overline{u'_2}) 
	\hspace{10pt} \mbox{ and } \hspace{10pt} 
		(F'_\nu) := ( F'_{1,\nu}, \ldots, F'_{m,\nu}) .
	$$ 
		
	Consider the localization of $ gr_{\nu}(R[y',u_2'])$ at the ideal $ \langle Y', U_1, \Psi \rangle $ corresponding to $ x' $. 
	In there, we have:
	\begin{equation}
	\label{eq:polyedra_inclusion}
		\poly{G}{U_1,\Psi}{Z}\subset
		\poly{F'_\nu}{U_1,\Psi}{Z} \subset 
		(\delta-1,0)+\R_{\geq 0}^2,
	\end{equation}
	where $ ( Z ) = ( Z_1, \ldots, Z_r ) $ is defined by $Z:=Y'-U_1^{\delta-1}\Theta$,
	(by abuse of notation we use the latter expression when we mean $ Z_j := Y'_j -U_1^{\delta-1}\Theta_j $) with 
	$\Theta_j\in k(x)[\overline{u'_2}]_{\Psi} $, 
	$1\leq j \leq r$.
		Suppose we make a translation on $ (y') $ to solve the vertex of abscissa $\delta -1 $ of 
	$\poly{f'}{u'_1,\psi}{y'}$, 
	say $ y'\mapsto z :=y'-u_1^{\delta-1}\theta $, for some $ \theta_j $ such that we have $\Theta_j \equiv \theta_j$ modulo $ \langle y',u_1 \rangle $, $1\leq j \leq r$.
	Then the vertex of abscissa $\delta-1$ is the same for 
	$ \poly{F'_\nu}{U_1,\Psi}{Z} $ and $ \poly{f'}{u'_1,\psi}{z}$.
	
		
	\begin{itemize}
		\item  \emph{Case 1}.
		For {all} $j \in \{ 1, \ldots, r \} $,
		$ G_j = {Y'_j}^q-U_1^{q(\delta-1)}\Theta_j^q$, $\Theta_j\in k(x)[\overline{u'_2}]$.
	\end{itemize}
	We claim that $\Delta(F'_{\nu};U_1,\Psi;Z)$ is minimal for all 
	$ \Psi = \phi(1,\overline{u'_2})$ (for all $x'$) 
	where $ Z := Y'-U_1^{\delta-1}\Theta $.
	Moreover, the ordinate of its unique vertex is  $\beta_1^O(x')$ and 
	$$
		\beta_1^O ( x' )\leq \frac{\beta_1^O ( x ) }{ d},
	$$ 
	where $ d := [k(x'):k(x)] $.
	Since $ ( f\rot{, u, y} ) $ is well-prepared it is, in particular, normalized at every vertex of the associated polyhedron (Definition \ref{Def:initial_vertex}).
	This and \eqref{eq:F_i_nu} imply that $ ( F'_\nu ) $ is normalized at the unique vertex of $\poly{F'_{\nu}}{U_1,\Psi}{Z}$.
	Let $(\delta-1,w)$ be this vertex and suppose it is solvable: 
	This means we could make a translation on $Z$, $Z':=Z-U_1^{\delta-1}\Psi^w$ to solve it.  
	Since $ G_j = Z_j^{q} $, for all $ 1 \leq j \leq m $, we obtain that $(\delta-1,w)$ would be the unique vertex of 
	$ \poly{G}{U_1, \Psi}{Z'} $.
	In particular,
	$ \poly{G}{U_1, \Psi}{Z'} \not\subset
	  \poly{F_{\nu}}{U_1,\Psi}{Z'} $ 
	 which contradicts \eqref{eq:polyedra_inclusion}.
	The reader ends the proof of the claim as in the case $ \delta \notin \IZ_+ $.

	
	\begin{itemize}
		\item  \emph{Case 2}. 
		For {some} $j \in \{ 1, \ldots, r \} $,
		$ G_j = {Y'_j}^q + \sum_{\vert B\vert <q} G_{j,B}(\overline{u'_2}) \, U_1^{(q-\vert B \vert)(\delta -1)} \, {Y'}^B$ 
		with $G_{j,B}\not=0$ for at least one $B\not=(0,\cdots,0)$.
	\end{itemize}
	Then, for any $Z:=Y'-U_1^{\delta-1}\Theta$, the ordinate of the vertex of 
	$ \poly{G}{ U_1, \Psi}{Z} $
	is smaller or equal than $\frac{\ord_{\Psi} (G_{j,B})}{q-\vert B \vert}$, with $\vert B \vert$ maximal such that $G_{j,B}\not=0$.
	As we have $ \deg_{\overline{u'_2}}(G_{j,B})\leq (q-\vert B \vert) \cdot \beta $, we get
	$$
		\beta_1^O ( x' )\leq \frac{\beta_1^O ( x )}{d}.
	$$ 


	\begin{itemize}
		\item \emph{Case 3}. 
		None of before: 
		for {all} $j \in \{ 1, \ldots, r \} $,
		$ G_j={Y'_j}^q-U_1^{q(\delta-1)}G_{j,0}$ and, 
		for at least one $j$, say $j_0$, $ G_{j_0,0} = G_{j_0,0} (\overline{u_2'} ) $ is not a $q$-power.
	\end{itemize}
	Let $ b:= \sup\{\ c \mid   \exists \, S \in k(x')[\overline{u'_2}],\ G_{j_0,0}={S}^{p^c} \, \}$. 
	Set $ S :=\sqrt[p^b]{ G_{j_0,0} }$, 
	$ S \in k(x)[\overline{u'_2}]$.
	Note that $ S $ has degree $\leq {q \cdot \beta \over p^b} $.
	As in \eqref{eq:polyedra_inclusion}, we consider some $ Z := Y' - U_1^{\delta - 1 } \Theta_ j $, for $ \Theta_j \in k(x)[\overline{u_2'}]_\Psi $.
	Now, the argument is classical: 
	there exists a derivation $ D $, 
	$D\in \mathrm{Der}_{k(x)/\F_p}$ or $D=\frac{\partial}{\partial \overline{u'_2}}$ such that 
	$ DS = D(S+\Theta_{j_0}^{q\over p^b})
	\not= 0 $. 
	When $x'$ is not rational, call $ \beta'_{\nu}$ the ordinate of the vertex of   $\poly{G}{U_1,\Psi}{Z}$.
	We have
	\begin{equation}
	\label{eq:inequi_beta_nu} \left\{ \hspace{20pt}
	\begin{array}{ll}
		\dfrac{q}{p^b} \cdot \beta'_{\nu}
		& \leq   \ord_{\Psi} ( S + \Theta_{j_0}^\frac{q}{p^b}) 
		\leq
		\\
		&\leq  	1 + \inf\left\{ \ord_{\Psi} (DS) \mid D\in \mathrm{Der}_{k(x)/\F_p}
		\mathrm{or\ } D=\frac{\partial }{ \partial \overline{u'_2}} \right\} 
		\leq
		\\
		&\leq 1+ \dfrac{q}{p^b} \cdot \dfrac{\beta}{d} \, ,
	\end{array} \right.
	 \end{equation}
	where $d$ is the degree of the residual extension. 
	As $ \frac{q}{p^b}\geq p$,
	we get 
	$\beta_\nu' \leq \frac{1}{p} + \frac{\beta}{d}$. 
	Since $ \beta  > 1 $ (and $ p \geq 2 $) this leads to $\beta_1^O(x') \leq \beta'_{\nu}< \beta = \beta_1^O(x)$ when $ d > 1 $.

	Hence,  $\beta_1^O(x')=\beta_1^O(x)$ implies that  $x'$ is rational over $x$: 
	we can choose $\lambda\in R$, $\lambda$ invertible or zero such that $(u_1, \frac{u_2}{u_1}-\lambda,  \frac{y}{u_1})$ are parameters of $x'$. 
	\rot{Replacing the \blau{variable}
	 $u_2$ by  $ u_2-\lambda u_1$, we may suppose that 
	$ x' $ is the origin of the $ U_1 $-chart, 
	i.e., the point of parameters $(u',y') :=(u_1, \frac{u_2}{u_1}  \frac{y}{u_1} )$.}
	
	By Proposition~\ref{Prop:originblowup}, \rot{$\poly{f'}{u'}{y'}$} is minimal, where $f'_i:=f_i/u_1^{\nu_i}$, $1\leq i \leq m$. The computation of the transform shows that, with natural notations, 
	$$
	 \rot{\poly{f'}{u'}{y'}}
	= \mathrm{Conv}\left\{
	\bigcup_{(v_1,v_2)\in \poly{f}{t}{y}}(v_1+v_2-1,v_2)+\R_{\geq 0 }^2\right\}.
	$$ 
		
	In a few words, the side of slope $-1$ of  \rot{$\poly{f}{u}{y}$} becomes vertical, vertices move horizontally, and 
	$ \beta_1^O ( x')$ is the ordinate of the vertex of smallest ordinate of  the side of slope $-1$ of  $\poly{f}{u}{y}$. 
	We have a decrease except in the extreme case where there is only one vertex on  the side of slope $-1$ of  \rot{$\poly{f}{u}{y}$}.
	That vertex is $ (\alpha , \beta ) = ( \alpha_1^O ( x), \beta_1^O ( x ))$, and $\gamma_1^O ( x )=\beta_1^O ( x )$. 
	
	As a consequence, 
	$$
	in_{\delta}(f_i) =F_i(Y)+\sum_{\vert B \vert < \nu_i} \lambda_{B,i} \, U_1^{\alpha(\nu_i-\vert B \vert )} U_2^{\beta(\nu_i-\vert B \vert )}Y^B ,\hspace{10pt}  \lambda_{B,i}\in k(x).
	$$
	By the following Lemma \ref{Lem:beta=}(2), there is at most one $x'$ $O$-near to $x$ such that  $ \beta_1^O (x') = \beta_1^O ( x ) $.
	We call it $x'_{bad}$.

	Furthermore,  by \rot{Lemma} \ref{Lem:beta=}(1), 
	we can choose  \rot{$u_2$} such that $ {\slope_1^O}{ ( \rot{u_2}) }= \sigma_1^O ( x )$,  \rot{$ u_2 \in R $} when $ \sigma_1^O ( x )<\infty$, and $ \rot{u_2}\in \hR$ when $s_1^O ( x ) = \infty$.
	The last case is impossible:
	Indeed, $ \sigma_1^O ( x ) = \infty $ would imply $ \beta_1^O ( x ) \geq 1$ by definition of  $ \sigma_1^O ( x )$, see \eqref{eq:def_sigma} 
	In this situation, the characteristic polyhedron  \rot{$\cpoly{J^O\hR}{u_1,u_2}$} would have only one vertex, namely $ ( \alpha_1^O ( x ), \beta_1^O ( x ) ) $, and $ \beta_1^O ( x) = \alpha_2^O (  \cpoly{J^O\hR}{u_1, \rot{u_2}}) \geq 1$. 
	Then $ V( \rot{u_2}, y)$ would be permissible which contradicts the fact that the maximal (log)-Hilbert-Samuel locus is a closed point.

	To end the proof, the inequality 
	$\iop ( X', \cB',  x' )< \iop ( X, \cB, x )$ is clear except if $x'_{bad}$ exists.  
	In that case, we get
	\begin{equation}
	 \label{blabla}
	( \,\beta_1^O(x'_{bad}) , \, \gamma_1^O(x'_{bad}),\, s_1^O{(\tfrac{\blau{u_2}}{u_1})}, \,  
	\alpha_1^O (x'_{bad}) \,)
	\leq_{lex}  
	(\beta,\,\beta,\,s_1^O{(\blau{u_2})} - 1 ,\delta -1).
	\end{equation}

	When $\beta = \beta_1^O ( x)<1$,  we have $ \sigma_1^O ( x ) = 1 = \sigma_1^O ( x'_{bad})$, by definition.  
	The following inequality yields the result: 
	$$ 
	\alpha_1^O ( x'_{bad} ) = \alpha_1^O ( x) + \beta_1^O ( x)  - 1 < \alpha_1^O ( x), 
	\hspace{10pt} \hbox{when} \ \beta_1^O ( x) < 1.
	$$
	
	If  $\beta_1^O ( x)\geq 1$ then $  \sigma_1^O ( x )={\slope_1^O}{ (\blau{u_2}) } >1$.
	We claim that  ${\slope_1^O}{ ( \frac{\blau{u_2}}{u_1}) }=\sigma_1^O(x'_{bad})$. 
	Suppose there exist $z'_1,\cdots,z'_r,v'\in R[ \frac{y}{u_1},\frac{u_2}{u_1}]_{\langle  \frac{y}{u_1},\frac{u_2}{u_1} \rangle }$ such that 
	${\slope_1^O}{ (v') } > {\slope_1^O}{ ( \blau{\frac{u_2}{u_1}}) } $ 
	and $\cpoly{{J^O}'}{u'_1,v'}=\poly{\tilde{f'}}{u'_1,v'}{z'}$, for some prepared basis $ (\tilde{f'}) $ of $ {J^O}' $.
	Then the curve $V(z',v')$ transversal to $ \mathrm{div}(u_1)$ projects on some curve $V(z,v) \subset \Spec (R) $ transversal to div$(u_1)$ and  
	$ {\slope_1^O}{ (v) } = {\slope^O_1}{ (v') } + 1 > {\slope^O_1}{ (\blau{u_2}) } $.
	This is a contradiction to $ {\slope^O_1}{ (\blau{u_2}) } = \sigma_1^O ( x ) $.
	
	Up to a permutation of $u_1 $ and $ u_2$ in the case where $E=\mathrm{div}(u_1u_2)$ when there are two boundary components, this ends the proof.
\end{proof}

\begin{Lem}\label{Lem:beta=} We use the notations above. 
	\begin{enumerate}
		\item[(1)]
		Assume that, for all $i \in \{ 1, \ldots, m \} $, 
		\begin{equation}
		\label{eq:bad}
		in_{\delta}(f_i) 
		=
		F_i(Y) +
		\sum_{\vert B \vert < \nu_i} \lambda_{B,i} \, U_1^{\alpha(\nu_i-\vert B \vert )}\, U_2^{\beta(\nu_i-\vert B \vert )}\,Y^B ,
		\end{equation}
		for $ \lambda_{B,i}\in k(x) $.
		Let $v_2\in R$ be such that $(u_1,v_2, y)$ is a {\RSP} for $ R $ and 
		$$ 
			v_2 \, \equiv \,  U_2-\lambda U_1 \mod \langle y,u_1,u_2 \rangle^2,
			\hspace{10pt}		
			 \lambda \in k(x)^\times.
		$$
		Then $\cpoly{J^O}{u_1,v_2}$ has two distinct vertices on the side of slope $-1$,
		i.e., $ {\slope_1^O}{ (v_2) } = 1 $.
		
		\item[(2)] After blowing up the origin $x$, there is a strict decrease of $ \beta_1^O $ at every point very near to $x$ except maybe for the point $x'_{bad}$ of parameters $( \frac{y}{u_1},u_1, \frac{u_2}{u_1})$ (assuming it is $O$-near to $x$).
	\end{enumerate}

\end{Lem}

\begin{proof}
	Assertion (2) is a consequence of (1). 
	Assertion (1) is clear in the easy case, 
	where $\delta(x)\not\in \IZ_+$.
	
	By \eqref{eq:bad}, the definitions \eqref{eq:F_i_nu}, \eqref{eq:def_G_j}  become:
	\begin{equation}
	F'_{i,\nu}= 
		F_i(Y') + \sum_{\vert B \vert< \nu_i} 
		 \lambda_{i,B}\,\overline{u'_2}^{(\nu_i-\vert B \vert)\beta}
		  \,{U_1}^{(\nu_i-\vert B \vert)(\delta-1)} \,{Y'}^B,
	\end{equation}
	%
	%
	\begin{equation}
	\label{eq:G_j_beta=}
	G_j = {Y'_j}^{q}+\sum_{\vert B \vert <q_j} \mu_{j,B}\,\overline{u'_2}^{(q-\vert B \vert)\beta} \, U_1^{(q-\vert B \vert)(\delta -1)} \, {Y'}^B ,
	\end{equation}
	with $ \mu_{j,B}\in k(x),\ 1\leq j \leq r $.
	We come back to the three sub-cases of $\delta\in \IZ_+$.
	
	
	\begin{itemize}
		\item  \emph{Case 1}.
		For {all} $j \in \{ 1, \ldots, r \} $,
		$ G_j = {Y'_j}^q-U_1^{q(\delta-1)}\Theta_j^q$, $\Theta_j\in k(x)[\overline{u'_2}]$.
	\end{itemize}
	Then $\Theta_j=0$ or $\Theta_j=\gamma_j \overline{u'_2}^{\beta}$, for some $\gamma_j\in k(x)$.
	With the previous notations
	$$
	F'_{i,\nu}=F_i(Z)+ \sum_{\vert B \vert< \nu_i} \lambda'_{i,B}\overline{u'_2}^{(\nu_i-\vert B \vert)\beta} {U_1}^{(\nu_i-\vert B \vert)(\delta-1)} Z^B , \hspace{10pt} \lambda'_{i,B}\in k(x).
	$$
	By Case 1  above, $\poly{F'_{\nu}}{U_1,\Psi}{Z} $ is minimal for all $ \Psi = \phi(1,\overline{u'_2})$ (for all $x'$). 
	Hence 
	$\beta_1^O( \poly{F'_{\nu}}{U_1,\Psi}{Z} ) = 0 $,
	for all $x'\not=x'_{bad}$. 
	Suppose (1) is  wrong.
	Then there would exist some $ x' \not=x'_{bad}$ with 
	$\beta_1^O( \poly{F'_{\nu}}{U_1,\Psi}{Z} ) = \beta_1^O(x) > 1 $. 
	A contradiction.
	
		
		\begin{itemize}
			\item  \emph{Case 2}. 
			For {some} $j \in \{ 1, \ldots, r \} $,
			and for some $B\not=(0,\cdots,0)$,  $\mu_{j,B}\not=0$ in \eqref{eq:G_j_beta=}.
		\end{itemize}
	Then, as   $ G_{j,B}= \mu_{j,B}\overline{u'_2}^{(\nu_i-\vert B \vert)\beta }$,
	we get 
	$$
		\beta_1^O(\poly{F'_{\nu}}{U_1, \Psi }{Z})
		\leq 
	 \frac{\ord_{\Psi} (G_{j,B})}{ q-\vert B \vert} = 0,
	$$
	with $\vert B \vert$ maximal such that $G_{j,B}\not=0 (\Leftrightarrow \mu_{j,B}\not=0)$. 
	We conclude as above.
	
		
		\begin{itemize}
			\item \emph{Case 3}. 
			None of before: 
			for {all} $j \in \{ 1, \ldots, r \} $,
			$ T_j={Y'_j}^q-U_1^{q(\delta-1)}\cdot \mu_{j,0} \cdot \overline{u'_2}^{q\beta_1^O }$ and, 
			for at least one $j$, say $j_0$, 
			$\mu_{j_0,0}\cdot \overline{u'_2}^{q\beta_1^O}  $ 
			is not a $q$-power.
		\end{itemize}
	Then $ S $ of the proof in Case 3 before is a monomial and the inequalities
	\eqref{eq:inequi_beta_nu} become:
	$$
	\frac{q}{p^b}  \cdot \beta'_{\nu}
	\leq 
	\ord_{\Psi}( S + \Theta_{j_0}^{\frac{q}{p^c}} ) 
	\leq 
	1 + \inf \left\{ \ord_{\Psi} DS \mid \ D\in \mathrm{Der}_{k(x)/\F_p}\ \mathrm{or\ } D=\frac{\partial}{\partial \overline{u'_2}} \right\} 
	= 1.
	$$
	As $\frac{q}{p^b}\geq p$, we get $\beta'_{\nu}\leq \frac{1}{p} < 1 < \beta = \beta_1^O(x)$ and we conclude as above. 
\end{proof}




\begin{Rk}
	As one sees from the proof, the role of $ \gamma_1^O ( X, x ) $ in the invariant is only to indicate if we have $ \beta_1^O ( X, x) > \gamma_1^O ( X, x ) $ or $ \beta_1^O ( X, x) = \gamma_1^O ( X, x ) $.
	Therefore one could alternatively replace $ \gamma_1^O ( X, x ) $ by
	$$
		\widetilde{\gamma}_1^O ( X, x ) := \left\{
			\begin{array}{cl}
				1,
				& \hspace{10pt}
				\mbox{ if } \beta_1^O ( X, x) = \gamma_1^O ( X, x ),
				\\[5pt]
				0	,
				&\hspace{10pt}
				\mbox{ if }	\beta_1^O ( X, x) > \gamma_1^O ( X, x )	.
			\end{array}					
		\right.
	$$
\end{Rk}

%
%
%
%
%
%
%
%
%
%
%
%
%
%
%
%


In Example \ref{Ex:easy_difference} we have seen that the algorithm of \cite{CJS} differs in characteristic zero from the usual one.
Further, we explained some pathologies that appear in dimension three in Examples \ref{Ex:dimensionthree} and \ref{Ex:dimensionthree_BM}.

The following {example illustrates} that the dimension of the directrix and $ \beta_1^{O}(X,x) $ are not upper semi-continuous.
Therefore our local invariant $ \iota(X,\cB, x) $ is not upper semi-continuous.

\begin{Ex}
\label{Ex:e_not_usc}
Suppose there is no boundary $ \cB = \emptyset $.
Consider the singularity $ X := \Spec ( k[u_1, u_2, y ]/ h ) $ given by the polynomial
$$ 
	h:= y^p+ \lambda u_1^p \in k[u_1,u_2, y], \hspace{10pt} \lambda\not\in k^p,
$$ 
where $k$ is a {\it non-perfect field} of characteristic $ p > 0 $. 
{Since $ X $ is a hypersurface the maximal Hilbert-Samuel locus coincides with the maximal order locus.
	Here, this is the locus of order $ p $ which is the line $L:=V(u_1, y)$.}
By the \cite{CJS}-algorithm, the center will be $ L $.

The dimension of the directrix is not stable along $ L $:
$e_X(x)=e_X^O(x) = 1 $ when $\lambda$ is not a $p$-power modulo the maximal ideal of $x \in L $.
On the other hand, it is $2$ if $\lambda$ is a $p$-power modulo the maximal ideal at $x$.

Hence the third part $\iota_{poly}$ of our invariant is also not stable along $L$:
  $\iota_{poly}(x)=(0,0,0,\delta_X(x))$ in the first case, $\iota_{poly}(x)=(\infty,\infty,\infty,\infty)$ in the second case. 
  This example shows that the locus where our invariant is maximum along $L$ may not even be constructible.    
 
  Moreover, when $\lambda$ is not a $p$-power in the residue field $k(x) $, the characteristic polyhedron is empty. 
  {Thus, $ \beta_1^O (X,x) = \infty $.}
  Suppose $\lambda$ is  a $p$-power in $ k(x) $.
	{For example, consider the point $ x $ with coordinates $ ( y, u_1, \phi) $, where $ \phi := \lambda - u_2^p $.
		Then, locally at $ x \in L $, we have
		$$ 
		h= z^p + \phi u_1^p \in k[u_1,u_2,y]_{\langle u_1,\phi, z \rangle},
		\hspace{20pt}
		 z := y +  u_1 u_2.
		$$	}
		 So, $\Delta(h;u_1,\phi)=(1,\frac{1}{p})+\R_{\geq 0}^2$ which yields 
		  $ \alpha_1^O(X,{x})=1$, $ \beta_1^O(X,{x})= {1 \over p}<1$, and $ \sigma_1^O(X,x)=1$.  	
 
 %
\smallskip
The main trouble in the example is that over a non perfect-field $ k $ the directrix may change if we pass to an extension of $ k $.
Thus all the data obtained from it may change. 

\smallskip 

{
Now change slightly the equation:
$$
 h := y^p+ \lambda u_1^{2p} \in k[u_1,u_2, y], \hspace{10pt} \lambda\not\in k^p.
$$
By similar computations, we get
$e_X(x)=e_X^O(x) = 1 $  along $ L = V(u_1, y) $ and
$ \beta_1^O (X,x) =  \beta_1 (X,x) = 0$ when $\lambda$ is not a $p$-power modulo the maximal ideal of $x \in L $.
On the other hand, it is $>0$ if $\lambda$ is a $p$-power modulo the maximal ideal at $x$.
The number $ \beta_1 (X,x)$  may change  if we pass to an extension of $ k $ and $\beta_1$ is not upper semi-continuous along $L$.}
\end{Ex}

\section{Termination of the Cossart-Jannsen-Saito algorithm}
\label{sec:finite}

Let $ X $ be a reduced Noetherian excellent scheme of dimension at most two.
	Recall that $ X^O_{max} = \bigcup_{\tnu \in \sum_X^{max,O}} X^O (\tnu) $ denotes the maximal log-Hilbert-Samuel locus of $ X $ and $  X^O (\tnu) $ are the disjoint maximal log-Hilbert-Samuel strata, for $ \tnu \in \sum_X^{max,O} $ (Definition \ref{Def:logHS_function}(2)).
	The goal of this section is to explain

\begin{Rk} 
	\label{Rk:finiteness}
Theorem \ref{MainThm} provides a short and direct proof for the finiteness of the sequence of permissible blow-ups constructed in \cite{CJS} to resolve the singularities of $ X $.
\end{Rk}

	{Let us adapt an extremely simple idea by Abhyankar in our context: in dimension two, one has to look at only a finite number of points}

\noindent  - the generic points of components of dimension one of $X^O_{max} $,

\noindent  - a finite number of closed points called ``bad points''.

\begin{Def}
	\label{Def:good-bad}
	Let $ x \in X^O_{max} $ be a closed point.
	In the following cases, $ x $ is called a {\em bad point}: 
	\begin{itemize}
		\item 
		if $ x $ is isolated in the log-Hilbert-Samuel locus, or
		
		\item 
		if $ x \in C $, for some component $ C \subset X^O_{max} $ of dimension one, 
		and the algorithm restricted at $ x $ does not coincide with the algorithm restricted at the generic point of $ C $.
	\end{itemize}
	Otherwise, $ x $ is called a {\em good point}.
\end{Def}

\begin{Prop}
	\begin{enumerate}
		\item[(1)] 
			There are only finitely many bad points.	
			
		\item[(2)]
			The \cite{CJS}-algorithm terminates over every point.
	\end{enumerate}

\end{Prop}

\begin{proof}
	There are finitely many isolated points in the maximal log-Hilbert-Samuel locus of $ X $. 
	Hence the statement is true for bad points of the first type.
	
	Let $ C \subset X^O_{max} $ be a component of dimension one. 
	For the other kind of bad points, 
	the algorithm will blow up a finite number of closed points, say
	$$ 
		\{x_{0,1},\ldots,x_{0,n_0}\},
	$$ 
	of $ C $ and perhaps will blow up some centers projecting on  these closed points before a possible blowing up along the strict transform of $C$: $ \{x_{0,1},\ldots,x_{0,n_0}\}$ are bad points. These points are 
the singular points of $C$ and some of the points where $C$ intersects  the boundary.
	When \cite{CJS} never blows up along the strict transform $ C' $ of $ C $,  $\{x_{0,1},\ldots,x_{0,n_0}\}$ are the  bad points of $C$ (in fact, this case will not exist,  $C'$ will be always blown up).

	Eventually, \cite{CJS} will blow up along the strict transform $ C' $ of $ C $.
	Let  $X_1$ be the strict transform of $X$ just after the blowing up of $C'$ and let $ \cB_1 $ be the corresponding boundary.	
	By Theorem \ref{Thm:HiroMizutani$^O$}, the components of the maximal log-Hilbert-Samuel locus lying above $ C' $ in $X_1$   and
	with the same value of the log-Hilbert-Samuel function 
	are  of the following form
	(see Observation \ref{Obs:GooOnear}):
	\cyan{either  a  finite number of isolated closed points or one component $ C_1 $ of dimension one.}

	\cyan{Suppose we are in the first case. 
	 Let  $\{x_{1,1},\ldots,x_{1,n_1} \}\subset X_1$ be the closed points with the same value of the log-Hilbert-Samuel function. They are bad points.
	 \rot{If we denote by $\{x_{1,1}',\ldots,x_{1,n_1}' \}$ their projections on $C$, then}
	   $$ \{x_{0,1},\ldots,x_{0,n_0}\} \cup \{x_{1,1}',\ldots,x_{1,n_1}' \}$$
	 is the set of bad points of $C$. }

	\cyan{In the second case, let $\eta $ (resp.~$ \eta_1 $) be the generic point of $ C $ (resp.~$ C_1 $).
	By Theorem \ref{MainThm}, we have 
	\begin{equation}
		\label{eq:inequality}
		\iota(X_1,\cB_1,\eta_1) < \iota(X,\cB,\eta).
	\end{equation}
	(Note that the value of the invariant at a generic point $ \eta' $ of $ C' $ coincides with $ \iota(X,\cB,\eta) $
	and $ \eta_1 $ lies above $ \eta' $).
		 By a descending induction on the value of our invariant $\iota(X,\mathcal{B},.)$, the set of bad points lying on $ C_1 $ is finite.
		 These bad points $\{x_{1,1},\ldots,x_{1,n_1} \}$ project on bad points of $ C $.
		Let $\{x_{1,1}',\ldots,x_{1,n_1}' \}$ be their projections on $C$. As before
	 $$ \{x_{0,1},\ldots,x_{0,n_0}\} \cup \{x_{1,1}',\ldots,x_{1,n_1}' \}$$
	 is the set of bad points of $C$. }
	  Therefore, (1) is proven.
	 
	 Let us prove (2). 
It remains to show the termination of the algorithm. If $\eta $ is the  generic point of $ C \subset X^O_{max} $  a component of dimension one, the result is clear by \eqref{eq:inequality}. 
If $ x \in X^O_{max} $ is a closed point, $x$ is good or bad. When $x$ is a good point, $x\in C$, as the resolution process above $ C $ is finite, it is the same above $x$, by definition of a good point.
	 When $x$ is a bad point, by the previous arguments, the \cite{CJS}-algorithm either terminates or, after finitely many steps, we obtain an isolated closed point $ x_1 \in X_1 $ that lies above $ x $ and with $ H_{X_1}^O(x_1) = H_X^O (x) $.
	By Theorem \ref{MainThm}, we have
	$$
		\iota(X_1,\cB_1, x_1) < \iota(X,\cB,x),	
	$$
	and, by induction, the assertion follows. 
\end{proof} 	

 Since the algorithm is finite over all points, the Noetherianity ends the proof of Remark \ref{Rk:finiteness}.

\subsection*{Acknowledgements}
The second author would like to thank Uwe Jannsen for many discussions and explanations on the original work \cite{CJS}.
Further, he is grateful to the Erwin Schroedinger Institute in Vienna for their support and hospitality during the Research in Teams {\em Resolution of Surface Singularities in Positive Characteristic} in Vienna in November 2012.
He thanks the other participants -- Dale Cutkosky, Herwig Hauser, Hiraku Kawanoue and Stefan Perlega -- for the things he learned during this time from them and for all their questions.

Both authors thank Edward Bierstone for valuable comments.
They send their warm thanks and compliments to the referee: her/his accurate reading, her/his numerous suggestions helped them to give a much more precise and pedagogical redaction.

Bernd Schober was supported by the Emmy Noether Programme ``Arithmetic over finitely generated fields" (Deutsche Forschungsgemeinschaft, KE 1604/1-1) and by Research Fellowships of the Deutsche Forschungsgemeinschaft (SCHO 1595/1-1 and SCHO 1595/2-1)


\begin{thebibliography}{BHM}



\bibitem[A]{AbhyDim2}
S.S.~Abhyankar,
\newblock {\em Resolution of singularities of embedded algebraic surfaces},
\newblock 
second edition, {Springer Monographs in Math.}, Springer Verlag (1998).

\bibitem[BV]{AngOrlDim2}
A.~Benito, and O.~Villamayor,
\newblock {Techniques for the study of singularities with applications to resolution of 2-dimensional schemes},
\newblock {\em Math. Ann.} {\bf 353} (2012), 1037--1068.


\bibitem[BHM]{ComputeRidge}
J.~Berthomieu, P.~Hivert, and H.~Mourtada,
\newblock {Computing {H}ironaka's invariants: ridge and directrix},
\newblock in {\em Arithmetic, geometry, cryptography and coding theory 2009}, 9--20,
  {\em Contemp. Math.} {\bf 521}, Amer. Math. Soc. (2010).
  
\bibitem[BM]{BM}
 E.~Bierstone and P.~Milman,
 \newblock {Canonical desingularization in characteristic zero by blowing up the
 	maximum strata of a local invariant},
 \newblock {\em Invent. Math.} {\bf 128} (1997), 207--302. 

\bibitem[C1]{CosToho}
V.~Cossart,
\newblock {Desingularization of embedded excellent surfaces},
\newblock {{\em Tohoku Math. J.}, II. Ser. {\bf 33} (1981), 25--33.}

\bibitem[C2]{VincentThesis} 
\bysame,
\newblock {\em Poly\`{e}dre caract\'{e}ristique d'une singularit\'{e}},
\newblock {Thesis, Universit\'{e} de Paris-Sud, Centre d'Orsay (1987), 1--424.}


\bibitem[C3] {CosRevista}
\bysame,
\newblock {Poly\`{e}dre caract\'{e}ristique et \'{e}clatements
combinatoires},
\newblock {{\em Rev. Mat. Iberoam.} {\bf 5} (1989), 67--95.}


\bibitem[CGO]{CGO}
V.~Cossart, J.~Giraud, and U.~Orbanz,
\newblock {\em Resolution of surface singularities},
\newblock with an appendix by H.~Hironaka,
\newblock {Lecture Notes in Mathematics} {\bf 1101}, Springer Verlag, (1984).

\bibitem[CJS]{CJS}
V.~Cossart, U.~Jannsen, and S.~Saito,
\newblock {Canonical embedded and non-embedded resolution of singularities for
  excellent two-dimensional schemes},
\newblock {\em preprint} (2009), available on arXiv:0905.2191.


\bibitem[CP1]{CPcompl}
V.~Cossart and O.~Piltant,
\newblock {Characteristic polyhedra of singularities without completion},
\newblock {{\em Math. Ann.} {\bf 361} (2015), 157--167, 
	DOI 10.1007/s00208-014-1064-0.}

\bibitem[CP2]{CPmixed}
\bysame,
\newblock {Resolution of Singularities of Threefolds in Mixed Characteristics.
Case of small multiplicity}, 
\newblock{\em Rev. R. Acad. Cienc. Exactas F\'{\i}s. Nat. Ser. A Math. RACSAM} {\bf 108} (2014), no. 1, 113-151.



\bibitem[CP3]{CPmixed2}
\bysame,
\newblock {Resolution of Singularities of Arithmetical Threefolds {II}},
\newblock {\em preprint} (2014), available on {HAL:hal01089140}, 
to be published in {\em Jour. of Algebra}.


\bibitem[CSc]{CSCcompl}
V.~Cossart and B.~Schober,
\newblock {Characteristic polyhedra of singularities without completion - Part II},
\newblock {\em preprint} (2014), available on arXiv:1411.2522.

\bibitem[Cu1]{DaleBook}
S.~D.~Cutkosky,
\newblock {\em Resolution of singularities},
\newblock {Grad. Stud. in Math.} {\bf 63}, Amer. Math. Soc. (2004).

\bibitem[Cu2]{DaleDim2}
\bysame,
\newblock {A skeleton key to Abhyankar's proof of embedded resolution of characteristic p surfaces},
\newblock {\em Asian J. Math.} {\bf 15} (2012), 369--416.

\bibitem[G1]{GiraudEtude}
J.~Giraud,
\newblock {\em \'{E}tude locale des singularit\'{e}s},
\newblock {Cours de $3^{\mbox{\`{e}me}}$ cycle, 1971-1972, Publ. Math. d'Orsay {\bf 26} (1972)}.

\bibitem[G2]{GiraudMaxPos}
\bysame,
\newblock {Contact maximal en caract\'{e}ristique positive},
\newblock {\em Ann. Sci. \'Ec. Norm. Sup.} $4^{\mbox{{\`e}me}}$ s\'erie { \bf 8} (1975), 201--234.

\bibitem[HaW]{HerWagn}
H.~Hauser and D.~Wagner,
\newblock {Alternative invariants for the embedded resolution of purely inseparable surface singularities},
\newblock {\em Enseign. Math.} {\bf 60} (2014), 177--224. 


\bibitem[H1]{HiroCharPoly}
H.~Hironaka,
\newblock {Characteristic polyhedra of singularities},
\newblock {\em J. Math. Kyoto Univ.} {\bf 7} (1967), 251--293.


\bibitem[H2]{HiroBowdoin}
\bysame,
\newblock {Desingularization of Excellent Surfaces},
\newblock Adv. Sci. Sem. in Alg. Geo, Bowdoin College, Summer 1967.
\newblock Notes by Bruce Bennett, {\em appendix of \cite{CGO}} (1984).


\bibitem[KM]{IFPdim2}
H.~Kawanoue and K.~Matsuki,
\newblock {Resolution of singularities of an idealistic filtration in dimension 3 after Benito-Villamayor},
\newblock {\em Minimal models and extremal rays} (Kyoto, 2011), 115--214, Adv. Stud. Pure Math. {\bf 70}, Math. Soc. Japan, [Tokyo], 2016.

\bibitem[L]{Lipmandim2}
J.~Lipman,
\newblock {Desingularization of two-dimensional schemes},
\newblock {\em Ann. of Math.} {\bf 107} (1978), 151--207.

\bibitem[Ma]{Maclagan}
D.~Maclagan,
\newblock {Antichains of monomial ideals are finite},
\newblock {\em Proc. Amer. Math. Soc.} {\bf 129} (2001), 1609--1615.

\bibitem[Mo1]{Moh}
T.T.~Moh,
\newblock {On a stability theorem for local uniformization in characteristic p},
\newblock {\em Publ. Res. Inst. Math. Sci.} {\bf 23} (1987), 965--973.


\bibitem[Mo2]{Moh1} 
\bysame,
\newblock {On a Newton polygon approach to the uniformization of singularities of characteristic $p$},
\newblock {in} {\em Algebraic geometry and singularities (La R\'abida, 1991)}, 49--93, {\em Progr. Math.} {\bf 134}, Birkh\"auser (1996).



\bibitem[Sc]{BerndThesis}
B.~Schober,
\newblock {\em Characteristic polyhedra of idealistic exponents with history},
\newblock Dissertation, Universit\"at
Regensburg (2013). \url{http://epub.uni-regensburg.de/28877/}


\end{thebibliography}
\end{document}